\documentclass[10pt]{article}

\makeatletter
\DeclareRobustCommand{\qed}{%
  \ifmmode 
  \else \leavevmode\unskip\penalty9999 \hbox{}\nobreak\hfill
  \fi
  \quad\hbox{\qedsymbol}}
\providecommand{\upd}[1]{{\color{red}#1}} 
\providecommand{\Q}{\mathbb{Q}}
\newcommand{\openbox}{\leavevmode
  \hbox to.77778em{%
  \hfil\vrule
  \vbox to.675em{\hrule width.6em\vfil\hrule}%
  \vrule\hfil}}
\newcommand{\qedsymbol}{\openbox}
\newenvironment{proof}[1][\proofname]{\par
  \normalfont
  \topsep6\p@\@plus6\p@ \trivlist
  \item[\hskip\labelsep\itshape
    #1.]\ignorespaces
}{%
  \qed\endtrivlist
}
\newcommand{\proofname}{Proof}
\makeatother

\usepackage[utf8]{inputenc}   
\usepackage{bm}
\usepackage{color}
\usepackage{latexsym}
\usepackage{dsfont}
\usepackage{amssymb}
\usepackage{comment}
\usepackage{graphicx}
\usepackage{amsmath,amsfonts,amssymb,theorem,euscript,array,enumerate,amsfonts,mathrsfs}
\usepackage{hyperref}
\usepackage{appendix}
\usepackage[T1]{fontenc}
\usepackage{babel}
\numberwithin{equation}{section}
\usepackage{bbm}
\usepackage{subfigure}
\usepackage{color}
\usepackage{stmaryrd}
\makeatletter\input{Ustmry.fd}\makeatother
\DeclareFontShape{U}{stmry}{b}{n}{<-> ssub * stmry/m/n}{}

\usepackage[algo2e,ruled,vlined]{algorithm2e} 
 \usepackage{algorithm}
 \usepackage{algorithmicx}
\usepackage{algpseudocode}
\usepackage{ marvosym }
\usepackage{hyperref}
\usepackage{bm}
\usepackage{xcolor, soul}

\usepackage{mathtools}
\mathtoolsset{showonlyrefs}

\usepackage{comment}

\usepackage{tikz}
\usetikzlibrary{fit,matrix,chains,positioning,decorations.pathreplacing,arrows}

\usepackage{mathtools}

\usepackage{geometry}
\usepackage{algpseudocode}
\usepackage{algorithm}

\usepackage[normalem]{ulem}

\def \Sum{\displaystyle\sum}

\def \b1{\bf{1}}

\def \I{\mathbb{I}}
\def \N{\mathbb{N}}
\def \R{\mathbb{R}}

\def \E{\mathbb{E}}
\def \F{\mathbb{F}}

\def \P{\mathbb{P}}
\def \Q{\mathbb{Q}}

\def \S{\mathbb{S}}

\def \H{\mathbb{H}}

\def \bV{\boldsymbol{V}} 
\def \bnu{\boldsymbol{\nu}}

\def \d{\mathrm{d}}

\def\esssup_#1{\underset{#1}{\mathrm{ess\,sup\, }}}

\def\argmin_#1{\underset{#1}{\mathrm{argmin\, }}}
\def\argmax_#1{\underset{#1}{\mathrm{argmax\, }}}

\def\dm#1{\frac{\delta}{\delta m}}

\def \Ac{{\cal A}}
\def \Bc{{\cal B}}
\def \Cc{{\cal C}}

\def \Ec{{\cal E}}
\def \Fc{{\cal F}}

\def \Ic{{\cal I}}

\def \Lc{{\cal L}}
\def \Pc{{\cal P}}

\def \Sc{{\cal S}}

\def \Wc{{\cal W}}

\def\bX{{\boldsymbol X}}

\def\bx{{\boldsymbol x}}
\def\by{{\boldsymbol y}}

\def\bm{{\bf m}}

\def \d{\mathrm{d}}

\def\beqs{\begin{eqnarray*}}
\def\enqs{\end{eqnarray*}}
\def\beq{\begin{eqnarray}}
\def\enq{\end{eqnarray}}

\renewcommand{\thefootnote}{\alph{footnote}}

\newcommand{\red}[1]{\textcolor{red}{#1}}
\newcommand{\bl}[1]{\textcolor{blue}{#1}}

\def\red#1{{\color{red}#1}}

\newtheorem{Theorem}{Theorem}[section] 
\newtheorem{Definition}[Theorem]{Definition} 
\newtheorem{Proposition}[Theorem]{Proposition}
\newtheorem{Assumption}[Theorem]{Assumption}
\newtheorem{Lemma}[Theorem]{Lemma}
\newtheorem{Corollary}[Theorem]{Corollary}
\newtheorem{Remark}[Theorem]{Remark}

\title{
Stochastic Maximum Principle for McKean–Vlasov Control with Discrete Path Dependence
}

\author{
Yadh Hafsi$^{*}$ \qquad
Samy Mekkaoui$^{\dagger}$ \qquad
Huyên Pham$^{\ddagger}$
}

\date{}

\begin{document}

\maketitle

\begingroup
\renewcommand{\thefootnote}{\fnsymbol{footnote}}
\footnotetext[1]{Ecole Polytechnique, CMAP, \sf yadh.hafsi at polytechnique.edu; This author acknowledges support from the Chaire Risque Financiers, Société Générale, at École Polytechnique, and from the Institut Europlace de Finance (IEF).}
\footnotetext[2]{Ecole Polytechnique, CMAP, \sf samy.mekkaoui at polytechnique.edu; This author is supported by the S-G Chair "Risques Financiers", and the "Deep Finance and Statistics" Qube-RT Chair.}
\footnotetext[3]{Ecole Polytechnique, CMAP, \sf huyen.pham at polytechnique.edu; This author is supported by
the Chair ``Risques Financiers" Soci\'et\'e G\'en\'erale,  by FiME,
Laboratoire de Finance des Marchés de l'Energie, and the ``Finance and Sustainable Development'' EDF -
CACIB Chair, and the Chair AIDA-F (Artificial Intelligence for Data Augmentation in Finance) with Bank of America.}
\endgroup

\begin{abstract}
We study a class of McKean--Vlasov control problems with discrete path dependence. The coefficients and cost functional may depend on finitely many past values of the controlled state, observed at fixed deterministic times, and on their joint law. We establish well-posedness of the controlled state equation and derive necessary and sufficient optimality conditions through a stochastic Pontryagin maximum principle. The adjoint process is characterized by a backward stochastic differential equation with jumps at fixed observation times. Each jump represents the conditional sensitivity of future costs with respect to the corresponding observed state and its distribution. In the linear-quadratic case, we prove global solvability of the resulting forward-backward system by combining a continuation argument with a mean-field Riccati reduction. We finally discuss an application to time-series generation, where the terminal cost is given by a kernel-based discrepancy between the law of the sampled path and a target path distribution. 
\end{abstract}

\vspace{5mm}


\vspace{5mm}


\section{Introduction}

McKean--Vlasov stochastic control has become a fundamental framework for optimization problems involving interactions through probability distributions. In its classical formulation, the dynamics and the cost functional depend on the current state of the controlled process and on its marginal distribution. Such problems arise naturally in mean-field control, economics, engineering, and more recently in machine learning and generative modeling. Over the last decade, a rich theory has been developed for these models, including dynamic programming principles, master equations, and stochastic Pontryagin maximum principles, see among others  \cite{bensoussan2013meanfield,carmona2018probabilistic_II}.

Early versions of the stochastic maximum principle for mean-field dynamics go back to \cite{andersson2011maximum,buckdahn2011general} for linear dependence on the law of the process and for general non-linear dependence in \cite{carmona_delarue_2015,acciaio2019extended}. In these works, the coefficients of the state equation and of the cost depend on the current state and control and on its marginal law. 

However, in many situations, the relevant information is not contained in the current marginal distribution alone, but in the joint distribution of observations collected at several dates. 
Given deterministic observation dates $0<t_1<\cdots<t_N=T$,  
a typical objective functional considered in this paper takes the form 
\begin{equation}
J(\alpha)  
=
\mathbb E\Big[
\int_0^T f(t,X_t,\Lc(X_t),\alpha_t) \d t \Big]
+
g\left(
\Lc(X_{t_1},\ldots,X_{t_N})
\right),
\label{eq:intro-functional}
\end{equation}
where $X$ is the controlled state process, $\Lc(X_{t_1},\ldots,X_{t_N})$ its joint law at the observation dates,  and $\alpha$ is an admissible control.

Unlike classical McKean--Vlasov control problems, the objective functional \eqref{eq:intro-functional} depends on the joint law of observations collected along the trajectory. Such dependence naturally arises whenever temporal interactions play a significant role. Indeed, properties such as autocorrelation, lead--lag effects or persistence phenomena  are encoded in the joint distribution of observations rather than in their individual marginals. Consequently, controlling only the marginal laws \(\Lc(X_t)\) is generally insufficient for reproducing the temporal structure of the underlying process.

The purpose of this paper is to develop a stochastic maximum principle for McKean--Vlasov control problems with such discrete path dependence. More precisely, we consider controlled diffusions whose coefficients and objective functional may depend on the current state, its current distribution, the vector of previously observed values $(X_{t_1},\ldots,X_{t_k})$ for $t_k\leq t<t_{k+1}$, and the joint distribution of these observations. The dependence is assumed to be non-anticipative  (see Definition \ref{def : non_anticipative}), so that only observations already revealed at time $t$ may influence the dynamics.

The framework considered here occupies an intermediate position between classical and fully path-dependent McKean--Vlasov control. In classical McKean--Vlasov control, the coefficients depend only on $(X_t,\Lc(X_t))$. 
At the other end of the spectrum, recent works have investigated control problems depending on the entire past trajectory $X_{[0,t]}$ and on its law. Such formulations naturally lead to infinite-dimensional state spaces and 
require sophisticated analytical tools including functional It\^o calculus and path-dependent partial differential equations; see \cite{Dupire2009}, \cite{ContFournie2013},  and  \cite{EkrenKellerTouziZhang2014}.
By contrast, the present paper focuses on dependence through a finite collection of observation dates. This discrete path dependence is rich enough to capture temporal interactions while preserving a finite-dimensional structure.

Recently, path-dependent McKean--Vlasov control problems have been studied in \cite{huyencosso_McKV2023} through dynamic programming methods and in \cite{buckdahn2025path} through a stochastic maximum principle.  The finite-observation structure considered here leads to a different adjoint characterization. A precise comparison with a formulation by a family of $N$ coupled continuous adjoint processes is given in Remark~\ref{rem : N_adjoints}.

\paragraph*{Our main contributions.} 
We first establish existence and uniqueness of the controlled McKean--Vlasov dynamics under standard Lipschitz and linear-growth assumptions. The main novelty of the present work is the appearance of  jumps in the adjoint equation. In the classical McKean--Vlasov setting, the adjoint process is continuous and satisfies a backward stochastic differential equation. In contrast, the adjoint process $(Y,Z)$ associated with \eqref{eq:intro-functional} is c\`adl\`ag and exhibits jumps at the observation dates $t_1,\ldots,t_{N-1}$. These jumps reflect the fact that additional path information becomes available whenever a new observation date is reached. More precisely, the jump size at each observation date $(t_k)_{1 \leq k \leq N-1}$ is given by
\begin{align}\label{eq : jump_terms0}
  -\Delta Y_{t_k}
  &=\mathbb{E}\Big[\int_{t_k}^{T}\Big(\partial_{x_k}H(t,\boldsymbol{\Theta}_t)
  +\tilde{\mathbb{E}}\big[\partial_{\tilde x_k}\tfrac{\delta}{\delta\boldsymbol{m}}
  H(t,\tilde{\boldsymbol{\Theta}}_t)(\boldsymbol{X})\big]\Big)\mathrm{d}t \\
  &\hspace{1.2 cm} +\partial_{x_k}g(\boldsymbol{X},\Lc(\bX)))
   + \tilde{\mathbb{E}}\big[\partial_{\tilde x_k}\tfrac{\delta}{\delta\boldsymbol{m}}
g(\tilde{\bX},\Lc(\bX))(\boldsymbol{X})\big]
  \,\Big|\,\mathcal{F}_{t_k}\Big],
\end{align}
where $H$ is the Hamiltonian function and 
where we set 
$\boldsymbol{\Theta}_t=\big((X_{t_1},\ldots, X_{t_l}),\Lc(X_{t_1},\ldots, X_{t_l}),X_t,\Lc({X_t}),Y_t,Z_t,\alpha_t \big)
$ for $t_l \leq t < t_{l+1}$ for any $1 \leq k \leq l \leq N-1$
, $\tilde{\boldsymbol{\Theta}}$ an independent copy and $\bX=(X_{t_1},\ldots,X_{t_N})$.


We derive the  Pontryagin stochastic maximum principle providing necessary conditions for optimality. We characterize the adjoint process as the solution of a backward stochastic differential equation with  jump at deterministic times given by \eqref{eq : jump_terms0}.  Under suitable convexity assumptions on the Hamiltonian and the terminal cost, we further obtain sufficient conditions for optimality.
Our second main result  concerns solvability of the resulting forward--backward system. In the linear-quadratic case, we prove global solvability by a continuation method, in the spirit of
\cite{peng1999fully,carmona_delarue_2015}. We also show that the system can be reduced to a mean-field Riccati system, related to the Riccati methods used in linear-quadratic McKean--Vlasov control \cite{yong2013linear,basei2019weak,sun2020stochastic}. This gives a novel tractable description of the optimal feedback control in the linear-quadratic setting.
Indeed, the Riccati equations together with the jump relations characterize the affine structure of the adjoint process, while the stability of the conditional expectation of the state process allows the coefficients arising in the jump terms to be identified. These ingredients lead to a recursive backward characterization of the system, in which all the coefficients characterizing the adjoint process can be successively identified.

One motivation for the present framework comes from recent stochastic approaches to generative modeling, see \cite{ma2025schr,boustany2026learning,boustany2026deep}. In applications involving time-series data, the target object is naturally a probability measure on path space, or more precisely the joint law of observations collected at several dates. Matching only marginal distributions generally fails to reproduce the temporal dependence structure of the data. The control framework developed in this paper provides a probabilistic formulation for such problems and yields an optimality system that explicitly incorporates path-distribution dependence. To establish the well-posedness of this optimality system, we consider the Maximum Mean Discrepancy (MMD) distance on the space of probability measures (see \cite{sriperumbudur2010hilbert}), and, under suitable regularity assumptions on the kernel $K$, prove local-in-time existence and uniqueness of the resulting forward-backward stochastic differential equation. 

\paragraph*{Outline of the paper}
The remainder of the paper is organized as follows. 
Section~\ref{sec : problem} introduces the control problem and establishes well-posedness of the state dynamics. Section~\ref{sec : hamiltonian_adjoint_equations} derives the variational equations and the adjoint backward stochastic differential equation. Section~\ref{sec : pontryagin} proves the stochastic maximum principle and presents sufficient conditions for optimality. Section \ref{sec : solvability} studies solvability of the
forward--backward system, first in the linear-quadratic case and then in
the application to time-series generation. Section~\ref{sec : conclusion} concludes, and additional proofs are gathered in the appendix.

\section{Problem formulation}
\label{sec : problem}

\subsection{Notations}

Throughout the paper, we fix a finite time horizon $T>0$ and a complete
probability space $(\Omega,\Fc,\P)$ supporting an $n$-dimensional Brownian
motion $W=(W_t)_{0\le t\le T}$. Let $\F = (\Fc_t)_{0 \leq t \leq T}$ be the natural filtration generated by $W$ augmented with the $\P$-null sets and with the $\sigma$-algebra generated by a rich enough random variable independent of $W$ used for randomization of the initial condition of the state process.
The expectation under $\P$ is denoted by $\E$. For a given random variable $\chi$ over $(\Omega,\Fc,\P)$, we denote by $\P_{\chi}$ its law under $\P$.

We denote by $\Cc([0,T];\R^d)$ the space of continuous functions from
$[0,T]$ into $\R^d$,  and 
\begin{align*}
\begin{cases}
\S^p([0,T];\R^d)
&:= \bigg\lbrace Y=(Y_t)_{0\leq t\leq T} :
Y \text{ is $\F$-prog. measurable, càdlàg and }
\E\bigg[\underset{0\leq t\leq T}{\mathrm{sup}} |Y_t|^p\bigg]<\infty
\bigg\rbrace,
\quad p\geq 1, \\[0.4em]
\S^\infty([0,T];\R^d)
&:= \bigg\lbrace Y=(Y_t)_{0\leq t\leq T} :
Y \text{ is $\F$-prog.  measurable, càdlàg and }
\bigg\|
\underset{0\leq t\leq T}{\mathrm{sup}} |Y_t|
\bigg\|_{L^\infty}<\infty
\bigg\rbrace, \\[0.4em]
\H^2([0,T];\R^{d\times n})
&:= \bigg\lbrace Z=(Z_t)_{0\leq t\leq T} :
Z \text{ is $\F$-prog.  measurable and }
\E\bigg[\int_0^T |Z_t|^2\,\d t\bigg]<\infty
\bigg\rbrace, \\[0.4em]
\mathrm{BMO}([0,T];\R^{d\times n})
&:= \bigg\lbrace Z=(Z_t)_{0\leq t\leq T} :
Z \text{ is $\F$-prog.  measurable and }
\underset{\tau}{\mathrm{sup}}\,
\bigg\|
\E\bigg[
\left.
\int_\tau^T |Z_t|^2\,\d t
\,\right|\,\F_\tau
\bigg]
\bigg\|_{L^\infty}<\infty
\bigg\rbrace,
\end{cases}
\end{align*}
where the supremum in the definition of $\mathrm{BMO}$ is taken over all $\F$-stopping times $\tau$ valued in $[0,T]$ and where $\lVert \chi \rVert_{L^{\infty}} := \text{inf } \big \lbrace C  \geq 0 : |\chi| \leq C, \quad \P-\text{a.s} \big \rbrace$ for any random variable $\chi$ over $(\Omega,\Fc,\P)$.

For a Polish space $E$ endowed with a metric $d_E$, we denote by
$\Pc_2(E)$ the set of Borel probability measures $\mu$ on $E$ with finite
second moment, i.e.,  $\int_E d_E(x,x_0)^2 \mu(\d x)$ $<$ $\infty$ for every $x_0\in E$. 
For $\mu,\nu\in\Pc_2(E)$, we
write $\Pi(\mu,\nu)$ for the set of probability measures on $E\times E$
with first marginal $\mu$ and second marginal $\nu$. The associated
$2$-Wasserstein distance is
$$
\Wc_2(\mu,\nu)
:=
\left(
\inf_{\pi\in\Pi(\mu,\nu)}
\int_{E\times E} d_E(x,y)^2\,\pi(\d x,\d y)
\right)^{1/2}.
$$
When $E=\R^d$, the metric $d_E$ is the Euclidean distance. When $E = \Cc([0,T];\R^d)$, we also denote by
\begin{align*}
    \Wc_{2,T}(\mu,\nu) := \bigg( \underset{\pi \in \Pi(\mu,\nu)}{\text{ inf }}  \int \underset{0 \leq t \leq T}{\text{ sup }} | \omega_t - \tilde{\omega}_t|^2 \pi(\d \omega, \d \tilde{\omega}) \bigg)^{1/2}.
\end{align*}

\noindent  We fix a finite time grid $\Pi := \big \lbrace 0=t_0 < t_1 < t_2 < \ldots < t_N = T \big \rbrace $ over the interval $[0,T]$. We introduce the  map $\eta$ defined as $[0,T] \owns t \mapsto \eta(t) := \max \big \lbrace i \in \llbracket 1,N \rrbracket :  t_i \leq t \big \rbrace$ with the convention $\eta(t) = 0$ for $t < t_1$, so that $t_{\eta(t)}$ is the largest grid point not exceeding $t$. We define the projection maps as follows: For any $i,j\in \llbracket 1, N \rrbracket$, we define the projection maps $\text{pr}_{1:i}$ and $\text{pr}_i$ as
\begin{align*}
\begin{cases}
     (\R^d)^N \owns \boldsymbol{x} = (x_1,\ldots,x_N)  &\mapsto \text{pr}_{1:i}(\bx)=\big( x_1, \ldots,x_i \big) 
    = : \bx_{i} \in (\R^{d})^i , \\
    (\R^d)^N \ni \bx =(x_1,\ldots, x_N) &\mapsto \text{pr}_i(\bx)=x_i \in \R^d.
\end{cases}
\end{align*}
Given any measure $\boldsymbol{\mu} \in \Pc_2((\R^d)^N)$, we denote by
   $ \boldsymbol{\mu}_{i} :=\text{pr}_{1:i} \sharp \boldsymbol{\mu},$
where $\sharp$ denotes the pushforward measure. By convention, we set $\bx_{\eta(t)}=0$ and $\boldsymbol{\mu}_{\eta(t)}= \delta_0$ if $t < t_1$.
\begin{Definition}\label{def : non_anticipative}
    Let $h$ be a Borel measurable map defined over $[0,T] \times (\R^d)^N \times \Pc_2\big((\R^d)^N\big) $ into $\R^{p \times q}$.
     We say that $h$ is a $\Pi$-non-anticipative  map if it is in the form
    \begin{align}
    \begin{cases}
        h(t,\boldsymbol{x},\boldsymbol{\mu}) &= h_0(t), \hspace{1.5 cm} t_0 \leq t < t_1, \\
        h(t,\bx, \boldsymbol{\mu}) &= h_{i}(t,\bx_{i}, \boldsymbol{\mu}_{i}), \quad t_{i} \leq t < t_{i+1} \quad  \text{ for $i \in \llbracket 1, N-1 \rrbracket$},  \\
        h(T,\bx,\boldsymbol{\mu}) &= h_N(\boldsymbol{x},\boldsymbol{\mu}),
    \end{cases}
\end{align}
for some measurable maps  $h_0 : [t_0,t_1) \to \R^{p \times q}$ and $h_i : [t_i, t_{i+1}) \times (\R^d)^i \times \Pc_2((\R^d)^i) \to \R^{p \times q} $, $i \in \llbracket 1, N-1 \rrbracket$ and $h_N : (\R^d)^N \times \Pc_2((\R^d)^N) \to \R^{p \times q}$
\end{Definition}
\begin{Remark}
    Whenever a non-anticipative map $h$ is differentiable with respect to its state argument, with gradient denoted by $\partial_{\bx} h(t,\bx,\boldsymbol{\mu})$ valued in $(\R^{(p \times q) \times d})^N$, we see that the   partial derivative with respect to the $j$-th coordinate of $\bx$ vanishes whenever $j > i$ , i.e. $\partial_{x_j} h(t,\bx, \boldsymbol{\mu})=0$, for $t \in [t_i, t_{i+1})$. We then set 
    \begin{align} \label{eq : def_derivative}
\partial_{\bx} h(t,\bx,\boldsymbol{\mu}) \cdot  \bx := \sum_{i=1}^{\eta(t)} \partial_{x_i} h(t,\bx,\boldsymbol{\mu}) x_i \in \R^{p \times q},
\end{align}
and  whenever $p=q=1$, $\partial_{x_i} h(t,\bx,\boldsymbol{\mu}) x_i= \partial_{x_i} h(t,\bx,\boldsymbol{\mu}) \cdot x_i$ is understood as the standard scalar product over $\R^d$.
\end{Remark}

\subsection{The control problem}

This section sets up the control problem studied throughout the paper. We introduce the controlled state equation, formulate the standing regularity assumptions on its coefficients, prove well-posedness, and define the cost functional to be minimized. Let $A \subset \R^{m}$ be a non-empty  closed and convex Borel subset for some $m \in \N^{\star}$. An admissible control is an $A$-valued $\F$-progressively measurable process $\alpha=(\alpha_t)_{0 \leq t \leq T}$ satisfying $\E \big[ \int_{0}^{T} | \alpha_t|^2 \d t \big] < \infty$. We denote by $\Ac$ the set of all admissible controls. The convexity of $A$ ensures that for any $\alpha,\beta \in \Ac$ and any $\epsilon \in [0,1]$, the convex combination $\alpha + \epsilon(\beta - \alpha)$ also lies in $\Ac$. We fix drift and  diffusion coefficients  $b,\sigma : [0,T] \times  (\R^d)^N \times \Pc_2((\R^d)^N) \times \R^d \times \Pc_2(\R^d) \times A \to \R^d, \R^{d \times n} $ on which we make the following assumptions. 

\begin{Assumption}\label{assumption : regularity_b_and_sigma} 
\begin{enumerate}
    \item [(1)] The functions $b$ and $\sigma$ are $\Pi$-non-anticipative.
    \item [(2)] There exists a constant $L \geq 0$ such that for $\phi$ $=$ $b,\sigma$, 
    \begin{align*}
        | \phi(t, \bx ,\boldsymbol{\mu}&, x,\mu, a) - \phi(t,\bx',\boldsymbol{\mu}', x',\mu',a) | \leq  L \big( |\bx_{\eta(t)} - \bx'_{\eta(t)}| + \Wc_2(\boldsymbol{\mu}_{\eta(t)}, \boldsymbol{\mu}'_{\eta(t)})   + \;  |x-x'| + \Wc_2(\mu,\mu') \big), 
    \end{align*}
     for $t \in [0,T], \bx,\bx' \in (\R^d)^N$, $\boldsymbol{\mu}, \boldsymbol{\mu}' \in \Pc_2((\R^d)^N)$, $x,x' \in \R^d$, $\mu,\mu' \in \Pc_2(\R^d)$ and $a \in A$.
    \item [(3)] There exists a constant $M \geq 0$ such that
    \begin{align*}
        | b(t,\boldsymbol{0}, \boldsymbol{\delta}_{0}, 0, \delta_0,a)| 
        +  | \sigma(t,\boldsymbol{0}, \boldsymbol{\delta}_0, 0, \delta_0,a) |&\leq  M \big( 1+ |a| \big), 
    \end{align*}
    for any $t \in [0,T]$ and $a \in A$.
\end{enumerate}
\end{Assumption}

\vspace{1mm}

Next, given $\alpha \in \Ac$, we consider the following state process $X^{\alpha}=X:=(X_t)_{0 \leq t \leq T}$ governed by the McKean-Vlasov SDE 
\begin{align}\label{eq : state_process_dynamics}
\begin{cases}
    \d X_t  = b(t, \bX, \P_{ \bX}, X_t, \P_{X_t}, \alpha_t) \d t + \sigma(t, \bX, \P_{ \bX}, X_t, \P_{X_t}, \alpha_t) \d W_t,\\
    X_0  = \xi,
\end{cases}
\end{align}
where $t \in [0,T]$ and $\bX=(X_{t_1},\ldots,X_{t_N})$. An admissible initial condition is given by an $\Fc_0$-measurable $\R^d$-valued random variable $\xi$ satisfying $\E \big[ |\xi|^2 \big] < \infty$, independent of the Brownian motion $W$.

\begin{Proposition}\label{thm : existence_unicity_X_SDE}
    Fix an admissible initial condition $\xi$ and an admissible control $\alpha \in \Ac$. Then, under Assumption \ref{assumption : regularity_b_and_sigma}, there exists a unique strong solution in $\S^2([0,T];\R^d)$ to the SDE \eqref{eq : state_process_dynamics}. Moreover, the following estimate holds.
    \normalfont
    \begin{align*}
        \E \Big[ ~\underset{0 \leq t \leq T}{\text{sup }} |X_t|^2 \Big] \leq C \bigg( 1+ \E \big[ |\xi|^2 \big] + \E \Big[ \int_{0}^{T}  |\alpha_t|^2 \d t \Big] \bigg),
    \end{align*}
    where $C$ is a constant depending on the Lipschitz and linear-growth coefficients of $b$ and $\sigma$, on $T$, and on $N$.
\end{Proposition}
\begin{proof}
    See Appendix \ref{app:proofs}.
\end{proof}

We next introduce two reward functions 
\begin{align*}
    f : [0,T] \times (\R^d)^{N} \times \Pc_2((\R^d)^N) \times \R^d \times \Pc_2(\R^d) \times  A \to \R, \\
    g: (\R^d)^N \times \Pc_2( (\R^d)^{N}) \to \R, \notag 
\end{align*}
on which we make the following assumptions.

\begin{Assumption}\label{assumption : regularity_f_and_g} 
\begin{enumerate}
  \item [(1)] The function $f$ is $\Pi$-non-anticipative and $g$ is Borel measurable.
  \item [(2)] There exists a constant $M \geq 0$ such that
  \begin{align*}
  \begin{cases}
       -M \big( 1+ |\bx_{\eta(t)}|^2 + \Wc_2^2(\boldsymbol{\mu}_{\eta(t)},\boldsymbol{\delta}_0) + |x|^2 + \Wc_2^2(\mu,\delta_0) \big)\leq f(t,\bx,\boldsymbol{\mu}, x,\mu,a),\\f(t,\bx,\boldsymbol{\mu}, x,\mu,a) \leq  M \big( 1+ |\bx_{\eta(t)}|^2 + \Wc_2^2(\boldsymbol{\mu}_{\eta(t)}, \boldsymbol{\delta}_0) + |x|^2 + \Wc_2^2(\mu,\delta_0) + |a|^2 \big), \\
       |g(\bx,\boldsymbol{\mu})| \leq M \big( 1+ |\bx|^2 + \Wc_2^2(\boldsymbol{\mu},\boldsymbol{\delta}_0) \big),
  \end{cases}
  \end{align*}
   for any $t \in [0,T]$, $\bx \in (\R^d)^N$, $\boldsymbol{\mu} \in \Pc_2((\R^d)^N)$, $x \in \R^d$, $\mu \in \Pc_2(\R^d)$ and $a \in A$.
\end{enumerate}
    
\end{Assumption}

We define the cost functional $J : \Ac \to \R$ as
\begin{align}\label{eq : cost_functional}
    J(\alpha) := \E \Big[ \int_{0}^{T} f(t, \bX, \P_{\bX},X_t, \P_{X_t}, \alpha_t) \d t  + g(\bX,\P_{\bX})  \Big]. 
\end{align}
Under Assumptions \ref{assumption : regularity_b_and_sigma} and \ref{assumption : regularity_f_and_g}, we get from Proposition \ref{thm : existence_unicity_X_SDE} that $J(\alpha)$ is well defined for any $\alpha \in \Ac$.
Our aim is to study the optimal control problem \eqref{eq : state_process_dynamics}-\eqref{eq : cost_functional} consisting of minimizing over $\Ac$ the cost functional $J$, that is, computing
\begin{align*}
    V_0 := \underset{\alpha \in \Ac}{\text{ inf }} J(\alpha),
\end{align*}
and finding an optimal control $\alpha^{\star} \in \Ac$, that is  $J(\alpha^{\star})= V_0$.

\section{Hamiltonian and adjoint equations}\label{sec : hamiltonian_adjoint_equations}

Before stating the Hamiltonian and the adjoint system, we briefly recall the notion of linear functional derivative on the Wasserstein space that will be used throughout this section, and introduce the natural notion of convexity for the $\Pi$-non-anticipative maps appearing in our setting.
\subsection{Linear functional derivatives}
\begin{Definition}\label{def : linear_functional_derivative}
\noindent A measurable map $U : \Pc_2((\R^d)^\ell) \to \R$ admits a linear functional derivative if there exists a measurable map
\begin{align}\label{eq : derivative_notations}
\Pc_2((\R^d)^{\ell}) \times (\R^d)^{\ell} \owns (\boldsymbol{\mu},\tilde{\bx}) \mapsto \frac{\delta U}{\delta \bm}(\boldsymbol{\mu})(\tilde{\bx}) \in \R,
\end{align}
with at most quadratic growth in $\tilde{\bx}$ locally uniformly in $\boldsymbol{\mu}$, such that for every $\boldsymbol{\mu},\boldsymbol{\mu}' \in \Pc_2((\R^d)^{\ell})$
\begin{align*}
    U(\boldsymbol{\mu}') - U(\boldsymbol{\mu}) = \int_{0}^{1} \int_{(\R^d)^{\ell}} \frac{\delta U}{\delta \bm}\big(\lambda \boldsymbol{\mu}' + (1-\lambda) \boldsymbol{\mu}\big)(\tilde{\bx}) \big(\boldsymbol{\mu}' - \boldsymbol{\mu}\big)(\d \tilde{\bx}) \d \lambda.
\end{align*}
Whenever $\tilde{\bx} \mapsto \frac{\delta U}{\delta \bm}(\boldsymbol{\mu})(\tilde{\bx})$ is differentiable on $(\R^d)^{\ell}$, we denote by $\partial_{\tilde{\bx}}\frac{\delta U}{\delta \bm}(\boldsymbol{\mu})(\tilde{\bx})$ its gradient. The linear functional derivative $\frac{\delta}{\delta m}$ with respect to a measure $\mu \in \Pc_2(\R^d)$ (and associated partial gradient $\partial_{\tilde{x}}\frac{\delta}{\delta m}$) is defined analogously. For a map $U$ depending on several variables, the notation $\frac{\delta U}{\delta \bm}$ (resp.\ $\frac{\delta U}{\delta m}$) refers to the linear functional derivative taken with respect to the path-measure argument $\boldsymbol{\mu} \in \Pc_2((\R^d)^{\ell})$ (resp.\ the state-measure argument $\mu \in \Pc_2(\R^d)$), the other arguments being kept fixed. Whenever $U : (\R^d)^\ell \times \Pc_2((\R^d)^\ell)$ is jointly differentiable, we shall denote by $\partial_\bx U(\bx,\boldsymbol{\mu})$ the partial derivative with respect to its $\bx$ argument while $\partial_{\tilde{\bx}} \frac{\delta}{\delta \bm} U(\bx,\boldsymbol{\mu})(\tilde{\bx})$ refers to the derivative with respect to $\tilde{\bx}$.
\end{Definition}


\subsection{Convexity for $\Pi$-non-anticipative differentiable maps}

\begin{Definition}\label{def: convexity_maps}

\noindent A $\Pi$-non-anticipative differentiable map $H : [0,T] \times (\R^d)^{N} \times \Pc_2( (\R^d)^{N}) \times \R^d \times \Pc_2(\R^d) \times A \to \R$ is convex if the following holds
\begin{align}\label{eq : convexity_def_Hamiltonian}
      H(t,\bx'&,\boldsymbol{\mu}',x',\mu',a') - H(t,\bx,\boldsymbol{\mu},x,\mu,a)  \\
    \geq & \; \;  \partial_{\bx} H(t,\bx,\boldsymbol{\mu}, x,\mu,a)  \cdot (\bx' - \bx)  
    + \;  \E\Big[ \partial_{\tilde{\bx}} \frac{\delta}{\delta \bm} H(t,\bx,\boldsymbol{\mu},x,\mu, a)(\bX) \cdot  (\bX' - \bX)  \Big]  \\
    &\qquad +\; \partial_x H(t,\bx,\boldsymbol{\mu},x,\mu,a) \cdot(x'-x) +  \E\big[\partial_{\tilde{x}} \frac{\delta}{\delta m} H(t,\bx,\boldsymbol{\mu},x,\mu,a)(X) \cdot (X'-X) \big]  \\
    &\qquad  + \; \partial_a H(t,\bx,\boldsymbol{\mu},x,\mu,a) \cdot (a'-a),  \end{align}
for any $\bx=(x_1,\ldots,x_N),\bx' = (x'_1,\ldots,x'_N) \in (\R^d)^{N}$, $\boldsymbol{\mu},\boldsymbol{\mu}' \in \Pc_2( (\R^d)^{N})$, $x,x' \in \R^d$, $\mu,\mu' \in \Pc_2(\R^d)$, $a,a' \in A$ and $t \in [0,T]$, and for any coupling,  i.e., for any pair of random vectors $(\bX, \bX') = ((X_1,\ldots,X_{N}),(X'_1,\ldots,X'_{N}))$ and any pair of random variables $(X,X')$ defined on a common probability space $(\Omega,\Fc,\P)$ with marginal laws $\bX \sim \boldsymbol{\mu}$, $\bX' \sim \boldsymbol{\mu}'$, $X \sim \mu$ and $X' \sim \mu'$ and where we recall the notation introduced in \eqref{eq : def_derivative}.
\end{Definition}

\begin{Remark}
    The notion of convexity in \eqref{eq : convexity_def_Hamiltonian} is the natural extension of the convexity assumption in \cite{carmona_delarue_2015}. Indeed, the additional terms are characterized by the path arguments $\bx$ and $\boldsymbol{\mu}$ and appear as contributions over the time grid.
\end{Remark}

\subsection{Hamiltonian and dual equations}

The Hamiltonian of the stochastic optimization problem is defined as the real-valued  measurable map $H$ given by
\begin{align}\label{eq : hamiltonian_map}
    H(t,\bx,\boldsymbol{\mu},x,\mu,y,z,a)  & := b(t,\bx,\boldsymbol{\mu},x,\mu,a) \cdot y + \sigma(t,\bx,\boldsymbol{\mu},x,\mu,a) \cdot z  + \;  f(t,\bx,\boldsymbol{\mu},x,\mu,a),  \end{align}
for any $(t,\bx,\boldsymbol{\mu},x,\mu,y,z,a) \in [0,T] \times (\R^d)^N \times \Pc_2((\R^d)^N) \times \R^d \times \Pc_2(\R^d) \times \R^d \times \R^{d \times n} \times A$, where the dot notation stands here for the inner product in a Euclidean space.
Consistently with the $\Pi$-non-anticipative convention introduced for the coefficients $b,\sigma, f$ in Section \ref{sec : problem}, we further notice that $H$ is also $\Pi$-non-anticipative.

\begin{Definition}\label{def : adjoint_processes}Suppose that $b,\sigma,f$ and $g$ are jointly differentiable with respect to $\bx,\boldsymbol{\mu},x$ and $\mu$.  Fix an admissible initial condition $\xi$ and  admissible control $\alpha \in \Ac$, and denote by $X=(X_t)_{0 \leq t \leq T}$ the solution to \eqref{eq : state_process_dynamics}. Suppose also  that 
{\footnotesize
\begin{align}
\begin{cases}
    &\E \bigg[ \int_{0}^{T} \Big[ | \partial_x f( t,\bX ,\P_{\bX} ,X_t,\P_{X_t},\alpha_t) |^2 + \tilde{\E} \Big[ | \partial_{\tilde{x}} \frac{\delta}{\delta m} {f}(t,\bX, \P_{\bX},X_t, \P_{X_t} , \alpha_t)(\tilde{X}_{t} )|^2 \Big] \d t  \\
    & \; +  \int_{0}^{T} \Sum_{k=1}^{N} \Big( |\partial_{x_k} f(t, \bX ,  \P_{\bX},X_t,\P_{X_t}, \alpha_t) |^2+ \tilde{\E} \big[  |\partial_{\tilde{x}_k} \frac{\delta}{\delta \bm} f(t,\bX , \P_{\bX},X_t,\P_{X_t} , \alpha_t)(\tilde{\bX}) |^2 \big] \Big)  \Big] \d t \\ 
    & \;  + \; \E \bigg[ \Sum_{k=1}^{N} \Big( | \partial_{x_k} g(\bX,\P_{\bX})|^2  + \tilde{\E} \big[ |\partial_{\tilde{x}_k} \frac{\delta}{\delta \bm} g(\bX, \P_{\bX})(\tilde{\bX}) |^2 \big] \Big) \bigg] < \infty,
\end{cases}
\end{align}}
where $(\tilde{\bX}, \tilde{X}, \tilde{\alpha}, \tilde{Y}, \tilde{Z})$ is an independent copy
of $(\bX,X,\alpha,Y,Z)$ defined on another probability space $(\tilde{\Omega},\tilde{\Fc},\tilde{\P})$ with expectation operator $\tilde{\E}$.

We call adjoint processes of $X$ any pair of processes $(Y,Z) \in \S^2([0,T];\R^d) \times \H^2([0,T]; \R^{d \times n})$ solution to the following adjoint equations
\begin{equation}
\resizebox{\textwidth}{!}{$
\begin{aligned}\label{defYadjoint} 
\begin{cases}
    \d Y_t & = \; - \bigg[ \partial_{x} {H}(t, \bX_{} , \P_{\bX_{}},X_t,\P_{X_t},Y_t,Z_t,\alpha_t) + \tilde{\E} \Big[ \partial_{\tilde{x}} \frac{\delta}{\delta m} {H}(t,\tilde{\bX}_{}, \P_{\bX_{}},\tilde{X}_t, \P_{X_t},\tilde{Y}_t, \tilde{Z}_t, \tilde{\alpha}_t)(X_t) \Big]\bigg] \d t \\
    & \qquad + \;  Z_t \d W_t,  \qquad \text{for } \; t_{k-1} \leq t < t_k,\ k \in \llbracket 1, N-1 \rrbracket,\ \text{and}\ t_{N-1} \leq t \leq t_N,  \\
    \Delta Y_{t_k} &:= \; Y_{t_k} - Y_{t_k^{-}} \; = \;  - \E \Bigg[ \int_{t_k}^{T}  \bigg(\partial_{x_k} {H}(t,\bX_{}, \P_{\bX_{}}, X_t,\P_{X_t},Y_t,Z_t,\alpha_t) + \;  \tilde{\E} \big[ \partial_{\tilde{x}_k} \frac{\delta}{\delta \bm} {H}(t,\tilde{\bX}_{}, \P_{\bX_{}},\tilde{X}_t,\P_{X_t}, \tilde{Y}_t,\tilde{Z}_t, \tilde{\alpha}_t)(\bX_{}) \big]  \bigg) \d t   \\
    & \qquad\qquad\qquad\qquad\qquad\qquad + \; 
    \Big( \partial_{x_k} g(\bX, \P_{\bX}) + \tilde{\E} \big[ \partial_{\tilde{x}_k} \frac{\delta}{\delta \bm} g(\tilde{\bX}, \P_{\bX})(\bX)  \big] \Big)
    \Bigg| \Fc_{t_k} \Bigg], \quad k\in \llbracket 1, N-1 \rrbracket,  \\
    Y_{t_N} &= \; \partial_{x_N} g(\bX,\P_{\bX}) + \tilde{\E} \Big[ \partial_{\tilde{x}_N} \frac{\delta}{\delta \bm}g(\tilde{\bX}, \P_{\bX})(\bX) \Big]. 
\end{cases}
\end{aligned}
$}
\end{equation}
\end{Definition}
\begin{Remark}\label{rem : N_adjoints}
It is worth mentioning that the adjoint process in Definition \ref{def : adjoint_processes} can be written equivalently in integral form as  
\begin{align*}
    Y_t =  Y_{T} + \int_{t}^{T} F_s \d s - \int_{t}^{T} Z_s \d W_s + \sum_{k \,:\, t < t_k \leq t_{N-1}} \xi_k, \quad t \in [0,T], \quad \P-\text{a.s},
\end{align*}
where we denoted for simplicity
\begin{align}\label{eq : system_FBSDE}
\begin{cases}
    F_s &:= \Big[ \partial_{x} {H}(s, \bX_{}, \P_{\bX_{}}, X_s, \P_{X_s}, Y_s, Z_s, \alpha_s) + \tilde{\E} \big[\partial_{\tilde{x}} \frac{\delta}{\delta m} {H}(s, \tilde{\bX}_{}, \P_{\bX_{}}, \tilde{X}_s, \P_{X_s} , \tilde{Y}_s , \tilde{Z}_s, \tilde{\alpha}_s)(X_s) \big] \Big], 
    \\
    \xi_{k} &:= \E \Bigg[ \int_{t_k}^{T}  \Big(\partial_{x_k} {H}(t,\bX_{}, \P_{\bX_{}}, X_t,\P_{X_t},Y_t,Z_t,\alpha_t) + \;  \tilde{\E} \big[ \partial_{\tilde{x}_k} \frac{\delta}{\delta \bm} {H}(t,\tilde{\bX}_{}, \P_{\bX_{}},\tilde{X}_t,\P_{X_t}, \tilde{Y}_t,\tilde{Z}_t, \tilde{\alpha}_t)(\bX_{}) \big]  \Big) \d t   \\
    & \quad \quad \quad \quad \quad \quad + \; 
    \Big( \partial_{x_k} g(\bX, \P_{\bX}) + \tilde{\E} \big[ \partial_{\tilde{x}_k} \frac{\delta}{\delta \bm} g(\tilde{\bX}, \P_{\bX})(\bX)  \big] \Big)
    \Bigg| \Fc_{t_k} \Bigg], \quad \text{\upshape for  $k \in \llbracket 1, N-1 \rrbracket,$} \\
    Y_T &= \partial_{x_N} g(\bX, \P_{\bX}) + \tilde{\E} \Big[ \partial_{\tilde{x}_N} \frac{\delta}{\delta \bm} g(\tilde{\bX}, \P_{\bX})(\bX) \Big].
\end{cases}
\end{align}
Moreover, the global BSDE $(Y,Z)$ of Equation \eqref{defYadjoint} can alternatively be formulated into a coupled system of $N $ forward backward. Starting from the terminal backward $(Y^{(N)},Z^{(N)})$ on $[t_{N-1}, t_N]$ given by
\begin{align}
\begin{cases}
    \d Y_t^{(N)} &=  -F_t^{(N)} \d t+ Z_t^{(N)} \d W_t, \quad t_{N-1} \leq t \leq t_N, \\
    Y_T^{(N)} &= \partial_{x_N} g(\bX, \mathbb{P}_{\bX}) + \tilde{\E} \big[ \partial_{\tilde{x}_N} \frac{\delta}{\delta \bm} g(\tilde{\bX}, \P_{\bX})(\bX) \big],
\end{cases}
\end{align}
we then proceed recursively  by backward induction  in time by solving $(Y^{(k)},Z^{(k)})$ on $[t_{k-1}, t_k]$, for $k \in \llbracket 1, N-1 \rrbracket$, we have 
\begin{align}
\begin{cases}
    \d Y_t{^{(k)}} &=  -F_t^{(k)} \d t + Z_t^{(k)} \d W_t, \quad t_{k-1} \leq t \leq t_{k}, \\
    Y_{t_{k}}^{(k)} &= Y_{t_{k}}^{(k+1)} + \xi_{k}^{(k+1)},
\end{cases}
\end{align}
where the process $(F_t^{(k)})_{t_{k-1} \leq t \leq t_k}$ is obtained by replacing $(Y_t,Z_t)$ in the first term of \eqref{eq : system_FBSDE} by $(Y_t^{(k)},Z_t^{(k)})$  and 
where $\xi_k^{(k+1)}$ is given by the second term \eqref{eq : system_FBSDE} where we replaced $(Y_t,Z_t)$  by $(Y_t^{(l+1)}, Z_t^{(l+1)})$ on each subinterval $t_{l} \leq t \leq t_{l+1}$,  for $k \leq l \leq N-1$.
Then, the global forward backward $(Y,Z)_{0 \leq t \leq T}$ is given by
\begin{align}
\begin{cases}
    Y_t &= \Sum_{k=1}^{N} \mathds{1}_{t_{k-1} \leq t < t_{k}} Y_{t}^{(k)} + \mathds{1}_{t=t_N} Y_{T}^{(N)}, \quad 0 \leq t \leq T, \quad \P-\text{a.s}, \\
    Z_t &= \Sum_{k=1}^{N}  \mathds{1}_{t_{k-1} \leq t < t_k} Z_t^{(k)}, \quad \d t \otimes \d \P-\text{a.e}.
\end{cases}
\end{align}

\end{Remark}

\section{Pontryagin principle for optimality}
\label{sec : pontryagin}
In this section, we discuss the necessary and sufficient conditions for optimality under suitable regularity assumptions on the Hamiltonian map $H$ and terminal condition map $g$. We start by introducing some regularity assumptions that will be used throughout this section and in the remainder of the paper. 

\begin{Assumption}\label{assumption : regularity_pontryagin_optimality}   

\noindent Assume the following holds:

\small

\begin{enumerate}
    \item [(1)] The functions $b$ and $\sigma$ are differentiable with respect to $(\bx,x,a)$ for any $t \in [0,T]$. Moreover, the maps $\partial_{\bx}(b,\sigma)$, $\partial_{x} (b,\sigma)$  and $\partial_a(b,\sigma)$ are assumed to be uniformly bounded and Lipschitz continuous and the maps
    \begin{align*}
            (\bx,\boldsymbol{\mu},x,\mu,a) \mapsto \partial_x(b,\sigma), \partial_{\bx}(b,\sigma), 
            \partial_a(b,\sigma)(t,\bx,\boldsymbol{\mu},x,\mu,a), 
    \end{align*}
    are assumed to be jointly continuous for any $t \in [0,T]$.
    \item [(2)] The functions $\phi$ $=$ $b,\sigma$ are assumed to have Fréchet differentiable linear functional derivatives $\partial_{\tilde{x}} \frac{\delta}{\delta m} \phi$ and $\partial_{\tilde{\bx}} \frac{\delta}{\delta \bm} \phi$, and there exist positive constants $L$ and $M$ such that        
\begin{align*}
\begin{cases}
\begin{aligned}
&\left|\partial_{\tilde{x}} \frac{\delta}{\delta m}
\phi(t,\bx,\boldsymbol{\mu},x,\mu,a)(\tilde{x})
-\partial_{\tilde{x}} \frac{\delta}{\delta m}
\phi(t,\bx',\boldsymbol{\mu}',x',\mu',a)(\tilde{x}')\right| \\
&\qquad \leq L\Big(
|\bx_{\eta(t)}-\bx'_{\eta(t)}|
+\Wc_2(\boldsymbol{\mu}_{\eta(t)},\boldsymbol{\mu}'_{\eta(t)})
+|x-x'| +\Wc_2(\mu,\mu')
+|\tilde{x}-\tilde{x}'|
\Big),
\end{aligned}
\\[0.5em]
\begin{aligned}
&\left|\partial_{\tilde{\bx}} \frac{\delta}{\delta \bm}
\phi(t,\bx,\boldsymbol{\mu},x,\mu,a)(\tilde{\bx})
-\partial_{\tilde{\bx}} \frac{\delta}{\delta \bm}
\phi(t,\bx',\boldsymbol{\mu}',x',\mu',a)(\tilde{\bx}')\right| \\
&\qquad \leq L\Big(
|\bx_{\eta(t)}-\bx'_{\eta(t)}|
+\Wc_2(\boldsymbol{\mu}_{\eta(t)},\boldsymbol{\mu}'_{\eta(t)})
+|x-x'|
+\Wc_2(\mu,\mu')
+|\tilde{\bx}_{\eta(t)}-\tilde{\bx}'_{\eta(t)}|
\Big),
\end{aligned}
\end{cases}
\end{align*}
    and
\begin{align*}
          &| \partial_{\tilde{x}} \tfrac{\delta}{\delta m} \phi(t, \bx ,\boldsymbol{\mu}, x,\mu, a)(\tilde{x})| + | \partial_{\tilde{\bx}} \tfrac{\delta}{\delta \bm} \phi(t, \bx ,\boldsymbol{\mu}, x,\mu, a)(\tilde{\bx}) | \;\leq\; M, 
      \end{align*}
        for any $t \in [0,T], \bx,\bx' \in (\R^d)^N$, $ \boldsymbol{\mu},\boldsymbol{\mu}' \in \Pc_2((\R^d)^N)$, $x, x' \in \R^d$, $\mu, \mu' \in \Pc_2(\R^d)$ and $a \in A$.
    
    \item [(3)] The functions $f$ and $g$ are assumed to be differentiable with respect to $(\bx,x,a)$ (resp. $\bx$), and  the maps
    \begin{align*}
        (\bx,\boldsymbol{\mu},x,\mu,a) &\mapsto \big( \partial_x f, 
        \partial_{\bx} f,  
    \partial_a f\big)(t,\bx,\boldsymbol{\mu},x,\mu,a), \quad
        (\bx,\boldsymbol{\mu}) \mapsto \partial_{\bx} g(\bx,\boldsymbol{\mu})
    \end{align*}
    are assumed to be jointly continuous for any $t \in [0,T]$ .
    Moreover, there exists a positive constant $L$ such that, for every $(t,\bx,\boldsymbol{\mu},x,\mu,a)$,
\begin{align*}
        \big( | \partial_{\bx} f | + | \partial_x f | + | \partial_a f | \big)(t,\bx,\boldsymbol{\mu},x,\mu,a) &\leq L \big( 1 + |\bx| + \Wc_2(\boldsymbol{\mu}, \boldsymbol{\delta}_0) + |x| + \Wc_2(\mu,\delta_0) + |a| \big), \\
        | \partial_{\bx} g(\bx,\boldsymbol{\mu}) | &\leq L \big( 1 + |\bx| + \Wc_2(\boldsymbol{\mu}, \boldsymbol{\delta}_0) \big).
    \end{align*}
    \item [(4)] The functions $f$ and $g$ are assumed to have Fréchet differentiable linear functional derivatives $\partial_{\tilde{x}} \frac{\delta}{\delta m} f$, $\partial_{\tilde{\bx}} \frac{\delta}{\delta \bm} f$ and $\partial_{\tilde{\bx}} \frac{\delta}{\delta \bm} g$. 
    Moreover, there exists a positive constant $L$ such that, for any $R \geq 0$, any $(t,\bx,\boldsymbol{\mu},x,\mu,a)$ with $|\bx| + \Wc_2(\boldsymbol{\mu},\boldsymbol{\delta}_0) + |x| + \Wc_2(\mu,\delta_0) + |a| \leq R$, and any $\tilde{X} \sim \mu$, $\tilde{\bX} \sim \boldsymbol{\mu}$,
\footnotesize{
\begin{align*}
          &\tilde{\E} \Big[ \big| \partial_{\tilde{x}} \tfrac{\delta}{\delta m}f(t,\bx,\boldsymbol{\mu},x,\mu,a)(\tilde{X})\big|^2 + \big| \partial_{\tilde{\bx}} \tfrac{\delta}{\delta \bm} f(t,\bx,\boldsymbol{\mu},x,\mu,a)(\tilde{\bX})\big|^2 + \big| \partial_{\tilde{\bx}} \tfrac{\delta}{\delta \bm} g(\bx,\boldsymbol{\mu})(\tilde{\bX})\big|^2 \Big] \;\leq\; L(1+R^2).
\end{align*}}
\end{enumerate}
\end{Assumption}

\subsection{A necessary condition}

Given $\alpha \in \Ac$ and an initial condition $\xi$, we denote by $X =X^{\alpha}$ the associated controlled state process solution to \eqref{eq : state_process_dynamics} and we  denote in the following  $\big(\boldsymbol{\theta}_t = (\bX, \P_{\bX},X_t,\P_{X_t},\alpha_t)\big)_{0 \leq t \leq T}$ . Moreover, given $\beta \in \Ac$, we define the variation process $V^{\beta}=V=(V_t)_{0 \leq t \leq T}$ as the solution to the following stochastic differential equation  valued in $\R^d$    
\small
\begin{align*}
\begin{cases}
       \d V_t &= \bigg[  \partial_{\bx} b(t,\boldsymbol{\theta}_t) \cdot  \bV+ \tilde{\E} \big[ \partial_{\tilde{\bx}} \frac{\delta}{\delta \bm} b(t, \boldsymbol{\theta}_t)(\tilde{\bX}) \cdot   \tilde{\bV} \big] \\
       &\qquad + \partial_{x} b(t,\boldsymbol{\theta}_t)  V_t + \tilde{\E} \big[ \partial_{\tilde{x}} \frac{\delta}{\delta m} b(t,\boldsymbol{\theta}_t)(\tilde{X}_t)  \tilde{V}_t \big] + \partial_a b(t,\boldsymbol{\theta}_t)   (\beta_t - \alpha_t)  \bigg] \d t ,\\
       &\quad + \bigg[  \partial_{\bx} \sigma(t,\boldsymbol{\theta}_t) \cdot  \bV + \tilde{\E} \big[ \partial_{\tilde{\bx}} \frac{\delta}{\delta \bm} \tilde{\sigma}(t,\boldsymbol{\theta}_t)(\tilde{\bX}) \cdot \tilde{\bV} \big] \\
       &\qquad +  \partial_{x} \sigma(t,\boldsymbol{\theta}_t)  V_t + \tilde{\E} \big[ \partial_{\tilde{x}} \frac{\delta}{\delta m} \sigma(t,\boldsymbol{\theta}_t)(\tilde{X}_t)  \tilde{V}_t \big]  +   \partial_{a} \sigma(t,\boldsymbol{\theta}_t) (\beta_t - \alpha_t) \bigg]  \d W_t, \\
    V_0 &= 0, 
\end{cases}
\end{align*}
where $\bV=(V_{t_1},\ldots,V_{t_N})$.
In fact, the pair process $(X,V)$ satisfies a SDE falling into the framework of \eqref{eq : state_process_dynamics} with control pair $(\alpha,\beta)$ and admissible initial condition $(\xi,0)$. Under Assumptions \ref{assumption : regularity_b_and_sigma} and \ref{assumption : regularity_pontryagin_optimality}, we get existence and uniqueness of $(X,V)$ hence $V$.
\begin{Lemma}\label{lemma : def_variation_process}
    Let $\epsilon \in (0,1)$ and $X^{\epsilon} = (X^{\epsilon}_t)_{0 \leq t \leq T}$ be the controlled state process with the control $\alpha^{\epsilon} := \big(\alpha_t + \epsilon (\beta_t- \alpha_t) \big)_{0 \leq t \leq T}$, Then, the following holds
    \normalfont 
    \begin{align*}
        \underset{\epsilon \to 0}{\text{ lim }} \E \Big[ \underset{0 \leq t \leq T}{\text{ sup }} |\frac{X_t^{\epsilon}-X_t}{\epsilon} - V_t|^2 \Big] = 0.
    \end{align*}
\end{Lemma}
\begin{proof}
Since $A$ is convex, we recall that $\alpha^{\epsilon} \in \Ac$. We denote by  $\boldsymbol{\theta}_{t}^{\epsilon} := (\bX^{\epsilon}, \P_{\bX^{\epsilon}}, X_t^{\epsilon},\P_{X_t^{\epsilon}}, \alpha_t + \epsilon (\beta_t - \alpha_t))$ for any $0 \leq t \leq T$ and we denote the process $V^{\epsilon}$ as $\big(V^{\epsilon}_t := \frac{1}{\epsilon}(X_t^{\epsilon} - X_t) - V_t\big)_{0 \leq t \leq T}$. Then, from Itô's dynamics, we have 
\begin{align}\label{eq : SDE_Vepsilon}
\begin{cases}
        \d V_t^{\epsilon} &= V_t^{\epsilon,1} \d t + V_t^{\epsilon,2} \d W_t, \\
        V_0^{\epsilon} &= 0,
\end{cases}
\end{align}
with
\begin{align*}
\begin{cases}
    V_t^{\epsilon,1} &= \frac{1}{\epsilon}  \big( b(t,\boldsymbol{\theta}_t^{\epsilon}) - b(t, \boldsymbol{\theta}_t) \big) -  \partial_{\bx}  b(t, \boldsymbol{\theta}_t) \cdot  \bV -  \tilde{\E} \big[ \partial_{\tilde{\bx}} \frac{\delta}{\delta \bm} b(t,\boldsymbol{\theta}_t)(\tilde{\bX})  \cdot \tilde{\bV} \big]  \\
    & \quad - \partial_{x} b(t,\boldsymbol{\theta}_t)  V_t - \tilde{\E} \big[ \partial_{\tilde{x}} \frac{\delta}{\delta m} b(t,\boldsymbol{\theta}_t)(\tilde{X}_t)  \tilde{V}_t \big] - \partial_a b(t,\boldsymbol{\theta}_t)   (\beta_t - \alpha_t),   \\
    V_t^{\epsilon,2} &=  \frac{1}{\epsilon}  \big( \sigma(t,\boldsymbol{\theta}_t^{\epsilon}) - \sigma(t, \boldsymbol{\theta}_t) \big) - \partial_{\bx}  \sigma(t,\boldsymbol{\theta}_t)  \cdot \bV -  \tilde{\E} \big[ \partial_{\tilde{\bx}} \frac{\delta}{\delta \bm} \sigma(t,\boldsymbol{\theta}_t)(\tilde{\bX}) \cdot \tilde{\bV} \big]  \\
    & \quad -  \partial_{x} \sigma(t,\boldsymbol{\theta}_t)  V_t - \tilde{\E} \big[ \partial_{\tilde{x}} \frac{\delta}{\delta m} \sigma(t,\boldsymbol{\theta}_t)(\tilde{X}_t)  \tilde{V}_t \big]  -  \partial_a \sigma(t,\boldsymbol{\theta}_t)   (\beta_t - \alpha_t).
\end{cases}
\end{align*}
We next define   the process $X^{\lambda,\epsilon}$, $\alpha^{\lambda,\epsilon}$ and $\boldsymbol{\theta}^{\lambda,\epsilon}$ as 
\begin{align*}
\begin{cases}
    X_t^{\lambda,\epsilon} &:= X_t + \lambda \epsilon (V_t^{\epsilon} + V_t) , \\
    \alpha_t^{\lambda,\epsilon} &:= \alpha_t + \lambda \epsilon (\beta_t - \alpha_t), \\
    \boldsymbol{\theta}_t^{\lambda,\epsilon} &:= (\bX^{\lambda,\epsilon}, \P_{\bX^{\lambda,\epsilon}}, X_t^{\lambda,\epsilon}, \P_{X_t^{\lambda,\epsilon}}, \alpha_t^{\lambda,\epsilon}),
\end{cases}
\end{align*}
for any $t \in [0,T]$ and any $\lambda \in [0,1]$. From the differentiability assumptions on $b$, it follows that 
\begin{align*}
    \frac{1}{\epsilon} \big[ b(t,\boldsymbol{\theta}_t^{\epsilon}) - b(t,\boldsymbol{\theta}_t) \big] = &\int_{0}^{1}  \partial_{\bx} b(t, \boldsymbol{\theta}_t^{\lambda,\epsilon}) \cdot (\bV^{\epsilon} + \bV) \d \lambda \\
    &\quad +  \int_{0}^{1} \partial_{a} b(t,\boldsymbol{\theta}_t^{\lambda,\epsilon}) (\beta_t- \alpha_t) \d \lambda  + \int_{0}^{1} \tilde{\E} \Big[  \partial_{\bx} \frac{\delta}{\delta \bm} b(t,\boldsymbol{\theta}_t^{\lambda,\epsilon})(\tilde{\bX}^{\lambda,\epsilon})  \cdot   \big( \tilde{\bV}^{\epsilon} + \tilde{\bV}\big) \Big] \d \lambda \\
    &\quad + \int_{0}^{1} \partial_x b(t,\boldsymbol{\theta}_t^{\lambda,\epsilon}) (V_t^{\epsilon} + V_t) \d \lambda + \int_{0}^{1} \tilde{\E} \big[ \partial_{\tilde{x}} \frac{\delta}{\delta m} b(t,\boldsymbol{\theta}_t^{\lambda,\epsilon})(\tilde{X}_t^{\lambda,\epsilon})  (\tilde{V}_t^{\epsilon} + \tilde{V}_t) \big] \d \lambda .
\end{align*}
Plugging this expression in $V_t^{\epsilon,1}$, we have
\small
\begin{align}\label{eq : def_V_epsilon_1}
    V_t^{\epsilon,1} &= \int_{0}^{1}  \partial_{\bx} b(t,\boldsymbol{\theta}_t^{\lambda,\epsilon}) \cdot \bV^{\epsilon} \d \lambda  + \int_{0}^{1} \tilde{\E} \Big[  \partial_{\tilde{\bx}} \frac{\delta}{\delta \bm} b(t,\boldsymbol{\theta}_t^{\lambda,\epsilon})(\tilde{\bX}^{\lambda,\epsilon}) \cdot  \tilde{\bV}^{\epsilon} \Big] \d \lambda  \\&\qquad+ \int_{0}^{1} \partial_x b(t,\boldsymbol{\theta}_t^{\lambda,\epsilon}) V_t^{\epsilon} \d \lambda  + \int_{0}^{1} \tilde{\E} \big[ \partial_{\tilde{x}} \frac{\delta}{\delta m} b(t,\boldsymbol{\theta}_t^{\lambda,\epsilon})(\tilde{X}_t^{\lambda,\epsilon}) \tilde{V}_t^{\epsilon} \big] \d \lambda  \\
    &\qquad+ I_t^{\epsilon,1} + I_t^{\epsilon,2} + I_t^{\epsilon,3} + I_t^{\epsilon,4} + I_t^{\epsilon,5},  \end{align}
where we denoted
\begin{align*}
\begin{cases}
    I_t^{\epsilon,1} := \int_{0}^{1} \Big( \partial_a b(t,\boldsymbol{\theta}_t^{\lambda,\epsilon}) - \partial_a b(t,\boldsymbol{\theta}_t) \Big) (\beta_t - \alpha_t) \d \lambda  , \\
    I_t^{\epsilon,2} := \int_{0}^{1} \Big( \partial_x b(t,\boldsymbol{\theta}_t^{\lambda,\epsilon}) - \partial_x b(t,\boldsymbol{\theta}_t) \Big) V_t \d \lambda , \\
    I_t^{\epsilon,3} :=  \int_{0}^{1}  \big( \partial_{\bx} b(t,\boldsymbol{\theta}_t^{\lambda,\epsilon}) - \partial_{\bx} b(t,\boldsymbol{\theta}_t) \big) \cdot \bV \d \lambda, \\
    I_t^{\epsilon,4} := \int_{0}^{1} \tilde{\E} \big[( \partial_{\tilde{x}} \frac{\delta}{\delta m} b(t,\boldsymbol{\theta}_t^{\lambda,\epsilon})(\tilde{X}_t^{\lambda,\epsilon}) - \partial_{\tilde{x}} \frac{\delta}{\delta m}b(t,\boldsymbol{\theta}_t)(\tilde{X}_t) \big) \tilde{V}_t \big] \d \lambda , \\
    I_t^{\epsilon,5} := \int_{0}^{1} \tilde{\E} \Big[ \big(\partial_{\tilde{\bx}} \frac{\delta}{\delta \bm}b(t,\boldsymbol{\theta}_t^{\lambda,\epsilon})(\tilde{\bX}^{\lambda,\epsilon}) - \partial_{\tilde{\bx}} \frac{\delta}{\delta \bm} b(t,\boldsymbol{\theta}_t)(\tilde{\bX}) \big) \cdot \tilde{\bV} \Big] \d \lambda .
\end{cases}
\end{align*}
Now, from \eqref{eq : SDE_Vepsilon}, and using Young's inequality, we have
\begin{align*}
    | V_t^{\epsilon}|^2 \leq  C \big( |\int_{0}^{t}  V_s^{\epsilon,1} \d s|^2 + |\int_{0}^{t} V_s^{\epsilon,2} \d W_s|^2 \big),
\end{align*}
where $C$ is a generic constant, depending on $T,N$ but not on $\epsilon$, which can vary from line to line. Taking the supremum and the expectation and then using Jensen's and BDG's inequality, we end up with
\begin{align}\label{eq : estimate_S2norm}
    \E \Big[ \underset{0 \leq t \leq r}{\text{ sup }} | V_t^{\epsilon}|^2 \Big] \leq C \Big( \int_{0}^{r} \E \big[ |V_t^{\epsilon,1}|^2 \big] \d t + \int_{0}^{r} \E \big[ | V_t^{\epsilon,2}|^2] \d t \Big),
\end{align}
for any $0 \leq r \leq T$. We now prove that $I^{\epsilon,1}, I^{\epsilon,2}, I^{\epsilon,3}, I^{\epsilon,4}$ and $I^{\epsilon,5}$ converge to $0$ in the space $L^2([0,T] \times \Omega)$, i.e.\ that
\begin{align*}
    \bar{C}_{\epsilon} := \sum_{i=1}^{5} \E \Big[ \int_{0}^{T} | I_t^{\epsilon,i}|^2 \d t \Big] \underset{\epsilon \to 0}{\longrightarrow} 0.
\end{align*}
First, since the coefficients are Lipschitz and $\alpha^{\epsilon} \to \alpha$ in $\H^2([0,T];\R^m)$, the standard stability estimate for the state equation gives $\E \big[ \sup_{0 \leq t \leq T} |X_t^{\epsilon} - X_t|^2 \big] \to 0$, so that, along any sequence $\epsilon_n \to 0$, there is a subsequence along which $\boldsymbol{\theta}_t^{\lambda,\epsilon} \to \boldsymbol{\theta}_t$ for a.e.\ $(t,\omega)$, uniformly in $\lambda \in [0,1]$. Since the maps $\partial_x b$, $\partial_a b$, $\partial_{\bx} b$, $\partial_{\tilde{x}} \frac{\delta}{\delta m} b$ and $\partial_{\tilde{\bx}} \frac{\delta}{\delta \bm} b$ are continuous, the integrands defining the $I^{\epsilon,i}$ then converge to $0$ pointwise along this subsequence and, these maps being uniformly bounded, they are dominated by $C \big( |V_t| + |\bV_{\eta(t)}| + |\beta_t - \alpha_t| \big)$, which is square integrable on $[0,T] \times \Omega$. The dominated convergence theorem therefore yields $\bar{C}_{\epsilon} \to 0$ along the subsequence and, the limit being $0$ regardless of the chosen sequence, along the whole family as $\epsilon \to 0$. Now, from the definition of $V^{\epsilon,1}$ in \eqref{eq : def_V_epsilon_1}, using the uniform boundedness of the maps $\partial_x b, \partial_{\bx} b$,  $\partial_{\tilde{x}} \frac{\delta}{\delta m} b$ and $\partial_{\tilde{\bx}} \frac{\delta}{\delta \bm} b$, together with $t_{\eta(t)} \leq t$ to bound $|\bV^{\epsilon}_{\eta(t)}|$ by the running supremum of $|V^{\epsilon}|$, we have
\begin{align*}
     \int_{0}^{r} \E \big[ |V_t^{\epsilon,1}|^2 \big] \d t   \leq C \bar{C}_{\epsilon} +  C \int_{0}^{r} \E \Big[ \underset{0 \leq s \leq t}{\text{ sup }} | V_s^{\epsilon}|^2 \Big] \d t, \quad \text{ for any $r \in [0,T]$.}
\end{align*}  The term $V^{\epsilon,2}$ is handled similarly using Jensen's and BDG's inequality and the similar assumptions on $\sigma$ from Assumption \ref{assumption : regularity_pontryagin_optimality} such that we end up with a similar estimate. Finally, plugging it into \eqref{eq : estimate_S2norm}, we have
\begin{align*}
     \E \Big[ \underset{0 \leq t \leq r}{\text{ sup }} | V_t^{\epsilon}|^2 \Big] \leq C \bar{C}_{\epsilon} + C \int_{0}^{r} \E \Big[ \underset{0 \leq s \leq t}{\text { sup }} |V_s^{\epsilon}|^2 \Big] \d  t,
\end{align*}
which is enough to conclude by applying Grönwall's Lemma to the map $r \mapsto \E \Big[ \underset{0 \leq t \leq r}{\text{ sup }} |V_t^{\epsilon}|^2 \Big]$ and sending $\epsilon \to 0$. 
\end{proof}
\begin{Lemma}\label{eq : Gateaux_derivative}
The cost functional map $J$ is Gâteaux differentiable in any $\alpha \in \Ac$ with Gâteaux derivative in the direction $\beta \in \Ac$ given by
\begin{align}\label{eq : lemma_variation_calculus}
    \underset{ \epsilon \to 0}{\text{ lim }} \frac{1}{\epsilon} \big( J(\alpha + \epsilon ( \beta- \alpha) ) - J(\alpha) \big) &= \E \bigg[ \int_{0}^{T} \Big(  \partial_{\bx} f(t,\boldsymbol{\theta}_t) \cdot \bV+  \partial_a f(t,\boldsymbol{\theta}_t) \cdot (\beta_t - \alpha_t)  \\
    &\qquad +\tilde{\E} \big[  \partial_{\tilde{\bx}} \frac{\delta}{\delta \bm} f(t,\boldsymbol{\theta}_t)(\tilde{\bX}) \cdot \tilde{\bV} \big] \Big) \d t \bigg]  \\
    &\qquad + \E  \bigg[ \int_{0}^{T} \Big(\partial_x f(t,\boldsymbol{\theta}_t) \cdot V_t  + \tilde{\E} \big[ \partial_{\tilde{x}} \frac{\delta}{\delta m} f(t,\boldsymbol{\theta}_t)(\tilde{X}_t) \cdot \tilde{V}_t \big] \Big) \d t \bigg]  \\
    &\qquad+ \E \bigg[  \partial_{\bx}g(\bX,\P_{\bX}) \cdot \bV  + \tilde{\E} \big[ \partial_{\tilde{\bx}} \frac{\delta}{\delta \bm} g(\bX,\P_{\bX})(\tilde{\bX}) \cdot \tilde{\bV} \big] \bigg].  \end{align}
\end{Lemma}
\begin{proof}
    We use the same notations as used in the proof of Lemma \ref{lemma : def_variation_process}, i.e.
\begin{align*}
\begin{cases}
    V_t^{\epsilon} &:= \frac{1}{\epsilon} \big( X_t^{\epsilon} - X_t) - V_t, \\
    \alpha_t^{\epsilon} &:= \alpha_t + \epsilon(\beta_t - \alpha_t), \\
    \bX^{\epsilon} &:= (X_{t_1}^{\epsilon},\ldots, X_{t_N}^{\epsilon}), \\
    \boldsymbol{\theta}^{\epsilon}_t &:=(\bX^{\epsilon},\P_{\bX^{\epsilon}}, X_t^{\epsilon}, \P_{X_t^{\epsilon}}, \alpha_t^{\epsilon} ), \\
    \bX^{\lambda,\epsilon} &:=  (X_{t_1}^{\lambda,\epsilon}, \ldots, X_{t_N}^{\lambda,\epsilon}),\\
    \alpha_t^{\lambda,\epsilon} &:= \alpha_t + \lambda \epsilon (\beta_t - \alpha_t), \\
\end{cases}
\end{align*}
for any $\lambda \in [0,1]$, $t \in [0,T]$ and where $X^{\epsilon}$ is the controlled state process by $\alpha^{\epsilon}$.
By definition of the cost functional $J$, we have 
\begin{align*}
    \underset{\epsilon \to 0}{\text{ lim }} \frac{1}{\epsilon} &\big( J\big(\alpha + \epsilon (\beta- \alpha) \big) - J(\alpha) \big) =   \underset{\epsilon \to 0}{\text{ lim }} \frac{1}{\epsilon} \E \Big[  \int_{0}^{T} \big[f(t,\boldsymbol{\theta}_t^{\epsilon}) - f(t,\boldsymbol{\theta}_t) \big] \d t + g(\bX^{\epsilon}, \P_{\bX^{\epsilon}}) - g(\bX, \P_{\bX}) \Big].
\end{align*}
We therefore have similarly to the computation done in the proof of Lemma \ref{lemma : def_variation_process}
\begin{align*}
    \frac{1}{\epsilon} \big( J\big(\alpha + \epsilon (\beta&- \alpha) \big) - J(\alpha) \big)  \\=  \E \bigg[ &\int_{0}^{T} \int_{0}^{1} \Big( \partial_{\bx} f(t,\boldsymbol{\theta}_t^{\lambda,\epsilon}) \cdot (\bV^{\epsilon} + \bV) + \partial_{x} f(t,\boldsymbol{\theta}_t^{\lambda,\epsilon}) \cdot (V_t^{\epsilon} + V_t) + \partial_{a} f(t,\boldsymbol{\theta}_t^{\lambda,\epsilon}) \cdot (\beta_t-\alpha_t) \\
    &\quad\quad\quad + \tilde{\E} \Big[  \partial_{\tilde{\bx}} \frac{\delta}{\delta \bm}f(t,\boldsymbol{\theta}_t^{\lambda,\epsilon})(\tilde{\bX}^{\lambda,\epsilon}) \cdot (\tilde{\bV}^{\epsilon} + \tilde{\bV}) \Big] + \tilde{\E} \Big[ \partial_{\tilde{x}} \frac{\delta}{\delta m} f(t,\boldsymbol{\theta}_t^{\lambda,\epsilon})(\tilde{X}_t^{\lambda,\epsilon}) \cdot (\tilde{V}_t^{\epsilon} + \tilde{V}_t) \Big] \Big) \d \lambda \d t \\
    &\quad + \int_{0}^{1} \Big( \partial_{\bx} g( \bX^{\lambda,\epsilon}, \P_{\bX^{\lambda,\epsilon}}) \cdot (\bV^{\epsilon} + \bV) + \tilde{\E} \Big[ \partial_{\tilde{\bx}} \frac{\delta}{\delta \bm} g(\bX^{\lambda,\epsilon} , \P_{\bX^{\lambda,\epsilon}})(\tilde{\bX}^{\lambda,\epsilon}) \cdot ( \tilde{\bV}^{\epsilon} + \tilde{\bV}) \Big] \Big) \d \lambda \bigg].
\end{align*}
Now, using the uniform convergence $\lVert \bV^{\epsilon} \rVert^2_{\S^2([0,T];\R^d)} \underset{\epsilon \to 0}{\to }0$ from Lemma \ref{lemma : def_variation_process} and the assumptions on $f$ and $g$ in Assumption \ref{assumption : regularity_pontryagin_optimality}, we conclude similarly to \cite[Lemma 6.11]{carmona2018probabilistic_I}.
\end{proof}
\begin{Lemma}\label{lemma : intermediate_computation}
    Let $(Y,Z)=(Y_t,Z_t)_{0 \leq t \leq T}$ be an adjoint process solution to Definition \ref{def : adjoint_processes}. We then have 
    \begin{align*}
        \sum_{i=1}^{N-1} \E \big[ - \Delta Y_{t_i} \cdot V_{t_i} \big] + \E \big[ Y_{t_N} \cdot V_{t_N} \big] &= \E \bigg[ \int_{0}^{T}  \bigg(\partial_{\bx} H(t,\boldsymbol{\Theta}_t) \cdot \bV + \tilde{\E} \big[ \partial_{\tilde{\bx}} \frac{\delta}{\delta \bm}H(t,\tilde{\boldsymbol{\Theta}}_t)(\bX) \cdot \bV \big]  \bigg) \d t  \bigg] \\
        &\qquad +  \E \bigg[  \partial_{\bx}g(\bX,\P_{\bX}) \cdot \bV + \tilde{\E} \big[  \partial_{\tilde{\bx}} \frac{\delta}{\delta \bm}g(\tilde{\bX},\P_{\bX})(\bX) \cdot \bV \big] \bigg],
    \end{align*}
    where we denote the process $\boldsymbol{\Theta}$ as $\big(\boldsymbol{\Theta}_t := (\bX, \P_{\bX}, X_t, \P_{X_t}, Y_t,Z_t,\alpha_t)\big)_{0 \leq t \leq T}$ and independent copy $\tilde{\boldsymbol{\Theta}}$ defined on another probability space $(\tilde{\Omega},\tilde{\Fc},\tilde{\P})$.
\end{Lemma}
\begin{proof}
This follows from the definition of the jumps $(\Delta Y_{t_i})_{1 \leq i \leq N-1}$ in Definition \ref{def : adjoint_processes}, after noticing that $i \leq \eta(t)$ is equivalent to $t \geq t_i$. Indeed, we notice that
\begin{align}
    \sum_{i=1}^{N-1} \E \Big[ \int_{t_i}^{T} \partial_{x_i}H(t,\Theta_t)  V_{t_i} \d t \Big] &=  \E\Big[ \int_{0}^{T} \sum_{i=1}^{N-1} \mathds{1}_{\{t \geq t_i\}} \partial_{x_i} H(t,\Theta_t)  V_{t_i} \d t \Big] \\
    &=\E \Big[ \int_{0}^{T} \sum_{i=1}^{\eta(t)} \partial_{x_i} H(t,\Theta_t)  V_{t_i} \d t \Big] \\
    &= \E \Big[ \int_{0}^{T} \partial_{\boldsymbol{x}} H(t,\Theta_t) \cdot \boldsymbol{V} \d t \Big],
\end{align}
recalling the notation in \eqref{eq : def_derivative}. 
Proceeding similarly for the other term yields the result.

\end{proof}

\begin{Corollary}\label{corollary:calculus_of_variations}
    We have
\normalfont
{\footnotesize\begin{align*}
     \underset{\epsilon \to 0}{\text{ lim }} \frac{1}{\epsilon} \big( J\big(\alpha + \epsilon (\beta- \alpha) \big) - J(\alpha) \big) = \E \Big[ \int_{0}^{T} \partial_a H(t, \bX, \P_{\bX},X_t,\P_{X_t},Y_t,Z_t,\alpha_t) \cdot (\beta_t- \alpha_t) \d t \Big].
\end{align*}}
\end{Corollary}
\begin{proof}
Let the processes $\boldsymbol{\theta}$, $\tilde{\boldsymbol{\theta}}$, $\boldsymbol{\Theta}$ and $\tilde{\boldsymbol{\Theta}}$ defined as
\begin{align*}
\begin{cases}
    \boldsymbol{\theta}_t:= (\bX, \P_{\bX}, X_t, \P_{X_t}, \alpha_t) ,\\
    \tilde{\boldsymbol{\theta}}_t := (\tilde{\bX}, \P_{\bX}, \tilde{X}_t, \P_{X_t}, \tilde{\alpha}_t), \\
    \boldsymbol{\Theta}_t := (\bX, \P_{\bX}, X_t, \P_{X_t}, Y_t, Z_t, \alpha_t), \\
    \tilde{\boldsymbol{\Theta}}_t :=(\tilde{\bX}, \P_{\bX}, \tilde{X}_t, \P_{X_t}, \tilde{Y}_t, \tilde{Z}_t, \tilde{\alpha}_t),
\end{cases}
\end{align*}
for any $t \in [0,T]$ and  where $(\tilde{X},\tilde{\alpha}, \tilde{Y},\tilde{Z})$ is an independent copy of $(X,\alpha,Y,Z)$ defined on another probability space $(\tilde{\Omega}, \tilde{\Fc}, \tilde{\P})$. We recall that $V$ is a continuous process while $Y$ has jumps at deterministic times $(t_k)_{1 \leq k \leq N-1}$. For convenience, we denote
\begin{align*}
\begin{cases}
   \d Y_t &= - A_t \d t + Z_t \d W_t, \quad t_{k-1} \leq t < t_k, \quad k \in \llbracket 1, N-1 \rrbracket, \quad \text{ and $t_{N-1} \leq t \leq t_N$,} \\
    \d V_t &= C_t \d t + D_t \d W_t
\end{cases}
\end{align*}
where
\begin{align}\label{eq : def_dynamics_variation_process}
\begin{cases}
    A_t &= \partial_{x} H(t,\boldsymbol{\Theta}_t) + \tilde{\E} \Big[ \partial_{\tilde{x}} \frac{\delta}{\delta m} H(t,\boldsymbol{\tilde{\Theta}_t})(X_t) \Big], \\
    C_t &=   \partial_{\bx}  b(t, \boldsymbol{\theta}_t) \cdot  \bV+ \tilde{\E} \big[ \partial_{\tilde{\bx}} \frac{\delta}{\delta \bm} b(t,\boldsymbol{\theta}_t)(\tilde{\bX}) \cdot \tilde{\bV} \big] \\
    &\qquad + \partial_{x} b(t,\boldsymbol{\theta}_t)  V_t + \tilde{\E} \big[ \partial_{\tilde{x}} \frac{\delta}{\delta m} b(t,\boldsymbol{\theta}_t)(\tilde{X}_t)  \tilde{V}_t \big] \\
    &\qquad+ \partial_a b(t,\boldsymbol{\theta}_t)   (\beta_t - \alpha_t), \\
    D_t &=   \partial_{\bx} \sigma(t,\boldsymbol{\theta}_t) \cdot \bV + \tilde{\E} \Big[ \partial_{\tilde{\bx}} \frac{\delta}{\delta \bm} \sigma(t,\boldsymbol{\theta}_t) (\tilde{\bX}) \cdot\tilde{\bV} \Big] \\
    &\qquad +  \partial_{x} \sigma(t,\boldsymbol{\theta}_t)  V_t + \tilde{\E} \big[ \partial_{\tilde{x}} \frac{\delta}{\delta m} \sigma(t,\boldsymbol{\theta}_t)(\tilde{X}_t)  \tilde{V}_t \big] \\
    &\qquad+ \partial_{a} \sigma(t,\boldsymbol{\theta}_t) (\beta_t - \alpha_t).
\end{cases}
\end{align}
Applying It\^o's formula for semi-martingales  to the product process $Y\cdot V$ on each deterministic interval $t_k \leq t < t_{k+1}$, on which $Y$ has no jumps and $V$ is continuous, gives
\begin{align*}
    Y_t \cdot  V_t &= Y_{t_k} \cdot V_{t_k}  + \int_{t_k}^{t} V_s \cdot \d Y_s + \int_{t_k}^{t} Y_s \cdot \d V_s + \int_{t_k}^{t}  Z_s \cdot D_s \d s  \\
    &= Y_{t_k} \cdot V_{t_k} + \int_{t_k}^{t} \Big( - A_s \cdot V_s +  Y_s \cdot C_s + Z_s \cdot D_s \Big) \d s  + \int_{t_k}^{t}  \Big( V_s \cdot  (Z_s \d W_s) + Y_s  \cdot (D_s  \d W_s) \Big).
\end{align*}
Letting $t \nearrow t_{k+1}$ and using the continuity in time of the process $V$, together with the continuity of the maps $t \mapsto \int_{t_k}^{t} (- A_s \cdot V_s +Y_s \cdot C_s + Z_s \cdot D_s) \d s $  and $t \mapsto \int_{t_k}^{t} \big(V_s \cdot (Z_s \d W_s)+ Y_s \cdot( D_s \d W_s) \big)$, we can write
\begin{align*}
    Y_{t_{k+1}^{-}} \cdot V_{t_{k+1}} &= Y_{t_k} \cdot V_{t_k} + \int_{t_k}^{t_{k+1}} ( - A_s \cdot V_s +Y_s \cdot  C_s + Z_s \cdot D_s) \d s  + \int_{t_k}^{t_{k+1}} \big(V_s  \cdot (Z_s \d W_s) +Y_s \cdot (D_s  \d W_s) \big). 
\end{align*}
Since $Y_{t_k} = (Y_{t_k}^{-} + \Delta Y_{t_k})$ for any $k \in \llbracket 1, N-1 \rrbracket$, we therefore have
{\small\begin{align*}
    Y_{t_{k+1}} \cdot V_{t_{k+1}} &= Y_{t_k} \cdot  V_{t_k} + \int_{t_k}^{t_{k+1}} (- A_s \cdot V_s +  Y_s \cdot  C_s + Z_s \cdot D_s ) \d s  + \int_{t_k}^{t_{k+1}} \Big( V_s \cdot  (Z_s \d W_s) + Y_s  \cdot (D_s  \d W_s) \Big) + \Delta Y_{t_{k+1}} \cdot  V_{t_{k+1}}.
\end{align*}}
Now, taking the sum from $k=0$ to $N-1$ and since $\Delta Y_{t_N} =0$ (no jumps), we have
\begin{align*}
   Y_{T} \cdot V_{T} &=  Y_0 \cdot  V_0 + \int_{0}^{T} \big( - A_s \cdot V_s +  Y_s\cdot  C_s + Z_s \cdot D_s \big) \d s + \int_{0}^{T} \Big( V_s \cdot  (Z_s \d W_s) + Y_s  \cdot (D_s  \d W_s) \Big) + \sum_{k=1}^{N-1} \Delta Y_{t_k} \cdot  V_{t_k}.
\end{align*}
Recalling that $V_0=0$, and taking the expectation, we have
\begin{align}\label{eq : identity_jumps}
   -\sum_{k=1}^{N-1} \E \big[ \Delta Y_{t_k} \cdot V_{t_k}  \big]  + \E \big[ Y_{T} \cdot V_T \big] = \E \Big[ \int_{0}^{T} (- A_t \cdot V_t +  Y_t \cdot C_t + Z_t \cdot D_t) \d t \Big].
\end{align}
Now, from Lemma \ref{lemma : intermediate_computation} and recalling the definitions of the processes $A,C$ and $D$ in \eqref{eq : def_dynamics_variation_process}, we have
\begin{equation}
    \label{eq : intermediate_computation_SMP}\begin{split}
 &\E \Bigg[ \int_{0}^{T} \bigg( \Big( \partial_{\bx} f(t,\boldsymbol{\theta}_t) + \tilde{\E} \big[ \partial_{\tilde{\bx}} \frac{\delta}{\delta \bm} f(t,\tilde{\boldsymbol{\theta}}_t)(\bX) \big]\Big)  \cdot \bV + \Big( \partial_{x} f(t,\boldsymbol{\theta}_t) + \tilde{\E} \big[ \partial_{\tilde{x}} \frac{\delta}{\delta m} f(t,\tilde{\boldsymbol{\theta}}_t)(X_t) \big] \Big) \cdot V_t  \bigg) \d t \Bigg] \\
        &\quad +  \E \bigg[  \partial_{\bx}g(\bX,\P_{\bX}) \cdot \bV  + \tilde{\E} \big[ \partial_{\tilde{\bx}} \frac{\delta}{\delta \bm}g(\tilde{\bX},\P_{\bX})(\bX) \cdot \bV \big] \bigg]  \\
        &= \sum_{i=1}^{N-1} \E \big[ - \Delta Y_{t_i} \cdot V_{t_i} \big] + \E \big[ Y_T \cdot  V_T \big] + \E \Bigg[ \int_{0}^{T} \bigg( \partial_{x} f(t,\boldsymbol{\theta}_t) + \tilde{\E} \big[ \partial_{\tilde{x}} \frac{\delta}{\delta m} f(t,\tilde{\boldsymbol{\theta}}_t)(X_t) \big] \bigg) \cdot V_t  \d t  \Bigg]  \\
        &\quad - \E \Big[ \int_{0}^{T}  Y_t \cdot \big(\partial_{\bx}  b(t,\boldsymbol{\theta}_t) \cdot \bV \big)  +  \tilde{\E} \big[ Y_t \cdot \big( \partial_{\tilde{\bx}} \frac{\delta}{\delta \bm}b(t,\boldsymbol{\theta}_t)(\tilde{\bX})  \cdot \tilde{\bV} \big) \big] \d t \Big]  \\
        & \quad - \E \Big[ \int_{0}^{T}  Z_t \cdot \big(\partial_{\bx}  \sigma(t,\boldsymbol{\theta}_t) \cdot  \bV \big)+  \tilde{\E} \big[ Z_t \cdot \big( \partial_{\tilde{\bx}} \frac{\delta}{\delta \bm} \sigma(t,\boldsymbol{\theta}_t)(\tilde{\bX})  \cdot \tilde{\bV}\big) \big] \d t \Big]  \\
        &= \E \bigg[ \int_{0}^{T} \Big(  Y_t \cdot (\partial_{a} b(t,\boldsymbol{\theta}_t)    (\beta_t - \alpha_t)) + Z_t \cdot (\partial_{a} \sigma(t,\boldsymbol{\theta}_t)  (\beta_t - \alpha_t))  \Big) \d t \bigg],  \end{split}\end{equation}
where we used in the last equality the identity \eqref{eq : identity_jumps}.
Now plugging \eqref{eq : intermediate_computation_SMP} into \eqref{eq : lemma_variation_calculus}  and using Lemma \ref{eq : Gateaux_derivative}gives
{\footnotesize\begin{align*}
     \underset{ \epsilon \to 0}{\text{ lim }} \frac{1}{\epsilon} \big( J(\alpha + \epsilon ( \beta- \alpha) ) - J(\alpha) \big) = \E \Big[ \int_{0}^{T} \partial_{a} H(t,\bX, \P_{\bX},X_t,\P_{X_t},Y_t,Z_t,\alpha_t)\cdot(\beta_t- \alpha_t) \d t  \Big],
\end{align*}}
which ends the proof.
\end{proof}
\begin{Theorem}\label{thm : necessary_condition_optimality}
    Let Assumption \ref{assumption : regularity_pontryagin_optimality} hold and assume that the map $H$ is furthermore a convex map with respect to its last variable, i.e.\ the map
    \begin{align*}
        A \ni a \mapsto H(t,\bx,\boldsymbol{\mu},x,\mu,y,z,a),
    \end{align*}
    is convex for any $(t,\bx,\boldsymbol{\mu},x,\mu,y,z) \in [0,T] \times (\R^d)^N \times \Pc_2( (\R^d)^N) \times \R^d \times \Pc_2(\R^d) \times  \R^d \times \R^{d \times n}$. Let $\hat{\alpha}=(\hat{\alpha}_t)_{0 \leq t \leq T}$ be an optimal control and $(X_t,Y_t,Z_t)_{0 \leq t \leq T}$ the associated forward and adjoint processes. Then, we have
    \begin{align*}
        H(t,\bX, \P_{\bX},X_t,\P_{X_t}, Y_t,Z_t, \hat{\alpha}_t) \leq H(t,\bX,\P_{\bX},X_t,\P_{X_t},Y_t,Z_t,a) \quad \d t \otimes \d \P-\text{a.e},
    \end{align*}
    for every $a \in A$.
\end{Theorem}
\begin{proof}
The proof is similar to the one in \cite[Theorem 6.14]{carmona2018probabilistic_I}.
\end{proof}

\subsection{A sufficient condition}

\begin{Theorem}\label{thm : sufficient_condition_optimality}
    Let $\alpha =(\alpha_t)_{0 \leq t \leq T} \in \Ac$ and $(X_t,Y_t,Z_t)_{0 \leq t \leq  T}$ be the associated controlled state and adjoint processes. If we assume
    \begin{enumerate}
        \item [(1)] $(\R^d)^N \times \Pc_2((\R^d)^N) \ni (\bx,\boldsymbol{\mu}) \mapsto g(\bx,\boldsymbol{\mu})$ is convex in the sense of Definition \ref{def: convexity_maps}.
        \item [(2)] $(\R^d)^N \times \Pc_2((\R^d)^N) \times \R^d \times \Pc_2(\R^d) \times A \ni (\bx,\boldsymbol{\mu},x,\mu,a) \mapsto H(t,\bx,\boldsymbol{\mu},x,\mu, Y_t,Z_t,a)$ is convex $\d t \otimes \d \P$-\text{a.e} in the sense of Definition \ref{def: convexity_maps}.
    \end{enumerate}
    If we assume furthermore
    {\footnotesize\begin{align*}
        H(t,\bX, \P_{\bX},X_t,\P_{X_t},Y_t,Z_t,\alpha_t) = \underset{a \in A}{\text{ inf }} H(t,\bX, \P_{\bX},X_t,\P_{X_t},Y_t,Z_t,a), \quad \d t \otimes \d \P-\text{a.e},
    \end{align*}}
    then $\alpha$ is an optimal control, i.e.\ $J(\alpha) = \underset{\alpha' \in \Ac}{\text{ inf }} J(\alpha')$.
\end{Theorem}
\begin{proof}
    Let $\hat{\alpha} = (\hat{\alpha}_t)_{0 \leq t \leq T} \in \Ac$ and we denote $\hat{\boldsymbol{\theta}}_t := (\hat{\bX}, \P_{\hat{\bX}},\hat{X}_t, \P_{\hat{X}_t},\hat{\alpha}_t)$ and $\hat{\boldsymbol{\Theta}}_t := (\hat{\bX}, \P_{\hat{\bX}}, \hat{X}_t, \P_{\hat{X}_t},Y_t,Z_t,\hat{\alpha}_t)$ for any $0 \leq t \leq T$ where $(\hat{X}_t)_{0 \leq t \leq T}$ denotes the controlled state process by $\hat{\alpha}$. We then have
    \begin{align}\label{eq : optimality_J}
        J(\alpha) -J(\hat{\alpha}) =  \E \Big[ g(\bX,\P_{\bX}) - g(\hat{\bX},\P_{\hat{\bX}}) \Big] + \E \Big[ \int_{0}^{T} \big( f(t,\boldsymbol{\theta}_t) - f(t,\hat{\boldsymbol{\theta}}_t) \big) \d t \Big].
    \end{align}
    From the convexity assumption on $g$ applied pointwise in $\omega$ to the pair $((\bx,\boldsymbol{\mu}),(\bx',\boldsymbol{\mu}')) = ((\bX(\omega),\P_{\bX}),(\hat{\bX}(\omega),\P_{\hat{\bX}}))$, with the iid tilde-copy coupling $(\tilde{\bX},\tilde{\hat{\bX}}) \sim (\P_{\bX},\P_{\hat{\bX}})$ with $(\tilde{\P}_{\tilde{\bX}}, \tilde{\P}_{\tilde{\hat{\bX}}}) = (\P_{\bX}, \P_{\hat{X}})$ on $(\tilde\Omega,\tilde\Fc,\tilde\P)$ playing the role of $(\bX,\bX')$ in Definition \ref{def: convexity_maps}, we obtain
    \begin{align}\label{eq : conv_g_pointwise}
        g(\bX,\P_{\bX}) - g(\hat{\bX}, \P_{\hat{\bX}}) \leq  \partial_{\bx} g(\bX, \P_{\bX}) \cdot (\bX - \hat{\bX}) + \tilde{\E} \Big[  \partial_{\tilde{\bx}} \frac{\delta}{\delta \bm} g(\bX, \P_{\bX})(\tilde{\bX}) \cdot (\tilde{\bX} - \tilde{\hat{\bX}}) \Big].
    \end{align}
    Taking the expectation and using Fubini's theorem we have
    \begin{align}\label{eq : relabel_g}
        \E  \bigg[\tilde{\E}\Big[ \partial_{\tilde{\bx}} \tfrac{\delta}{\delta \bm} g(\bX, \P_{\bX})(\tilde{\bX}) \cdot (\tilde{\bX} - \tilde{\hat{\bX}}) \Big] \bigg] = \E \bigg[ \tilde{\E} \Big[ \partial_{\tilde{\bx}} \tfrac{\delta}{\delta \bm} g(\tilde{\bX}, \P_{\bX})(\bX) \cdot (\bX - \hat{\bX}) \Big] \bigg],
    \end{align}
    so that recalling the notation \eqref{eq : derivative_notations}, we have
\begin{equation}
        \label{eq : ineq_convexity_g}\begin{split}
        \E \Big[  g(\bX,\P_{\bX}) - g(\hat{\bX}, \P_{\hat{\bX}})  \Big] &\leq \E \Big[ \sum_{i=1}^{N-1} - \Delta Y_{t_i} \cdot (X_{t_i} - \hat{X}_{t_i}) \Big] + \E \Big[ Y_{t_N} \cdot (X_{t_N} - \hat{X}_{t_N}) \Big] \\
        & \qquad- \E \Bigg[ \int_{0}^{T} \bigg(\partial_{\bx}  H(t,\boldsymbol{\Theta}_t) + \tilde{\E} \big[ \partial_{\tilde{\bx}} \frac{\delta}{\delta \bm} H(t,\tilde{\boldsymbol{\Theta}}_t)(\bX) \big] \bigg) \cdot  (\bX - \hat{\bX}) \d t \Bigg], \end{split}\end{equation}
    where we have used the i.i.d. relabeling identity \eqref{eq : relabel_g} (applied to $\tilde{H}$ in place of $g$) to rewrite the Lions-measure-derivative term in terms of the outer increment $(X_{t_i} - \hat{X}_{t_i})$; this is precisely the shape produced by expanding the adjoint jump $-\Delta Y_{t_i}$ in Definition \ref{def : adjoint_processes}.
By Itô's formula, and following the same computations as done in the proof of Corollary \ref{corollary:calculus_of_variations}, we have
{\small\begin{align*}
    \E \Big[\sum_{i=1}^{N-1} &- \Delta Y_{t_i} \cdot (X_{t_i} - \hat{X}_{t_i}) \Big] + \E \Big[ Y_{t_N} \cdot (X_{t_N} - \hat{X}_{t_N}) \Big] \\
    &=\E \Big[ \int_{0}^{T} - A_t \cdot (X_t - \hat{X}_t) +Y_t \cdot \big( b(t,\boldsymbol{\theta}_t ) - b(t,\hat{\boldsymbol{\theta}}_t) \big) + Z_t \cdot \big(\sigma(t,\boldsymbol{\theta}_t) - \sigma(t,\hat{\boldsymbol{\theta}}_t)\big) \d t \Big],
\end{align*}}
where the process $A$ has been defined in \eqref{eq : def_dynamics_variation_process}.
Recalling the definition of the Hamiltonian map $H$ in \eqref{eq : hamiltonian_map} and plugging \eqref{eq : ineq_convexity_g} into \eqref{eq : optimality_J} gives
\begin{align*}
    J(\alpha) - J(\hat{\alpha}) &\leq \E \Big[ \int_{0}^{T} \Big(H(t,\boldsymbol{\Theta}_t) - H(t,\hat{\boldsymbol{\Theta}}_t) \Big) \d t \Big] \\
    &\quad - \E \Big[ \int_{0}^{T}  \Big(\partial_{\bx} H(t,\boldsymbol{\Theta}_t) + \tilde{\E} \big[ \partial_{\tilde{\bx}} \frac{\delta}{\delta \bm} H(t,\tilde{\boldsymbol{\Theta}}_t)(\bX) \big] \Big) \cdot (\bX - \hat{\bX}) \d t \Big] \\
    &\quad - \E \Big[ \int_{0}^{T} \Big(  \partial_{x} H(t,\boldsymbol{\Theta}_t) + \tilde{\E} \big[ \partial_{\tilde{x}} \frac{\delta}{\delta m} H(t,\tilde{\boldsymbol{\Theta}}_t)(X_t) \big] \Big) \cdot (X_t - \hat{X}_t) \d t  \Big]
    \\
    &\leq \E \Big[ \int_{0}^{T} \partial_{a} H(t,\boldsymbol{\Theta}_t) \cdot  (\alpha_t - \hat{\alpha}_t) \d t \Big],
\end{align*}
where the last inequality follows from the convexity assumption on $H$ in Definition \ref{def: convexity_maps}: the pointwise convexity inequality is applied at $\omega \in \Omega$ to the pair $((\bx,\boldsymbol{\mu},x,\mu,a),$ $(\bx',\boldsymbol{\mu}',x',\mu',a'))$ $=(\boldsymbol{\theta}_t(\omega),\hat{\boldsymbol{\theta}}_t(\omega))$ with the i.i.d. tilde-copy coupling $(\tilde{\bX},\tilde{\hat{\bX}},\tilde{X},\tilde{\hat{X}})$ on $(\tilde\Omega,\tilde\Fc,\tilde\P)$, and the Lions-measure-derivative terms are then turned into the outer-increment shape above via the relabeling identity \eqref{eq : relabel_g} applied to $\tilde{H}_t$ in place of $g$ (and, for the $\frac{\delta}{\delta m}$ term, to $(X_t,\hat{X}_t)$ in place of $(\bX,\hat{\bX})$). From the convexity of $H$ with respect to its control argument, together with the minimality assumption on $\alpha$, the first-order optimality condition yields
\begin{align*}
    \partial_{a} H(t,\boldsymbol{\Theta}_t) \cdot (\alpha_t-a) \leq 0 \quad \d t \otimes \d \P-\text{a.e},
\end{align*}
for every $a \in A$, which yields the result.
\end{proof}

\section{Solvability of the Pontryagin FBSDE}
\label{sec : solvability}
We now propose two applications of this extension of McKean--Vlasov optimal control to the path-dependent setting. We first extend the classical linear-quadratic control problem to the path-dependent McKean--Vlasov setting, and  propose a natural formulation for the generation of time series.

\subsection{Linear dynamics and \texorpdfstring{$L$}{L}-convexity of \texorpdfstring{$f$}{f} and \texorpdfstring{$g$}{g}}

\begin{Assumption}\label{assumption:solvability_LQ}
There exists two constants $L \geq 0$ and $\lambda > 0$ such that:

\begin{enumerate}
    \item [(1)] The drift and volatility functions $b$ and $\sigma$ are linear in $\boldsymbol{x}$, $\boldsymbol{\mu}$, $x$, $\mu$ and $a$, i.e.\ for any $(t,\bx,\boldsymbol{\mu},x,\mu,a) \in [0,T] \times (\R^d)^N \times \Pc_2((\R^d)^N) \times \R^d \times \Pc_2(\R^d) \times A$, we assume that
    \begin{align*}
    \begin{cases}
        b(t,\bx,\boldsymbol{\mu},x,\mu,a) &:= b_0(t) + \sum_{j=1}^{\eta(t)} B_j(t) x_j + \sum_{j=1}^{\eta(t)} \bar{B}_j(t) \bar{\mu}_j+ C(t) x + \bar{C}(t) \bar{\mu} + D(t) a , \\
        \sigma(t,\bx,\boldsymbol{\mu},x,\mu,a) &:= \sigma_0(t) + \sum_{j=1}^{\eta(t)} E_j(t)x_j + \sum_{j=1}^{\eta(t)} \bar{E}_j(t) \bar{\mu}_j + F(t) x + \bar{F}(t) \bar{\mu} + G(t) a,
    \end{cases}
    \end{align*}
where all coefficients are deterministic  functions of time, with
$b_0$ valued in $\R^d$,
$B_i,\bar B_i,C,\bar C$ in $\R^{d\times d}$,
$D$ in $\R^{d\times m}$,
$\sigma_0$ in $\R^{d\times n}$,
$E_i,\bar E_i,F,\bar F$ in $\R^{(d\times n)\times d}$,
and $G$ in $\R^{(d\times n)\times m}$,
for $i \in \llbracket 1, N \rrbracket$ and where we used 
    \begin{align*}
        \bar{\mu} := \int_{\R^d} x \mu(\d x), \quad \bar{\mu}_{j} :=  \int_{\R^d} x \mu_j(\d x) , \quad j \in \llbracket 1, N \rrbracket,
    \end{align*}
   for, respectively, the mean of $\mu$ and the mean of $\mu_j= \text{pr}_j \sharp \boldsymbol{\mu}$.
    \item [(2)] The derivatives of $f$ and $g$ with respect to $(\bx,x,a)$ and $\bx$ respectively are $L$-Lipschitz continuous with respect to $(\bx,\boldsymbol{\mu},x,\mu,a)$ and $(\bx,\boldsymbol{\mu})$ respectively. Moreover, the following holds
{
\begin{align*}
\begin{cases}
&\E\Big[
\big|
\partial_{\tilde{x}} \frac{\delta}{\delta m} f(t,\boldsymbol{\theta}')(X')
-
\partial_{\tilde{x}} \frac{\delta}{\delta m} f(t,\boldsymbol{\theta})(X)
\big|^2
\Big] \\
&\qquad \leq L\Big(
|\bx'_{\eta(t)}-\bx_{\eta(t)}|^2
+\Wc_2^2(\boldsymbol{\mu}'_{\eta(t)},\boldsymbol{\mu}_{\eta(t)})
+|x'-x|^2
+\Wc_2^2(\mu',\mu)
+|a'-a|^2
\Big), \\[0.4em]
&\E\Big[
\big|
\partial_{\tilde{\bx}} \frac{\delta}{\delta \bm} f(t,\boldsymbol{\theta}')(\bX')
-
\partial_{\tilde{\bx}} \frac{\delta}{\delta \bm} f(t,\boldsymbol{\theta})(\bX)
\big|^2
\Big] \\
&\qquad \leq L\Big(
|\bx'_{\eta(t)}-\bx_{\eta(t)}|^2
+\Wc_2^2(\boldsymbol{\mu}'_{\eta(t)},\boldsymbol{\mu}_{\eta(t)})
+|x'-x|^2 +\Wc_2^2(\mu',\mu)
+|a'-a|^2
\Big), \\[0.4em]
&\E\Big[
\big|
\partial_{\tilde{\bx}} \frac{\delta}{\delta \bm} g(\bx',\boldsymbol{\mu}')(\bX')
-
\partial_{\tilde{\bx}} \frac{\delta}{\delta \bm} g(\bx,\boldsymbol{\mu})(\bX)
\big|^2
\Big] \\
&\qquad \leq L\Big(
|\bx'-\bx|^2
+\Wc_2^2(\boldsymbol{\mu}',\boldsymbol{\mu})
\Big),
\end{cases}
\end{align*}
}

    for any $t \in [0,T]$, $\boldsymbol{\theta}=(\bx,\boldsymbol{\mu}, x,\mu,a), \boldsymbol{\theta}' =(\bx',\boldsymbol{\mu}',x',\mu',a') \in (\R^d)^N \times \Pc_2((\R^d)^N) \times \R^d \times \Pc_2(\R^d) \times A$ and where $\bX =(X_{1},\ldots,X_N) \sim \boldsymbol{\mu}$ and $\bX' =(X'_1,\ldots, X'_N) \sim \boldsymbol{\mu}'$ and $X \sim \mu$ and $X' \sim \mu'$. 
    \item [(3)] The function $f$ satisfies the $L$-convexity property
\begin{align*}
    &f(t,\boldsymbol{\theta}') - f(t,\boldsymbol{\theta})
    - \partial_x f(t,\boldsymbol{\theta}) \cdot (x'-x)
    - \E\Big[
        \partial_{\tilde{x}} \frac{\delta}{\delta m} f(t,\boldsymbol{\theta})(X)
        \cdot (X'-X)
    \Big] \\
    &\qquad
    - \partial_{\bx} f(t,\boldsymbol{\theta}) \cdot (\bx' - \bx)
    - \E\Big[
        \partial_{\tilde{\bx}} \frac{\delta}{\delta \bm} f(t,\boldsymbol{\theta})(\bX)
        \cdot (\bX' - \bX)
    \Big]
    - \partial_a f(t,\boldsymbol{\theta}) \cdot (a'-a)
    \geq \lambda |a'-a|^2 ,
\end{align*}
for any $t \in [0,T]$, $\boldsymbol{\theta}=(\bx,\boldsymbol{\mu}, x,\mu,a), \boldsymbol{\theta}' =(\bx',\boldsymbol{\mu}',x',\mu',a') \in (\R^d)^N \times \Pc_2((\R^d)^N) \times \R^d \times \Pc_2(\R^d) \times A$ and where $\bX =(X_{1},\ldots,X_N) \sim \boldsymbol{\mu}$ and $\bX' =(X'_1,\ldots, X'_N) \sim \boldsymbol{\mu}'$ and $X \sim \mu$ and $X' \sim \mu'$. Moreover, the function $g$ is also assumed to be $L$-convex, i.e.
\begin{align*}
    g(\bx',\boldsymbol{\mu}') - g(\bx,\boldsymbol{\mu}) -  \partial_{\bx} g(\bx,\boldsymbol{\mu}) \cdot (\bx' - \bx) -  \E \Big[ \partial_{\tilde{\bx}} \frac{\delta}{\delta \bm} g(\bx,\boldsymbol{\mu})(\bX) \cdot(\bX' - \bX) \Big]  \geq 0.
\end{align*}
\end{enumerate}
\end{Assumption}
\begin{Remark}\label{rmk : hamiltonian_L_convexity}
    Under the specific dynamics of the state process, it is straightforward to see that the Hamiltonian map $H$ satisfies the $L$-convexity property, i.e.
    \begin{align*}
    &H(t,\boldsymbol{\Theta}') - H(t,\boldsymbol{\Theta})
    - \partial_x H(t,\boldsymbol{\Theta}) \cdot (x'-x)
    - \E\big[
        \partial_{\tilde{x}} \frac{\delta}{\delta m} H(t,\boldsymbol{\Theta})(X)
        \cdot (X'-X)
    \big] \\
    &\qquad
    - \partial_{\bx} H(t,\boldsymbol{\Theta}) \cdot (\bx' - \bx)
    - \E\big[
        \partial_{\tilde{\bx}} \frac{\delta}{\delta \bm} H(t,\boldsymbol{\Theta})(\bX)
        \cdot (\bX' - \bX)
    \big]
    - \partial_a H(t,\boldsymbol{\Theta}) \cdot (a'-a)
    \geq \lambda |a'-a|^2 ,
\end{align*}
    where we used that $\partial_{a} H(t,\boldsymbol{\Theta}') - \partial_a H(t,\boldsymbol{\Theta}) = \partial_a f(t,\boldsymbol{\theta}') - \partial_a f(t,\boldsymbol{\theta})$
    for any $t \in [0,T]$, $\boldsymbol{\Theta}=(\bx,\boldsymbol{\mu},x,\mu,y,z,a), \boldsymbol{\Theta}'=(\bx',\boldsymbol{\mu}',x',\mu',y,z,a') \in (\R^d)^N \times \Pc_2((\R^d)^N) \times \R^d \times \Pc_2(\R^d) \times \R^d \times \R^{d \times n} \times A$ and where $(X_1,\ldots,X_N) \sim \boldsymbol{\mu}$, $(X'_1, \ldots, X'_N) \sim \boldsymbol{\mu}'$, $X \sim \mu$ and $X' \sim \mu'$. 
\end{Remark}
The Hamiltonian map $H$ takes the following form
\begin{equation}  \label{eq:Hamiltonian_map_def_L_convex}
\begin{split}
    H(t,\bx,\boldsymbol{\mu},x,\mu,y,z,a) &= \big[  b_0(t) + \sum_{j=1}^{\eta(t)} B_j(t) x_j + \sum_{j=1}^{\eta(t)} \bar{B}_j(t) \bar{\mu}_j+ C(t) x + \bar{C}(t) \bar{\mu} + D(t) a ] \cdot y  \\
    &\quad \; + \big[ \sigma_0(t) + \sum_{j=1}^{\eta(t)} E_j(t)x_j + \sum_{j=1}^{\eta(t)} \bar{E}_j(t) \bar{\mu}_j + F(t) x + \bar{F}(t) \bar{\mu} + G(t) a \big ] \cdot z  \\
    &\quad \; + f(t,\bx,\boldsymbol{\mu},x,\mu,a).  \end{split}
\end{equation}

\noindent It is clear that $H$ is a $\Pi$-non-anticipative map. Moreover, given $(t,\bx,\boldsymbol{\mu},x,\mu,y,z)$ $\in [0,T] \times (\R^d)^N \times \Pc_2((\R^d)^N) \times \R^d \times$ $\Pc_2(\R^d) \times \R^d \times \R^{d \times n}$, the map
$A \ni a \mapsto H(t,\bx,\boldsymbol{\mu},x,\mu,y,z,a)$ is strictly convex so that there exists a unique minimizer\linebreak $\hat{a}(t,\bx,\boldsymbol{\mu},x,\mu,y,z)$ $= \underset{a \in A}{\text{ arg min }} H(t,\bx,\boldsymbol{\mu},x,\mu,y,z,a)$. Moreover, following the same arguments as in  \cite[Lemma 6.18]{carmona2018probabilistic_I}, the map  $[0,T] \times (\R^d)^N \times \Pc_2((\R^d)^N)  \times \R^d \times \Pc_2(\R^d) \times \R^d \times \R^{d \times n} \ni (t,\bx,\boldsymbol{\mu},x,\mu,y,z) \mapsto \hat{a}(t,\bx,\boldsymbol{\mu},x,\mu,y,z) \in A$
is measurable, $\Pi$-non-anticipative, locally bounded, and Lipschitz continuous with respect to $(\bx,\boldsymbol{\mu},x,\mu,y,z)$, uniformly in $t \in [0,T]$.

Now, given the necessary and sufficient conditions from Theorem \ref{thm : necessary_condition_optimality} and Theorem \ref{thm : sufficient_condition_optimality}, we are looking to find a solution to the following FBSDE with jumps where we use the control $\hat{\alpha} =(\hat{\alpha}_t)_{0 \leq t \leq T}$ defined by $\hat{\alpha}_t = \hat{a}(t,\bX, \P_{\bX},X_t,\P_{X_t}, Y_t,Z_t)$ where $\hat{a}$ is the unique minimizer map of the Hamiltonian $H$ defined in \eqref{eq:Hamiltonian_map_def_L_convex}. 

\begin{align}\label{eq : dynamics_linear_cost_alpha_convex}
\left\{
\begin{array}{rcl}
\d X_t
&=&
\displaystyle
\Big[
b_0(t)
+ \sum_{j=1}^{\eta(t)} B_j(t) X_{t_j}
+ \sum_{j=1}^{\eta(t)} \bar{B}_j(t)\E[X_{t_j}]
+ C(t)X_t
+ \bar{C}(t)\E[X_t]
+ D(t)\hat{\alpha}_t
\Big]\d t
\\[0.2cm]
&&\displaystyle 
+
\Big[
\sigma_0(t)
+ \sum_{j=1}^{\eta(t)} E_j(t) X_{t_j}
+ \sum_{j=1}^{\eta(t)} \bar{E}_j(t)\E[X_{t_j}]
+ F(t)X_t
+ \bar{F}(t)\E[X_t]
+ G(t)\hat{\alpha}_t
\Big]\d W_t,
\\[0.3cm]
X_0
&=&
\xi,
\\[0.3cm]
\d Y_t
&=&
\displaystyle
-\Big[
\partial_x f(t,\bX,\P_{\bX},X_t,\P_{X_t},\hat{\alpha}_t)
+ C(t)^\top Y_t
+ F(t)^\top Z_t
\Big]\d t
+ Z_t\d W_t
\\[0.2cm]
&&\displaystyle
-\Big[
\tilde{\E}\big[
\partial_{\tilde{x}}\frac{\delta}{\delta m}
f(t,\tilde{\bX},\P_{\bX},\tilde{X}_t,\P_{X_t},\tilde{\hat{\alpha}}_t)(X_t)
\big]
+ \bar{C}(t)^\top \E[Y_t]
+ \bar{F}(t)^\top \E[Z_t]
\Big]\d t,
\\[0.2cm]
&&\displaystyle \hspace{3 cm}
t_{k-1}\le t<t_k,\qquad
k\in\llbracket 1,N-1\rrbracket,
\qquad
t_{N-1}\le t\le t_N,
\\[0.3cm]
Y_{t_k}-Y_{t_k^-}
&=&
\displaystyle
-\E\Big[
\int_{t_k}^{T}
\Big(
\partial_{x_k}f(t,\bX,\P_{\bX},X_t,\P_{X_t},\hat{\alpha}_t)
+ B_k(t)^\top Y_t
+ E_k(t)^\top Z_t
\Big)\d t
\,\Big|\,
\Fc_{t_k}
\Big]
\\[0.2cm]
&&\displaystyle \hspace{0.1 cm}
-\E\Big[
\int_{t_k}^{T}
\Big(
\tilde{\E}\big[
\partial_{\tilde{x}_k}\frac{\delta}{\delta \bm}
f(t,\tilde{\bX},\P_{\bX},\tilde{X}_t,\P_{X_t},\tilde{\hat{\alpha}}_t)(\bX)
\big]
+ \bar{B}_k(t)^\top \E[Y_t]
+ \bar{E}_k(t)^\top \E[Z_t]
\Big)\d t
\,\Big|\,
\Fc_{t_k}
\Big],
\\[0.2cm]
&&\displaystyle \hspace{0.1 cm}
-\E\Big[
\partial_{x_k}g(\bX,\P_{\bX})
+
\tilde{\E}\big[
\partial_{\tilde{x}_k}\frac{\delta}{\delta \bm}
g(\tilde{\bX},\P_{\bX})(\bX)
\big]
\,\Big|\,
\Fc_{t_k}
\Big], \qquad k \in \llbracket 1, N-1 \rrbracket,
\\[0.3cm]
Y_{t_N}
&=&
\displaystyle
\partial_{x_N}g(\bX,\P_{\bX})
+
\tilde{\E}\Big[
\partial_{\tilde{x}_N}\frac{\delta}{\delta \bm}
g(\tilde{\bX},\P_{\bX})(\bX)
\Big].
\end{array}
\right.
\end{align}
From now, the notation $(\boldsymbol{\Theta}_t)_{0 \leq t \leq T}$ will stand for $(\bX_{\eta(t)}, \P_{\bX_{\eta(t)}}, X_t, \P_{X_t}, Y_t, Z_t,\alpha_t)_{0 \leq t \leq T}$ with values in $(\R^d)^{\eta(t)} \times \Pc_2((\R^d)^{\eta(t)}) \times \R^d \times \Pc_2(\R^d) \times \R^d \times \R^{d \times n} \times A$. We will denote by $\mathfrak{S}$ the space of processes $\boldsymbol{\Theta} = (\boldsymbol{\Theta}_t)_{0 \leq t \leq T}$ such that $(\bX_{\eta(t)} , \P_{\bX_{\eta(t)}}, X_t, \P_{X_t}, Y_t,Z_t,\alpha_t)$ is $\F$-progressively measurable, $X=(X_t)_{0 \leq t \leq T}$ and $Y=(Y_t)_{0 \leq t \leq T}$ have càdlàg paths and
\begin{align*}
    \lVert \boldsymbol{\Theta} \rVert_{\mathfrak{S}} := \E \Big[ \underset{0 \leq t \leq T}{\text{ sup }} \big[ |X_t|^2 + |Y_t|^2 \big] + \int_{0}^{T} \big[ |Z_t|^2 + |\alpha_t|^2 \big] \d t  \Big]^{\frac{1}{2}} < \infty.
\end{align*}
Similarly, the notations $(\boldsymbol{\theta}_t)_{0 \leq t \leq T}$ will stand for processes $(\bX_{\eta(t)} , \P_{\bX_{\eta(t)}}, X_t, \P_{X_t}, \alpha_t)_{0 \leq t \leq T}$ with values in $(\R^d)^{\eta(t)} \times \Pc_2(( \R^d)^{\eta(t)}) \times \R^d \times \Pc_2(\R^d) \times A$. In the following, the processes $(\boldsymbol{\theta}_t)_{0 \leq t \leq T}$ will be the restriction of the processes $(\boldsymbol{\Theta}_t)_{0 \leq t \leq T} \in \mathfrak{S}$.
An input for \eqref{eq : dynamics_linear_cost_alpha_convex} will be a four-tuple $\Ic := \big( (\Ic_t^b, \Ic_t^{\sigma}, \Ic_t^f)_{0 \leq t \leq T}, (\Ic_{t_k}^g)_{1 \leq k \leq N} \big)$ where $(\Ic^b_t)_{0 \leq t \leq T}, (\Ic_t^{\sigma})_{0 \leq t \leq T}$ and $(\Ic_t^f)_{0 \leq t \leq T}$ are three square-integrable progressively measurable, processes with values respectively in $\R^d$, $\R^{d \times n}$ and $\R^d$ and $\Ic_{t_k}^g$ is an $\Fc_{t_k}$-measurable random variable with values in $\R^d$ for any $1 \leq k \leq N$. Such an input $\Ic$ is meant to be injected into the dynamics of \eqref{eq : dynamics_linear_cost_alpha_convex}, with $\Ic^b$ being plugged into the drift of the forward equation, $\Ic^{\sigma}$ into the volatility of the forward equation, $\Ic^f$ into the bounded variation term  of the backward process and $\Ic^g$ into the jump terms and  terminal condition of the backward process. The space of inputs is denoted by $\I$ and is endowed with the norm:
\begin{align*}
    \lVert \Ic \rVert_{\I} := \E \Big[ \int_{0}^{T} \big(| \Ic_t^b|^2 + | \Ic_t^{\sigma}|^2 + |\Ic_t^f|^2\big) \d t + \sum_{k=1}^{N} |\Ic_{t_k}^g|^2 \Big]^{\frac{1}{2}}.
\end{align*}
By misuse of notation, we shall denote $\phi(t,\bX, \P_{\bX}, X_t, \P_{X_t}, \alpha_t) = \phi(t,\boldsymbol{\theta}_t)$ and $H(t,\bX, \P_{\bX},X_t,\P_{X_t},Y_t,Z_t,\alpha_t) = H(t,\boldsymbol{\Theta}_t)$ for any $0 \leq t \leq T$ since the coefficients $\phi=(b,\sigma,f)$  of  the dynamics, the running cost and the Hamiltonian are $\Pi$-non-anticipative maps.
\begin{Definition}
    For any $\gamma \in [0,1], \xi \in L^2(\Omega;\Fc_0;\P;\R^d)$ and $\Ic \in \I$, the FBSDE
    {\footnotesize\begin{align*}
    \begin{cases}
        \d X_t  &= (\gamma b(t, \boldsymbol{\theta}_t) + \Ic_t^b \big) \d t + (\gamma \sigma(t,\boldsymbol{\theta}_t) + \Ic_t^{\sigma}) \d W_t, \\
        X_0 &= \xi, \\
        \d Y_t &= - \Big( \gamma \big( \partial_x H(t,\boldsymbol{\Theta}_t) + \tilde{\E} \big[ \partial_{\tilde{x}} \frac{\delta}{\delta m} H(t, \tilde{\boldsymbol{\Theta}}_t)(X_t) \big] \big)  + \Ic_t^{f} \Big) \d t + Z_t \d W_t, ~t_{k-1} \leq t < t_k, ~ k \in \llbracket 1, N-1 \rrbracket, \text{ and } t_{N-1} \leq t \leq t_N, \\
        Y_{t_k} - Y_{t_k^{-}} &=  - \bigg( \gamma \E \Big[ \int_{t_k}^{T} \Big( \partial_{x_k} H(t,\boldsymbol{\Theta}_t) + \tilde{\E} \big[ \partial_{\tilde{x}_k} \frac{\delta}{\delta \bm} H(t, \tilde{\boldsymbol{\Theta}}_t)(\bX) \big] \Big) \d t + \partial_{x_k} g(\bX,\P_{\bX}) + \tilde{\E} \big[ \partial_{\tilde{x}_k} \frac{\delta}{\delta \bm} g(\tilde{\bX}, \P_{\bX})(\bX) \big]  \Big| \Fc_{t_k} \Big] + \Ic_{t_k}^g \bigg), \\
        Y_{t_N} &= \gamma \Big( \partial_{x_N} g(\bX, \P_{\bX}) + \tilde{\E} \big[ \partial_{\tilde{x}_N} \frac{\delta}{\delta \bm} g(\tilde{\bX}, \P_{\bX})(\bX) \big] \Big) + \Ic_{t_N}^g,
    \end{cases}
    \end{align*}}
    with $\alpha_t = \hat{a}(t,\bX,\P_{\bX},X_t,$ $\P_{X_t},Y_t,Z_t)$, $0 \leq t \leq T$, is referred to as $\Ec(\gamma,\xi, \Ic)$. Whenever $(X_t,Y_t,Z_t)_{0 \leq t \leq T}$ is a solution, the full process
    $\boldsymbol{\Theta} := \big(\boldsymbol{\Theta}_t = (\bX_{\eta(t)}, \P_{\bX_{\eta(t)}}, X_t, \P_{X_t} , Y_t,Z_t, \alpha_t)\big)_{0 \leq t \leq  T} $ is referred to as the extended solution.
\end{Definition}

\begin{Definition}
    For any $\gamma \in [0,1]$, we say that $(\Sc_{\gamma})$ holds if for any $\xi \in L^2(\Omega;\Fc_0;\P;\R^d)$ and any $\Ic \in \I$, the FBSDE $\Ec(\gamma,\xi,\Ic)$ has a unique extended solution in $\mathfrak{S}$.
\end{Definition}
\begin{Lemma}\label{lemma : stability_solution}
    Let $\gamma \in [0,1]$ such that $(\Sc_{\gamma})$ holds. Then, there exists a constant $C$ independent of $\gamma$ such that for any $\xi,\xi' \in L^2(\Omega;\Fc_0,\P;\R^d)$ and $\Ic,\Ic' \in \I$,  the respective extended solutions $\boldsymbol{\Theta}$ and $\boldsymbol{\Theta}'$ of $\Ec(\gamma,\xi,\Ic)$ and $\Ec(\gamma,\xi', \Ic')$ satisfy
    \begin{align*}
        \lVert \boldsymbol{\Theta} - \boldsymbol{\Theta}' \rVert_{\mathfrak{S}}  \leq C \Big( \E \big[ | \xi - \xi'|^2 \big]^{\frac{1}{2}} + \lVert \Ic - \Ic' \rVert_{\I} \Big).
    \end{align*}
\end{Lemma}
\begin{proof}
    We use the same notation as introduced before, i.e. $$\boldsymbol{\Theta} = (\bX_{\eta(t)}, \P_{\bX_{\eta(t)}},X_t,\P_{X_t}, Y_t,Z_t,\alpha_t)_{0 \leq t \leq T} \text{ and } \boldsymbol{\theta}_t = (\bX_{\eta(t)}, \P_{\bX_{\eta(t)}}, X_t,\P_{X_t}, \alpha_t).$$ We have similarly to  computations in the proof of Theorem \ref{thm : sufficient_condition_optimality}
    \begin{align}\label{eq : intermediate_compute_1}
        \E \Big[\sum_{i=1}^{N-1} &- (X'_{t_i} - X_{t_i}) \cdot \Delta Y_{t_i} \Big] + \E \Big[ Y_{t_N} \cdot (X'_{t_N} - X_{t_N}) \Big] \\
        &=  \E \Big[ (\xi'-\xi) \cdot Y_0 \Big] - \gamma  \bigg( \E \Big[ \int_{0}^{T} \big( \partial_x H(t, \boldsymbol{\Theta}_t) + \tilde{\E} \big[ \partial_{\tilde{x}} \frac{\delta}{\delta m} H(t,\tilde{\boldsymbol{\Theta}}_t)(X_t) \big] \big) \cdot (X'_t - X_t) \d t \Big]  \\
        & \quad \quad \quad \quad - \E \Big[ \int_{0}^{T} ( b(t,\boldsymbol{\theta}'_t) - b(t, \boldsymbol{\theta}_t) ) \cdot Y_t + Z_t \cdot (\sigma(t,\boldsymbol{\theta}'_t) - \sigma(t, \boldsymbol{\theta}_t)  ) \d t \Big] \bigg)  \\ 
        &\quad \quad \quad \quad -  \bigg( \E \Big[ \int_{0}^{T} \Ic_t^f  \cdot (X_t' - X_t) + (\Ic_t^{b} - \Ic_t^{b'}) \cdot Y_t + Z_t \cdot (\Ic_t^{\sigma} - \Ic_t^{\sigma'}) \d t  \Big] \bigg)  \\ 
        &= T_0 - \gamma T_1 - T_2,  \end{align}
where we  denoted
\begin{align*}
\begin{cases}
    T_0 := \E \Big[ (\xi'- \xi) \cdot Y_0 \Big], \\
    T_1 :=   \E \bigg[ \int_{0}^{T} \Big( \big( \partial_x H(t, \boldsymbol{\Theta}_t) + \tilde{\E} \big[ \partial_{\tilde{x}} \frac{\delta}{\delta m} H(t,\tilde{\boldsymbol{\Theta}}_t)(X_t) \big] \big) \cdot (X'_t - X_t) - \big( b(t,\boldsymbol{\theta}'_t) - b(t, \boldsymbol{\theta}_t) \big) \cdot Y_t - Z_t \cdot \big(\sigma(t,\boldsymbol{\theta}'_t) - \sigma(t, \boldsymbol{\theta}_t)  \big) \Big) \d t \bigg] , \\
    T_2 :=  \E \Big[ \int_{0}^{T} \Ic_t^f  \cdot (X_t' - X_t) + (\Ic_t^{b} - \Ic_t^{b'}) \cdot Y_t + Z_t \cdot (\Ic_t^{\sigma} - \Ic_t^{\sigma'}) \d t  \Big]  . \\
\end{cases}
\end{align*}
Using the convexity assumption on $g$, we have
\begin{align}\label{eq : intermediate_compute_2}
  \E &\Big[ \sum_{i=1}^{N-1} - \Delta Y_{t_i} \cdot (X'_{t_i} - X_{t_i}) + Y_{t_N} \cdot (X'_{t_N} -X_{t_N})\Big]  \\&= \gamma  \E \Big[ \sum_{i=1}^{N}  \Big(\partial_{x_i} g(\bX, \P_{\bX}) + \tilde{\E} \big[ \partial_{\tilde{x}_i} \frac{\delta}{\delta \bm} g(\tilde{\bX},\P_{\bX})(\bX) \big] \Big) \cdot ( X'_{t_i} - X_{t_i}) \Big] + \E \Big[ \sum_{i=1}^{N} \Ic_{t_i}^g \cdot (X'_{t_i} - X_{t_i}) \Big] ,  \\
  &\quad +  \gamma \E \bigg[ \sum_{k=1}^{N-1} \int_{t_k}^{T} \Big( \partial_{x_k} H(t,\boldsymbol{\Theta}_t) + \tilde{\E} \big[ \partial_{\tilde{x}_k} \frac{\delta}{\delta \bm} H(t, \tilde{\boldsymbol{\Theta}}_t)(\bX) \big] \Big) \cdot (X'_{t_k} - X_{t_k}) \d t \bigg]  \\
  &\leq \gamma \E \Big[ g({\bX'}, \P_{{\bX'}}) - g(\bX, \P_{\bX}) \big] + \E \Big[ \sum_{i=1}^{N} \Ic_{t_i}^g \cdot (X'_{t_i} - X_{t_i}) \Big]  \\
  &\quad + \gamma \E \bigg[ \sum_{k=1}^{N-1} \int_{t_k}^{T} \Big( \partial_{x_k} H(t,\boldsymbol{\Theta}_t) + \tilde{\E} \big[ \partial_{\tilde{x}_k} \frac{\delta}{\delta \bm} H(t, \tilde{\boldsymbol{\Theta}}_t)(\bX) \big] \Big) \cdot (X'_{t_k} - X_{t_k}) \d t \bigg].  \end{align}
Now, combining \eqref{eq : intermediate_compute_1} and \eqref{eq : intermediate_compute_2} , we end up with
\begin{align*}
   \gamma \E \big[ g({\bX'},\P_{{\bX'}}) - g(\bX, \P_{\bX}) \big] &\geq  - \gamma \E \Big[ \int_{0}^{T} \big( \partial_{x} H(t,\boldsymbol{\Theta}_t) + \tilde{\E} \big[ \partial_{\tilde{x}} \frac{\delta}{\delta m} H(t, \tilde{\boldsymbol{\Theta}}_t)(X_t) \big] \big) \cdot (X'_t - X_t) \d t  \Big] \\
   & \quad -  \gamma \E \bigg[ \sum_{k=1}^{N-1} \int_{t_k}^{T} \Big( \partial_{x_k} H(t,\boldsymbol{\Theta}_t) + \tilde{\E} \big[ \partial_{\tilde{x}_k} \frac{\delta}{\delta \bm} H(t, \tilde{\boldsymbol{\Theta}}_t)(\bX) \big] \Big) \cdot  (X'_{t_k} - X_{t_k}) \d t \bigg] \\
   &\quad +  \gamma \E \Big[ \int_{0}^{T} \big(b(t,\boldsymbol{\theta}_t') - b(t,\boldsymbol{\theta}_t) \big) \cdot Y_t + Z_t \cdot \big(\sigma(t,\boldsymbol{\theta}'_t) - \sigma(t,\boldsymbol{\theta}_t) \big) \d t \Big] \\
   &\quad + T_0 - T_2 + \E \Big[ \sum_{i=1}^{N} \Ic_{t_i}^g \cdot (X_{t_i} - X'_{t_i}) \Big].
\end{align*}
Recalling the definition of $J(\alpha)$, and the $L$-convexity assumption of the Hamiltonian in Remark \ref{rmk : hamiltonian_L_convexity} combined with the optimality condition of $(\alpha_t)_{0 \leq t \leq T}$ and the strict convexity of $H$   and using similar arguments as one used in the proof of Theorem \ref{thm : sufficient_condition_optimality}, we have 
\begin{align}\label{eq : ineq_cost_functional_J}
    \gamma \big(J(\alpha') - J(\alpha) \big) \geq \gamma \lambda \E  \Big[ \int_{0}^{T} |\alpha_t - \alpha'_t|^2 \d t \Big] + T_0 - T_2 + \E \Big[ \sum_{i=1}^{N} \Ic_{t_i}^g \cdot (X_{t_i}  - X'_{t_i}) \Big].
\end{align}
Now, by reversing the roles of $\alpha$ and $\alpha'$ in \eqref{eq : ineq_cost_functional_J} and by denoting $T_0'$ and $T_2'$ the associated terms and summing the inequalities, we end up with
\begin{align*}
    2 \gamma \lambda \E \Big[ \int_{0}^{T} |\alpha_t - \alpha'_t|^2 \d t \Big] +  T_0 + T_0'  - (T_2 + T_2') + \E \Big[ \sum_{i=1}^{N} ( \Ic_{t_i}^g - \Ic_{t_i}^{g'}) \cdot (X_{t_i} - X'_{t_i}) \Big] \leq 0,
\end{align*}
and where the sums $T_0 + T_0'$ and $T_2 + T_2'$ read
\begin{align*}
\begin{cases}
    T_2 +  T_2' &= \E \Big[ \int_{0}^{T} - (\Ic_t^f - \Ic_t^{f'}) \cdot (X_t - X'_t) + (\Ic_t^b - \Ic_t^{b'}) \cdot (Y_t - Y'_t) + (\Ic_t^{\sigma} - \Ic_t^{\sigma'}) \cdot (Z_t - Z'_t) \d t \Big], \\
    T_0 + T_0' &= - \E \Big[ (\xi-\xi') \cdot (Y_0 - Y_0') \Big].
\end{cases}
\end{align*}
Now, using Young's inequality, there exists a constant $C$ independent of $\gamma$ such that for any $\epsilon > 0$
\begin{align*}
    \gamma  \E \Big[ \int_{0}^{T} |\alpha_t - \alpha'_t|^2 \d t \Big] \leq \epsilon \lVert \boldsymbol{\Theta} - \boldsymbol{\Theta}' \rVert^2_{\mathfrak{S}} + \frac{C}{\epsilon} \big( \E\big[|\xi - \xi'|^2 \big] + \lVert \Ic - \Ic' \rVert^2_{\I} \big).
\end{align*}
\noindent \textbf{Forward estimate.} Applying It\^o's formula to $|X_t - X'_t|^2$, using the Lipschitz property of $b$ and $\sigma$ (Assumption \ref{assumption:solvability_LQ}) and Young's inequality, we obtain for any $t \in [0,T]$
\begin{align*}
    |X_t - X'_t|^2 &\leq |\xi - \xi'|^2 + C \int_{0}^{t} \Big( \gamma^2 |\bX_{\eta(s)} - \bX'_{\eta(s)}|^2 + \gamma^2 |X_s - X'_s|^2 + \gamma^2 |\alpha_s - \alpha'_s|^2  + |\Ic_s^b - \Ic_s^{b'}|^2 + |\Ic_s^{\sigma} - \Ic_s^{\sigma'}|^2 \Big) \d s \\
    &\quad + 2 \int_{0}^{t} \big( X_s - X'_s \big) \cdot \Big( \sigma(s,\boldsymbol{\theta}_s) - \sigma(s,\boldsymbol{\theta}'_s) + \Ic_s^{\sigma} - \Ic_s^{\sigma'} \Big) \d W_s.
\end{align*}
Taking the supremum over $[0,t]$, using Doob's maximal and BDG inequalities on the stochastic integral term, then Young's inequality to absorb the resulting $\E\big[\sup_{0\leq s \leq t}|X_s - X'_s|^2\big]$ on the left-hand side, and the bound $$\E\big[|\bX_{\eta(s)} - \bX'_{\eta(s)}|^2\big] \leq N \E\big[\sup_{0 \leq r \leq s} |X_r - X'_r|^2 \big],$$ we get using $\gamma^2 \leq \gamma$ since $\gamma \in [0,1]$, and
\begin{align*}
    \E \Big[ \underset{0 \leq s \leq t}{\text{sup }} |X_s - X'_s|^2 \Big] &\leq \E \big[ |\xi - \xi'|^2 \big] + C \int_{0}^{t} \E \Big[ \underset{0 \leq r \leq s}{\text{sup}} |X_r - X'_r|^2 \Big] \d s + C \gamma \E \Big[ \int_{0}^{T} |\alpha_s - \alpha'_s|^2 \d s \Big] + C \lVert \Ic - \Ic' \rVert^2_{\I}.
\end{align*}
Grönwall's inequality then yields
\begin{align}\label{eq : forward_stability}
    \E \Big[ \underset{0 \leq t \leq T}{\text{sup }} |X_t - X'_t|^2 \Big] \leq  C \bigg( \E \big[ |\xi - \xi'|^2 \big] + \gamma \E \Big[ \int_{0}^{T} | \alpha_t - \alpha'_t|^2  \d t \Big] + \lVert \Ic - \Ic' \rVert^2_{\I} \bigg),
\end{align}
where $C$ can be chosen independent of $\gamma$ since $0 \leq \gamma \leq 1$.

\noindent \textbf{Backward estimate.} Applying It\^o's formula to $|Y_t - Y'_t|^2$ separately on each subinterval $[t_i,t_{i+1})$ and on $[t_{N-1},t_N]$, summing the jump contributions at the deterministic times $(t_k)_{1 \leq k \leq N-1}$ and using the Lipschitz property of the driver (through the $L$-convex Hamiltonian, see Remark \ref{rmk : hamiltonian_L_convexity}), Young's inequality, Doob's maximal inequality applied to the martingale term $\int_0^{\cdot} (Y_s - Y'_s) \cdot (Z_s - Z'_s) \d W_s$, and the BDG inequality, we deduce, by the backward recursion over the subintervals detailed below, the stability estimate
\begin{align}\label{eq : backward_stability}
    &\E \Big[ \underset{0 \leq t \leq T}{\text{sup }} | Y_t - Y'_t|^2 + \int_{0}^{T} | Z_t - Z'_t|^2 \d t \Big] \leq C \bigg(\gamma \E \Big[ \underset{0 \leq t \leq T}{\text{ sup }} |X_t - X'_t|^2  + \int_{0}^{T} |\alpha_t - \alpha'_t|^2 \d t \Big] + \lVert \Ic - \Ic' \rVert^2_{\I} \bigg),  \end{align}
for a constant $C$ independent of $\gamma$. The jump terms are controlled via Jensen's inequality applied to the conditional expectations together with the Lipschitz property of $\partial_{\bx} g$ and of $\partial_{\bx} H + \tilde{\E}[\partial_{\tilde{\bx}}\frac{\delta}{\delta \bm} H(\tilde{\boldsymbol{\Theta}}_t)(\bX)]$, see  Assumption \ref{assumption:solvability_LQ} (2). Writing $H = b\cdot y + \sigma \cdot z + f$ and using the boundedness of $\partial_{x_k} b$ and $\partial_{x_k} \sigma$ together with the Lipschitz property of $\partial_{x_k}(b,\sigma,f)$ in $(\bx,\boldsymbol{\mu},x,\mu,a)$, we obtain for each $1 \leq k \leq N$
\begin{align*}
    \E \big[|\Delta Y_{t_k} - \Delta Y'_{t_k}|^2 \big] &\leq C \Big( \gamma \E\big[ |\bX - \bX'|^2\big] + \gamma \E\Big[ \int_{t_k}^{T} \big( |\bX_{\eta(t)} - \bX'_{\eta(t)}|^2 + |X_t - X'_t|^2 \\
    &\qquad + |Y_t - Y'_t|^2 + |Z_t - Z'_t|^2 + |\alpha_t - \alpha'_t|^2 \big) \d t \Big] + \E\big[|\Ic_{t_k}^g - \Ic_{t_k}^{g'}|^2\big] \Big),
\end{align*}
and similarly for the terminal condition. Note that the jump at $t_k$ is a conditional expectation of the solution on $[t_k,T]$, so \eqref{eq : backward_stability} does not follow from a stability estimate with exogenous jump data; it is instead obtained by a backward recursion over the subintervals. On $[t_{N-1}, t_N]$, the terminal data and the Lipschitz-in-$(y,z)$ driver give the standard bound for $\E \big[ \sup_{t_{N-1} \leq t \leq t_N} |Y_t - Y'_t|^2 + \int_{t_{N-1}}^{t_N} |Z_t - Z'_t|^2 \d t \big]$ in terms of $X - X'$, $\alpha - \alpha'$ and $\Ic - \Ic'$, as in \cite[Lemma 6.23]{carmona_delarue_2015}. The jump bound above at $t_{N-1}$ then involves only quantities that are already estimated, and provides the terminal data $Y_{t_{N-1}^-} - Y'_{t_{N-1}^-}$ for the interval $[t_{N-2}, t_{N-1})$. Iterating backward over the $N$ subintervals, whose number is fixed, yields \eqref{eq : backward_stability} with a constant $C = C(N)$ independent of $\gamma$.

\noindent \textbf{Conclusion.} Combining \eqref{eq : forward_stability} and \eqref{eq : backward_stability} with \eqref{eq : forward_stability} inserted in the right-hand side of \eqref{eq : backward_stability}, and recalling the definition of $\lVert \cdot \rVert_{\mathfrak{S}}$, we obtain
\begin{align}\label{eq : stability_Theta_intermediate}
    \lVert \boldsymbol{\Theta} - \boldsymbol{\Theta}' \rVert^2_{\mathfrak{S}} \leq C \bigg( \E \big[|\xi - \xi'|^2\big] + \gamma \E \Big[ \int_{0}^{T} |\alpha_t - \alpha'_t|^2 \d t \Big] + \lVert \Ic - \Ic' \rVert^2_{\I} \bigg).
\end{align}
Combining \eqref{eq : stability_Theta_intermediate} with the Young-inequality bound derived above from the $L$-convexity of the Hamiltonian, namely
\begin{align*}
    \gamma  \E \Big[ \int_{0}^{T} |\alpha_t - \alpha'_t|^2 \d t \Big] \leq \epsilon \lVert \boldsymbol{\Theta} - \boldsymbol{\Theta}' \rVert^2_{\mathfrak{S}} + \frac{C}{\epsilon} \big( \E[|\xi - \xi'|^2] + \lVert \Ic - \Ic' \rVert^2_{\I} \big),
\end{align*}
we obtain
\begin{align*}
    \big( 1 - C \epsilon \big) \lVert \boldsymbol{\Theta} - \boldsymbol{\Theta}' \rVert^2_{\mathfrak{S}} \leq \frac{C}{\epsilon} \Big( \E \big[ |\xi - \xi'|^2 \big] + \lVert \Ic - \Ic' \rVert^2_{\I} \Big).
\end{align*}
Choosing $\epsilon > 0$ such that $C \epsilon < 1$, we conclude that
\begin{align*}
    \lVert \boldsymbol{\Theta} - \boldsymbol{\Theta}' \rVert^2_{\mathfrak{S}} \leq C \Big( \E \big[ |\xi - \xi'|^2 \big] + \lVert \Ic - \Ic' \rVert^2_{\I} \Big),
\end{align*}
for a constant $C$ independent of $\gamma \in [0,1]$, which is the desired stability estimate.
\end{proof}
\begin{Lemma}\label{lemma : Sgammarecursive}
    There exists $\delta_0 > 0$ such that if $(\Sc_{\gamma})$ holds for some $\gamma \in [0,1)$, then $(\Sc_{\gamma + \eta})$ also holds for any $\eta \in (0,\delta_0]$ satisfying $\gamma + \eta \leq 1$.
\end{Lemma}
\begin{proof}
    The proof is similar in spirit to the continuation-method argument of \cite{carmona_delarue_2015}. Indeed, we define the map $\Phi$ from $\mathfrak{S}$ into $\mathfrak{S}$ as follows. Given a process $\boldsymbol{\Theta} \in \mathfrak{S}$, we denote by $\boldsymbol{\Theta}'$ the extended solution to the FBSDE $\Ec(\gamma,\xi,\Ic')$ with
    \begin{align*}
    \begin{cases}
          \Ic_t^{b,'} &:= \eta b(t,\boldsymbol{\theta}_t) + \Ic_t^{b}, \\
    \Ic_t^{\sigma,'} :&= \eta \sigma(t,\boldsymbol{\theta}_t)+ \Ic_t^{\sigma}, \\
    \Ic_t^{f,'} &:= \eta \partial_x H(t,\boldsymbol{\Theta}_t ) + \eta  \tilde{\E} \Big[ \partial_{\tilde{x}} \frac{\delta}{\delta m} H(t,\tilde{\boldsymbol{\Theta}}_t)(X_t) \Big] + \Ic_t^f, \\
    \Ic_{t_k}^{g,'} &:= \E \bigg[ \int_{t_k}^{T} \eta \Big( \partial_{x_k} H(t,\boldsymbol{\Theta}_t) + \tilde{\E} \big[ \partial_{\tilde{x}_k} \frac{\delta}{\delta \bm} H(t,\tilde{\boldsymbol{\Theta}}_t)(\bX) \big] \Big) \d t +  \eta \Big(\partial_{x_k} g(\bX,\P_{\bX}) +  \tilde{\E} \Big[ \partial_{\tilde{x}_k} \frac{\delta}{\delta \bm}g(\tilde{\bX},\P_{\bX})(\bX)  \Big] \Big) \bigg| \Fc_{t_k} \bigg] + \Ic_{t_k}^{g}.
    \end{cases}
    \end{align*}
    It is clear that $\Ic' \in \I$ and by assumption $\boldsymbol{\Theta}' \in \mathfrak{S}$ and is uniquely defined. Therefore, the map $\Phi$ is a well defined map from the complete metric space $(\mathfrak{S}, \lVert \cdot \rVert_{\mathfrak{S}})$ into itself. Moreover, by construction, it is clear that $\boldsymbol{\Theta} \in \mathfrak{S}$ is a fixed point for $\Phi$ if and only if $\boldsymbol{\Theta}$ is an extended solution to $\Ec(\gamma + \eta, \xi,\Ic)$. Now, the proof is completed by noticing that $\Phi$ is a contraction for $\eta$ small enough: indeed, by Lemma \ref{lemma : stability_solution},
    \begin{align*}
        \lVert \Phi(\boldsymbol{\Theta}') - \Phi(\boldsymbol{\Theta}) \rVert_{\mathfrak{S}} \leq C \eta \lVert \boldsymbol{\Theta}' - \boldsymbol{\Theta} \rVert_{\mathfrak{S}}.
    \end{align*}
    Hence, for $\eta \leq \delta_0 := 1/(2C)$, the map $\Phi$ is a strict contraction on the complete metric space $(\mathfrak{S}, \lVert \cdot \rVert_{\mathfrak{S}})$; its unique fixed point is the unique extended solution to $\Ec(\gamma + \eta, \xi, \Ic)$, so that $(\Sc_{\gamma + \eta})$ holds.
\end{proof}
\begin{Theorem}
    Under Assumption \ref{assumption:solvability_LQ} and for any initial condition $\xi \in L^2(\Omega;\Fc_0;\P;\R^d)$, the FBSDE \eqref{eq : dynamics_linear_cost_alpha_convex} is uniquely solvable.
\end{Theorem}
\begin{proof}
    The proof is immediate and follows from Lemma \ref{lemma : Sgammarecursive} since $\delta_0$ is independent of $\gamma$ and after noticing that $(\Sc_0)$ clearly holds and since, by iterating Lemma \ref{lemma : Sgammarecursive} finitely many times, we can choose a partition $0 = \gamma_0 < \gamma_1 < \ldots < \gamma_n = 1$ with $\gamma_{k+1} - \gamma_k \leq \delta_0$ for any $0 \leq k \leq n-1$, which yields that $(\Sc_1)$ holds. Taking $\Ic \equiv 0$, the FBSDE \eqref{eq : dynamics_linear_cost_alpha_convex} is then uniquely solvable.
\end{proof}


\paragraph*{Case of linear quadratic control problems.}

We specify the class of linear dynamics and $L$-convexity for $f$ and $g$ for the subclass of quadratic cost functional. The dynamics of the system is hence given by 
\begin{align*}
\begin{cases}
    \d X_t &=  \Big[ \beta + \sum_{j=1}^{\eta(t)} B_j X_{t_j} + \sum_{j=1}^{\eta(t)} \bar{B}_j \E[ X_{t_j}] + C X_t + \bar{C} \E[X_t] + D \alpha_t  \Big] \d t +  \sigma \d W_t , \\
    X_0 &= \xi,
\end{cases}
\end{align*}
where $\beta \in \R^d$, $ B_1, \bar{B}_1, \ldots, B_N, \bar{B}_N \in \R^{d \times d}$, $C,\bar{C} \in \R^{d \times d}$, $D \in \R^{d \times m}$ , and $\sigma \in \R^{d \times n}$. The associated cost functional takes the form
\begin{align*}
    J(\alpha) = &\E \bigg[ \int_{0}^{T} \Big(\ P_{\eta(t)} \cdot ( \bX_{\eta(t)}^{\top} \bX_{\eta(t)})  + \bar{P}_{\eta(t)} \cdot  (\E[\bX_{\eta(t)}]^{\top} \E[ \bX_{\eta(t)}]) + X_t \cdot Q X_t   +\E[X_t] \cdot  \bar{Q} \E[X_t]   \\
    & \qquad \quad + \tfrac{1}{2}  \alpha_t \cdot R \alpha_t  \Big)\d t  +  G \cdot  (\bX^{\top} \bX)  + \bar{G} \cdot (\E[\bX]^{\top} \E[\bX] )  \bigg],
\end{align*}
where $\bX$ is viewed as the $d \times N$ matrix with $i$-th column $X_{t_i}$ and $\bX_{\eta(t)}$ as its $d \times \eta(t)$ truncation, $P, \bar{P}, G, \bar{G} \in \S^N_{+}$ are symmetric positive semi-definite (with $P_{\eta(t)}, \bar{P}_{\eta(t)}$ their leading $\eta(t) \times \eta(t)$ blocks), $Q, \bar{Q} \in \S^d_{+}$, and $R \in \S^m_{>+}$  is positive definite.
\begin{Remark} Note that, with this convention, the quantity $\E[\bX^{\top} \bX]$ is the classical auto-correlation matrix associated to the time series $(X_{t_1},\ldots,X_{t_N})$, and $\E[ \bX_{\eta(t)}^{\top} \bX_{\eta(t)}]$ to that of the truncated series, for any $0 \leq t \leq T$.
\end{Remark}
This framework falls in our setting, with control set $A = \R^m$, through the coefficients
\begin{align*}
\begin{cases}
    b(t,\bx,\boldsymbol{\mu},x,\mu,a) &= \beta + \sum_{j=1}^{\eta(t)} B_j x_j + \sum_{j=1}^{\eta(t)} \bar{B}_j \bar{\mu}_j + Cx + \bar{C} \bar{\mu} + Da , \\
    \sigma(t,\bx,\boldsymbol{\mu},x,\mu,a) &= \sigma, \\
    f(t,\bx,\boldsymbol{\mu},x,\mu,a) &= P_{\eta(t)} \cdot( \bx_{\eta(t)}^{\top} \bx_{\eta(t)}) +  \bar{P}_{\eta(t)} \cdot (\bar{\boldsymbol{\mu}}_{\eta(t)}^{\top} \bar{\boldsymbol{\mu}}_{\eta(t)})  + x \cdot Qx + \bar{\mu} \cdot \bar{Q} \bar{\mu} + \tfrac{1}{2}  a \cdot  Ra , \\
    g(\bx,\boldsymbol{\mu}) &=  G \cdot ( \bx^{\top} \bx ) +  \bar{G} \cdot ( \bar{\boldsymbol{\mu}}^{\top} \bar{\boldsymbol{\mu}}  ).
\end{cases}
\end{align*}
We then  compute
\begin{align*}
\begin{cases}
        \partial_{x} H(t,\bx, \boldsymbol{\mu}, x ,\mu,y,z,a) &= C^{\top} y + 2Q x , \\
        \partial_{\tilde{x}} \frac{\delta}{\delta m} H(t,\bx,\boldsymbol{\mu},x,\mu,y,z,a)(\tilde{x}) &= \bar{C}^{\top} y + 2\bar{Q} \bar{\mu}, \\
        \partial_{x_k}  H(t,\bx,\boldsymbol{\mu},x,\mu,y,z,a) &= B_k^{\top} y +2\sum_{l=1}^{\eta(t)} P_{lk}  \bx_l , \\
        \partial_{\tilde{x}_k} \frac{\delta}{\delta \bm} H(t,\bx, \boldsymbol{\mu},x,\mu,y,z,a)(\tilde{\bx}_{\eta(t)}) &= \bar{B}_k^{\top} y + 2\sum_{l=1}^{\eta(t)} \bar{P}_{lk}  \bar{\mu}_l, \\
        \partial_{x_k} g(\bx, \boldsymbol{\mu}) &=  2\sum_{l=1}^{N}  G_{lk} \bx_l  \in \R^d , \\
        \partial_{\tilde{x}_k} \frac{\delta}{\delta \bm} g(\bx,\boldsymbol{\mu})( \tilde{\bx}) &=  2\sum_{l=1}^{N}  \bar{G}_{lk} \bar{\boldsymbol{\mu}}_l \in \R^d.
\end{cases}
\end{align*}
Following Theorem \ref{thm : necessary_condition_optimality} and since the Hamiltonian map $H$ is clearly strictly convex in $a$ with $R \in \S^m_{>+}$, it admits a unique minimizer, which suggests looking for an optimal control candidate $\alpha$ defined as follows
\begin{align*}
    \alpha_t = - R^{-1}D^{\top} Y_t \in \R^m, \quad 0 \leq t \leq T.
\end{align*}

Therefore, the FBSDE \eqref{eq : dynamics_linear_cost_alpha_convex} reads as
{\footnotesize\begin{align}\label{eq : FBSDE_LQ}
\begin{cases}
    \d X_t &= \Big[\beta + \Sum_{j=1}^{\eta(t)} B_j X_{t_j} + \Sum_{j=1}^{\eta(t)} \bar{B}_j \E[X_{t_j}]  +C X_t + \bar{C} \E[X_t] - DR^{-1} D^{\top} Y_t \Big] \d t + \sigma \d W_t, \\
    X_0 &= \xi, \\
    \d Y_t &= - \Big[C^{\top} Y_t + 2Q X_t   + \bar{C}^{\top} \E[Y_t] + 2 \bar{Q} \E[X_t] \Big ] \d t + Z_t \d W_t, \\
    Y_{t_k} - Y_{t_k^-} &= -  \E \Big[ \int_{t_k}^{T}  \Big(B_k^{\top} Y_t + \bar{B}_k^{\top} \E[Y_t] + 2 \Sum_{l=1}^{\eta(t)} P_{lk}  X_{t_l}  + 2 \Sum_{l=1}^{\eta(t)} \bar{P}_{lk} \E[X_{t_l}] \Big)  \d t + 2 \Sum_{l=1}^{N} G_{lk} X_{t_l} + 2 \Sum_{l=1}^{N} \bar{G}_{lk} \E[X_{t_l}] \Big| \Fc_{t_k} \Big]   , \\
    Y_{t_N} &=   2\Sum_{l=1}^{N} G_{lN} X_{t_l} + 2\Sum_{l=1}^{N} \bar{G}_{lN} \E[X_{t_l}]. 
\end{cases}
\end{align}}
Due to the structure of the FBSDE above, we make the following ansatz for $Y$ as follows
\begin{align}\label{eq : LQ_ansatz}
    Y_t = \sum_{i=1}^{\eta(t)} \Lambda_i(t) X_{t_i}+ \sum_{i=1}^{\eta(t)} \bar{\Lambda}_{i}(t) \E[ X_{t_i}] + \Gamma(t) X_t + \bar{\Gamma}(t) \E[X_t] + \chi(t), \quad 0 \leq t \leq T,
\end{align}
where $(\Lambda_i, \bar{\Lambda}_i) : [t_i, T] \to \R^{d \times d}$, $(\Gamma, \bar{\Gamma}) : [0,T] \to \R^{d \times d}$ and $\chi : [0,T] \to \R^d$ for any $1 \leq i \leq N-1$, and $\Lambda_N \equiv \bar{\Lambda}_N \equiv 0$. In order to find the system characterizing the coefficients of the adjoint process, we will proceed in two steps. We first characterize the Ricatti equation over each subinterval $[t_{k-1},t_{k})$ for any $1 \leq k \leq N-1 $ and over $[t_{N-1},t_N]$ and we carefully look at the contributions due to the jumps in \eqref{eq : FBSDE_LQ} . Applying Itô's formula on each subinterval $t \in [t_{k-1}, t_k)$ for $k \in \llbracket 1, N-1 \rrbracket$ and over $t \in [t_{N-1},  t_N]$, we have denoting $\Sigma := D R^{-1} D^{\top}$
\begin{align*}
    \d Y_t = \,& \Big\{ \sum_{i=1}^{\eta(t)}\dot\Lambda_{i}(t) X_{t_i} + \sum_{i=1}^{\eta(t)} \dot{\bar\Lambda}_{i}(t) \E[X_{t_i}] + \dot\Gamma(t) X_t + \dot{\bar\Gamma}(t) \E[X_t] + \dot\chi(t) + (\Gamma(t) + \bar\Gamma(t))\beta \\
    &\quad + \Gamma(t) \Big( \sum_{i=1}^{N-1} B_i X_{t_i} + \sum_{i=1}^{N-1} \bar B_i \E[X_{t_i}] + C X_t + \bar C \E[X_t] - \Sigma Y_t\Big) \\
    &\quad +  \bar\Gamma(t) \Big( \sum_{i=1}^{N-1}(B_i + \bar B_i)\E[X_{t_i}] + (C + \bar C)\E[X_t] - \Sigma \E[Y_t] \Big) \Big\} \d t + \Gamma(t) \sigma \d W_t. 
\end{align*}
Hence, matching the terms $(X_{t_i})_{  1 \leq i \leq \eta(t)}, (\E[X_{t_i}])_{1 \leq i \leq \eta(t)}, (X_t) , \E[X_t]$ and the constant term in the continuous part of \eqref{eq : FBSDE_LQ}  yields the following triangular Riccati system 
\begin{align}\label{eq : Riccati_system}
\begin{cases}
    \dot{\Gamma}(t) + C^{\top} \Gamma(t) + \Gamma(t) C - \Gamma(t) \Sigma  \Gamma(t)  + 2Q &= 0, \\
    \dot{\bar{\Gamma}}(t) + \bar{C}^{\top} \bar{\Gamma}(t)+ \Gamma(t) \bar{C} - \Gamma(t) \Sigma \bar{\Gamma}(t) + \bar{\Gamma}(t) (C+ \bar{C}) - \bar{\Gamma}(t) \Sigma (\Gamma(t) + \bar{\Gamma}(t)) + \bar{C}^{\top}(\Gamma(t) + \bar{\Gamma}(t)) + 2 \bar{Q} &=0 ,\\
\dot{\Lambda}_{i}(t) + \Gamma(t) B_i + C^{\top} \Lambda_{i}(t) - \Gamma_t \Sigma \Lambda_{i}(t) &=0 ,\\
\dot{\bar{\Lambda}}_{i}(t) + \Gamma(t) \bar{B}_i - \Gamma(t) \Sigma \bar{\Lambda}_{i}(t) + \bar{\Gamma}(t) (B_i + \bar{B}_i) - \bar{\Gamma}(t) \Sigma (\Lambda_{i}(t) + \bar{\Lambda}_{i}(t)) + C^{\top} \bar{\Lambda}_{i}(t) + \bar{C}^{\top} (\Lambda_{i}(t) + \bar{\Lambda}_{i}(t))  &=0,\\
\dot{\chi}(t) + (\Gamma(t) + \bar{\Gamma}(t)) \beta - (\Gamma(t) + \bar{\Gamma}(t))  \Sigma \chi(t) + (C+ \bar{C})^{\top}   \chi(t) &=0,
\end{cases}
\end{align}
for $1 \leq i <k$ when $t \in [t_{k-1},t_k)$ whenever $k \in \llbracket 1, N-1 \rrbracket$ and for $1 \leq i < N$ when $t \in [t_{N-1},t_N]$,
with the following terminal conditions
\begin{align}\label{eq : terminal_condition}
\begin{cases}
    \Gamma(T) &= 2 G_{NN} I_d, \\
    \bar{\Gamma}(T) &=2 \bar{G}_{NN} I_d, \\
    \Lambda_{i}(T) &= 2 G_{iN} I_d, \quad 1 \leq i \leq N-1 \\
    \bar{\Lambda}_{i}(T) &= 2 \bar{G}_{iN} I_d, \quad 1 \leq i \leq N-1, \\
    \chi(T) &=0.
\end{cases}
\end{align}
In order to fully characterize the system, it is now enough to characterize the terminal conditions $t \uparrow t_k$ for any $1 \leq k \leq N-1$  for the coefficients of the adjoint process in \eqref{eq : FBSDE_LQ}. To do so, we shall rely on the following stability Lemma.
\begin{Lemma}\label{lemma : stability}
    For any $k \in \llbracket 1, N-1 \rrbracket$ and for any $t \in [t_k,T]$, it holds 
    \begin{align}\label{eq : conditional_expectation}
        \E[X_t|\Fc_{t_k}] = \sum_{j=1}^{k} \Bc_{j,k}(t) X_{t_j} + \sum_{j=1}^{k} \bar{\Bc}_{j,k}(t) \E[X_{t_j}] + \rho_k(t), \quad \P-\text{a.s},
    \end{align}
for some analytic  deterministic maps  $(\Bc_{j,k}, \bar{\Bc}_{j,k})_{1 \leq j \leq k,}$ and $\rho_k$ defined from $[t_k,T]$ into respectively $\R^{d \times d}, \R^{d \times d}$ and $\R^d$.
\end{Lemma}
\begin{proof} 
Fix $k$ $\in$  $\llbracket 1, N-1 \rrbracket$. 
From the ansatz form of $Y$ in \eqref{eq : LQ_ansatz}, we see that the dynamics of $X$ can be rewritten over each interval $[t_i, t_{i+1})$, $k\leq i\leq N-1$, as 
\begin{align}\label{eq : SDE_sub_interval}
    \d X_t = [A(t) X_t + \bar{A}(t) \E[X_t] + \sum_{j=1}^{i} F_{j}(t) X_{t_j} + \sum_{j=1}^{i} \bar{F}_j(t) \E[X_{t_j}] + \kappa(t) ] \d t + \sigma \d W_t, \quad t_i \leq t < t_{i+1},
\end{align}
for some continuous deterministic maps $(A, \bar{A},(F_{j})_{1 \leq j \leq i}, (\bar{F}_j)_{1 \leq j \leq i}, \kappa)$ and where we notice that the SDE can be extended to $t_{i+1}$ due to the drift term beeing absolutely continuous with respect to the Lebesgue measure.  Taking the conditional expectation with respect to $\Fc_{t_i}$ in \eqref{eq : SDE_sub_interval} yields that
\begin{align}
    \d \E[X_t|\Fc_{t_i}]= \big(A(t) \E[X_t|\Fc_{t_i}] + \bar{A}(t)\E[X_t] +\sum_{j=1}^{i} F_{j}(t) X_{t_j} + \sum_{j=1}^{i} \bar{F}_j(t) \E[X_{t_j}] +  \kappa(t)  \big) \d t , \quad \P-\text{a.s},
\end{align}
with $\E[X_{t_{\red{i}}}|\Fc_{t_i}] = X_{t_i}$. Then, from standard ODE analysis (see, eg \cite[Theorem 4.1]{sideris2013ordinary}) applied to the map $t \mapsto \E[X_t|\Fc_{t_i}](\omega)$ for $\P(\d \omega)-\text{a.s}$, it follows that the unique solution of the map  is given by
\begin{footnotesize}
\begin{align}\label{eq : conditional_expectation_form}
    \E[X_t|\Fc_{t_i}] = \Phi(t,t_i) X_{t_i} + \int_{t_i}^{t} \Phi(t,s)\big(\sum_{j=1}^i F_{j}(s) X_{t_j} + \bar{A}(s) \E[X_s] + \sum_{j=1}^{i} \bar{F}_j(s) \E[X_{t_j}] + \kappa(s) \big) \d s, \quad \P-\text{a.s}, \quad t_i \leq t \leq t_{i+1}.
\end{align}
\end{footnotesize}
where $\Phi$ is solution to
\begin{align}
\begin{cases}
    \partial_t \Phi(t,s) &= A(t) \Phi(t,s), \\
    \Phi(s,s) &= I_d,
\end{cases}
\end{align}
for any $t_i \leq s \leq t \leq t_{i+1}$. 
Finally, by concatenating \eqref{eq : conditional_expectation_form}  on each interval $[t_i,t_{i+1}]$ for $k \leq i \leq N-1$, 
and  using the tower property, we obtain the relation \eqref{eq : conditional_expectation} with coefficients $(\Bc_{j,k}, \bar{\Bc}_{j,k})_{1 \leq j \leq k}$, $\rho_k$ explicitly given as
\begin{align}
\begin{cases}
\Bc_{j,k}(t)
&=
\Phi(t,t_i) \Bc_{j,k}(t_i)
+
\displaystyle\int_{t_i}^{t}
\Phi(t,s)
\left(
F_j(s)
+
\displaystyle\sum_{\ell=k+1}^{i}
F_\ell(s) \Bc_{j,k}(t_\ell)
\right)
\d s,
\qquad t_i \leq t \leq t_{i+1},
\\[2mm]
\Bc_{j,k}(t_k)
&=
\mathds{1}_{j=k} I_d, 
\qquad j \in \llbracket 1,k \rrbracket,
\end{cases}
\end{align}
and
\begin{align}
\begin{cases}
\bar{\Bc}_{j,k}(t)
&=
\Phi(t,t_i) \bar{\Bc}_{j,k}(t_i)
+
\displaystyle\int_{t_i}^{t}
\Phi(t,s)
\Big(
\bar{F}_j(s)
+
\bar{A}(s) \big( \Bc_{j,k}(t) + \bar{\Bc}_{j,k}(t) \big)
\\
&\hspace{3.2cm}
+
\displaystyle\sum_{\ell=k+1}^{i}
\left(
F_\ell(s) \bar{\Bc}_{j,k}(t_\ell)
+
\bar{F}_\ell(s) \big( \Bc_{j,k} + \bar{\Bc}_{j,k} \big)(t_\ell)
\right)
\Big)
\d s,
\qquad t_i \leq t \leq t_{i+1},
\\[2mm]
\bar{\Bc}_{j,k}(t_k)
&=
0,
\qquad j \in \llbracket 1,k \rrbracket,
\end{cases}
\end{align}
and
\begin{align}
\begin{cases}
\rho_k(t)
&=
\Phi(t,t_i) \rho_k(t_i)
+
\displaystyle\int_{t_i}^{t}
\Phi(t,s)
\Big(
\kappa(s)
+
\bar{A}(s) \rho_k(s)
\\
&\hspace{3.2cm}
+
\displaystyle\sum_{\ell=k+1}^{i}
\left(
(F_\ell(s) + \bar{F}_l(s) ) \rho_k(t_\ell)
\right)
\Big)
\d s,
\qquad t_i \leq t \leq t_{i+1},
\\[2mm]
\rho_k(t_k)
&=
0,
\end{cases}
\end{align}
\end{proof}
Hence, Lemma \ref{lemma : stability} makes the ansatz  consistent with \eqref{eq : FBSDE_LQ} and we can therefore identify the jump terms associated to $\Lambda_i, \bar{\Lambda}_i$,$\Gamma$, $\bar{\Gamma}$ and $\chi$. Combining this, we are able to compute $\Lambda_i(t_k^{-}), \bar{\Lambda}_i(t_k^{-})$, $\Gamma(t_k^{-})$, $\bar{\Gamma}(t_k^{-})$ and $\chi(t_k^{-})$ for any $k \in \llbracket 1, N-1 \rrbracket$.  We first write that
\begin{align}
    \E[Y_t | \Fc_{t_k}]= \sum_{j=1}^{k} \Cc_{j,k}(t) X_{t_j} +  \sum_{j=1}^{k} \bar{\Cc}_{j,k}(t) \E[X_{t_j}] + c_k(t), \quad t \in [t_k,T],
\end{align}
where we set
\begin{align}
\begin{cases}
    \Cc_{j,k}(t) &:=  \Lambda_j(t) + \Sum_{l=k+1}^{\eta(t) } \Lambda_l(t) \Bc_{j,k}(t_l) + \Gamma(t) \Bc_{j,k}(t), \\
    \bar{\Cc}_{j,k}(t) &:= \bar{\Lambda}_j(t) + \Sum_{l=k+1}^{\eta(t)} \Lambda_l(t) \bar{\Bc}_{j,k}(t_l) + \bar{\Lambda}_l(t) (\Bc_{j,k} + \bar{\Bc}_{j,k})(t_l) + \Gamma(t) \bar{\Bc}_{j,k}(t) + \bar{\Gamma}(t) (\Bc_{j,k} + \bar{\Bc}_{j,k})(t), \\
    c_k(t) &:=  \chi(t) + \Sum_{l=k+1}^{\eta(t)} \big(\Lambda_l(t) + \bar{\Lambda}_l(t) \big)\rho_k(t_l)   + \big(\Gamma(t) + \bar{\Gamma}(t) \big)\rho_k(t) .
\end{cases}
\end{align}
Indeed, it follows directly that the jump term can be rewritten as 
\begin{align}
    Y_{t_k^{-}} = Y_{t_k} + \sum_{j=1}^{k} J_{j,k} X_{t_j} + \sum_{j=1}^{k} \bar{J}_{j,k} \E[X_{t_j}] + J_k^0,
\end{align}
where we have that
\begin{align}\label{eq : def_jump_terms}
\begin{cases}
    J_{j,k} &= \displaystyle\int_{t_k}^{T} \big(B_k^{\top} \Cc_{j,k}(t) + 2 P_{jk} + 2 \sum_{l=k+1}^{\eta(t)} P_{l,k} \Bc_{j,k}(t_l) \big) \d t  + 2 G_{j,k} + 2 \sum_{l=k+1}^{N} G_{l,k} \Bc_{j,k}(t_l) , \\
    \bar{J}_{j,k} &= \int_{t_k}^{T} \Big(\big(B_k^{\top} \bar{\Cc}_{j,k}(t) + \bar{B}_k^{\top} (\Cc_{j,k}(t) + \bar{\Cc}_{j,k}(t) \big) + 2 \displaystyle\sum_{l=k+1}^{\eta(t)} P_{l,k} \bar{\Bc}_{j,k}(t_l) + 2 \bar{P}_{jk} + 2 \sum_{l=k+1}^{\eta(t)} \bar{P}_{l,k} (\Bc_{j,k}(t_l) + \bar{\Bc}_{j,k}(t_l)) \Big) \d t , \\
    &\qquad \displaystyle+ 2 \sum_{l=k+1}^{N} G_{lk} \bar{\Bc}_{j,k}(t_l) + 2 \bar{G}_{jk} + 2 \sum_{l=k+1}^{N} \bar{G}_{lk} (\Bc_{j,k}(t_l) + \bar{\Bc}_{j,k}(t_l)), \\
    J_k^0 &= \displaystyle\int_{t_k}^{T} \Big( (B_k + \bar{B}_k)^{\top} c_k(t) + 2 \sum_{l=k+1}^{\eta(t)} \big(P_{lk} + \bar{P}_{lk} \big) \rho_k(t_l) \big) \Big)\d t + 2 \sum_{l=k+1}^{N} \big(G_{lk}  + \bar{G}_{lk} \big)\rho_k(t_l) .
\end{cases}
\end{align}
Therefore, the terms are characterized for any $k \in \llbracket 1, N-1 \rrbracket$ by
\begin{align}\label{eq : jump_terms}
\begin{cases}
\Lambda_j(t_k^{-} ) &= \Lambda_j(t_k)  + J_{j,k}, \quad 1 \leq j \leq k-1, \\
\bar{\Lambda}_j(t_k^{-}) &= \bar{\Lambda}_j(t_k) + \bar{J}_{j,k}, \quad 1 \leq j \leq k-1, \\ 
\Gamma(t_k^{-}) &= \Gamma(t_k) + \Lambda_k(t_k) +J_{k,k}, \\
    \bar{\Gamma}(t_k^{-}) &= \bar{\Gamma}(t_k) + \bar{\Lambda}_k(t_k) + \bar{J}_{k,k},  \\
     \chi(t_k^{-}) &= \chi(t_k) + J_k^0,
\end{cases}
\end{align}
where the terms have been defined in \eqref{eq : def_jump_terms}. 
Hence, the Riccati system is fully characterized through \eqref{eq : Riccati_system}, \eqref{eq : terminal_condition} and \eqref{eq : jump_terms}. Indeed, starting from \eqref{eq : terminal_condition}, we compute for any $1 \leq j \leq N-1$  the Riccati system \eqref{eq : Riccati_system} and hence the quantities $\Lambda_{i}(t_{j}), \bar{\Lambda}_i(t_{j}), \Gamma(t_{j}), \bar{\Gamma}(t_{j}), \chi(t_{j})$ for any $1 \leq i \leq j$. Then, following \eqref{eq : jump_terms}, we compute the quantities at $t_{j}^{-}$ for any $1 \leq i \leq j-1$ and proceed similarly to compute the quantities at date $t_{j-1}$ following again \eqref{eq : Riccati_system}.
Hence, the solution of the forward-backward equation $(X,Y,Z)$ is given by
\begin{align}
\begin{cases}
    Y_t &= \Sum_{i=1}^{\eta(t)} \Lambda_i(t) X_{t_i} + \Sum_{i=1}^{\eta(t)} \bar{\Lambda}_i(t) \E[X_{t_i}] + \Gamma(t)X_t + \bar{\Gamma} (t)  \E[X_t]+ \chi(t),   \\
    Z_t &= \Gamma(t) \sigma,
\end{cases}
\end{align}
for any $0 \leq t \leq T$ and  where $X=(X_t)_{0 \leq t \leq T}$ is the unique solution (see Proposition \ref{thm : existence_unicity_X_SDE}) to the linear SDE
\begin{align}
\begin{cases}
    \d X_t &= \beta - \Sigma \chi(t) + \Sum_{j=1}^{\eta(t)} \big(B_j - \Sigma \Lambda_j(t) \big) X_{t_j} + \Sum_{j=1}^{\eta(t)} \big( \bar{B}_j - \Sigma \bar{\Lambda}_j(t) \big) \E[X_{t_j}]  + \big(C- \Sigma \Gamma(t) \big) X_t + \big( \bar{C}- \Sigma \bar{\Gamma}(t) \big) \E[X_t], \\
    X_0 &= \xi,
\end{cases}
\end{align}
where $(\Lambda_i)_{1 \leq i \leq N}$, $(\bar\Lambda_i)_{1 \leq i \leq N}$,  $\Gamma$, $\bar\Gamma$ and $\chi$ are solutions to the  system \eqref{eq : Riccati_system}, \eqref{eq : terminal_condition} and \eqref{eq : jump_terms}.

\subsection{A framework for time series generation}\label{subsec : time_series}

We now discuss a class of models which can be used for the purpose of time series generation through diffusion models. The controlled state process $X=(X_t)_{t \in [0,T]}$ is given by
\begin{align}\label{eq : state_process_X}
\begin{cases}
    \d X_t &= \alpha_t \d t + \sigma(t,X_t) \d W_t, \\
    X_0 &= 0,
\end{cases}
\end{align}
for some given map $\sigma : [0,T] \times \R^d \to \R^{d \times n}$ satisfying Assumptions \ref{assumption : regularity_b_and_sigma} and \ref{assumption : regularity_pontryagin_optimality} and, given some target time series distribution $\rho \in \Pc_2\big((\R^d)^N\big)$, we aim to minimize over $\alpha \in \Ac$ the following cost functional parametrized by $\lambda > 0$
\begin{align}\label{eq : cost_function_time_series}
    \Ac \ni \alpha \mapsto J^{\lambda}(\alpha) := \E \Big[\frac{1}{2} \int_{0}^{T} |\alpha_t|^2 \d t \Big] + \lambda  D_{K}\big(\P_{(X_{t_1},\ldots,X_{t_N})} \, \big| \, \rho \big),
\end{align}
where the Maximum Mean Discrepancy (MMD) map $D_K(\cdot |\rho)$ is defined as $\Pc_2((\R^d)^N) \ni \mu \mapsto D_K(\mu | \rho)  = \frac{1}{2} \int_{((\R^d)^N)^2} K(\bx,\by) (\mu - \rho)(\d \bx) (\mu -\rho)(\d \by)$, for some measurable kernel map $K : (\R^d)^N \times (\R^d)^N \to \R$ for which $D_K$ is characteristic, i.e. $D_K(\mu|\rho) =0$ if and only if $\mu=\rho$.
This setting falls in our framework, with control set $A = \R^d$ and the following maps
\begin{align*}
\begin{cases}
    b(t,\bx,\boldsymbol{\mu}, x,\mu,a) &:= a, \\
    \sigma(t,\bx,\boldsymbol{\mu},x,\mu,a) &:= \sigma(t,x), \\
    f(t,\bx,\boldsymbol{\mu},x,\mu,a) &:= \frac{1}{2} |a|^2, \\
    g(\bx,\boldsymbol{\mu}) &:= \lambda D_K(\boldsymbol{\mu}|\rho).
\end{cases}
\end{align*}
\begin{Remark}
The motivation behind the control problem \eqref{eq : state_process_X}--\eqref{eq : cost_function_time_series} is closely related to generative modeling. As discussed in the introduction, the objective is to generate samples from a prescribed target distribution $\rho$. In the context of time series, this target distribution is identified with the joint law of the random vector $(X_{t_1},\ldots,X_{t_N})$.
The control problem associated with the cost functional $(J^\lambda)$ is designed so that, for an optimal control $\alpha^\star$ and the corresponding optimal state process $(X^{\lambda,\alpha^\star}_t)_{0 \leq t \leq T}$, one formally expects
$
D_K\Big(\mathbb P_{(X_{t_1}^{\lambda,\alpha^\star},\ldots,X_{t_N}^{\lambda,\alpha^\star})} \, \big\vert \, \rho\Big)
\to 0,
$
when $\lambda$ goes to $ \infty$.
In other words, the joint distribution of the optimally controlled trajectory at the prescribed observation dates should converge, in the sense induced by the divergence $D_K$, to the target law $\rho$.
Consequently, solving the optimal control problem \eqref{eq : state_process_X}--\eqref{eq : cost_function_time_series} and simulating independent copies
$\bigl(X_{t_1}^i,\ldots,X_{t_N}^i\bigr)_{i \in \llbracket 1, m \rrbracket }$ of the optimally controlled trajectory provides a collection of samples whose empirical joint distribution is expected to be close to $\rho$.
From a numerical perspective, other discrepancy functionals could also be considered, including divergence- or transport-based measures such as the Kullback--Leibler divergence or Wasserstein distances. From a theoretical standpoint, however, these alternatives are generally more delicate to handle, as they may fail to satisfy the convenient Lipschitz regularity properties enjoyed by kernel-based discrepancies. Extending the analysis to such settings would therefore require additional technical arguments.
\end{Remark}

We now make the following assumptions on the choice of the kernel $K$ which are relevant in the proof of existence and uniqueness of the associated FBSDE system.
\begin{Assumption}\label{assumption: existence_unicity_FBSDE}
The kernel $K : (\R^d)^N \times (\R^d)^N \to \R$ is symmetric, of class $\Cc^2$, and the gradient $\partial_{\bx} K$ is bounded and Lipschitz on $(\R^d)^N \times (\R^d)^N$. Equivalently, $\partial_{\bx} K$, $\partial^2_{\bx} K$ and the mixed Hessian $\partial_{\by}\partial_{\bx} K$ are all uniformly bounded:
\begin{align*}
    \underset{\bx, \by \in (\R^d)^N}{\sup}\; \bigl( \lVert \partial_{\bx} K(\bx,\by) \rVert + \lVert \partial^2_{\bx} K(\bx,\by) \rVert + \lVert \partial_{\by}\partial_{\bx} K(\bx,\by) \rVert \bigr) < \infty.
\end{align*}
\end{Assumption}
We shall notice that the Gaussian kernel defined as $K(\bx,\by)=e^{- \frac{\lVert \bx - \by \rVert^2}{2 \epsilon^2}}$ for any $\bx,\boldsymbol{y} \in (\R^d)^N$ and for some $\epsilon > 0$ is characteristic and satisfies the requirements of Assumption \ref{assumption: existence_unicity_FBSDE} (see \cite{sriperumbudur2010hilbert})

\begin{Remark}\label{rmk : lipschitz_kernel} 
\noindent Under Assumption \ref{assumption: existence_unicity_FBSDE}, for any $1 \leq k \leq N$, the map
    \begin{align*}
        \Pc_2((\R^d)^N) \times (\R^d)^N \ni (\boldsymbol{\mu},\bx) \mapsto \partial_{x_k} \frac{\delta}{\delta \bm} D_K(\boldsymbol{\mu} |\rho)(\bx),
    \end{align*}
    satisfies the following conditions
    \begin{enumerate}
        \item [(1)] $(\boldsymbol{\mu},\bx) \mapsto \partial_{x_k} \frac{\delta}{\delta \bm} D_K(\boldsymbol{\mu}|\rho)(\bx)$ is uniformly bounded.
        \item [(2)] There exists a positive constant $C_K > 0$ such that for any $\boldsymbol{\mu},\boldsymbol{\mu}' \in \Pc_2((\R^d)^N)$ and $\bx,\bx' \in (\R^d)^N$,
        \begin{align*}
            | \partial_{x_k} \frac{\delta}{\delta \bm} D_K(\boldsymbol{\mu}'|\rho)(\bx') - \partial_{x_k} \frac{\delta}{\delta \bm} D_K(\boldsymbol{\mu}|\rho)(\bx) | \leq C_K \big( \Wc_2(\boldsymbol{\mu},\boldsymbol{\mu}') + |\bx-\bx'| \big).
        \end{align*}
    \end{enumerate}
\end{Remark}
It is easy to verify that in this setting $\partial_{x_k} \frac{\delta}{\delta \bm} D_{K}(\boldsymbol{\mu} | \rho)(\bx ) = \int_{(\R^d)^N} \partial_{x_k} K(\bx,\by) \boldsymbol{\mu}(\d \by) - \int_{(\R^d)^N} \partial_{x_k} K(\bx, \by) \rho(\d \by)$ from which Remark \ref{rmk : lipschitz_kernel} follows. Therefore the associated FBSDE system with jumps is given by
\begin{align}\label{eq : FBSDE_jumps_kernel}
\begin{cases}
    \d X_t &= - Y_t \d t + \sigma(t,X_t) \d W_t, \\
    X_0 &= 0, \\
    \d Y_t &= - \partial_x \big(\sigma(t,X_t) \cdot Z_t \big)\d t + Z_t \d W_t, \quad t_{k-1} \leq t < t_{k}, \quad k \in \llbracket 1, N-1 \rrbracket, \quad \text{ and $t_{N-1} \leq t \leq t_N$,} \\
    Y_{t_k} - Y_{t_k^{-}} &= - \lambda  \E \Big[  \int_{(\R^d)^N} \partial_{x_k} K(\bX, \by)\P_{(X_{t_1},\ldots,X_{t_N})}(\d \by) - \int_{(\R^d)^N} \partial_{x_k} K(\bX, \by)\rho(\d \by) \Big|\Fc_{t_k} \Big], \quad k \in \llbracket 1, N-1 \rrbracket, \\
    Y_{t_N} &= \lambda \Big( \int_{(\R^d)^N} \partial_{x_N} K(\bX,\by) \P_{(X_{t_1},\ldots,X_{t_N})}(\d \by) - \int_{(\R^d)^N} \partial_{x_N} K(\bX,\by) \rho(\d \by) \Big).
\end{cases}
\end{align}

We make the additional assumption on the diffusion coefficient $\sigma$.

\begin{Assumption} \label{eq:assumsig} 
    Suppose that there exists a measurable map $\Sigma :=(\Sigma_1,\ldots, \Sigma_n)  : [0,T] \times \R^d \to \R^n$ such that
    \begin{align}
        \partial_x \sigma^j(t,x) = \Sigma_j(t,x) I_d, \quad j \in \llbracket 1, n \rrbracket,
    \end{align}
    where $\sigma^j$ is the $j$-th column of $\sigma$. 
\end{Assumption}

\begin{Remark}
Assumption \ref{eq:assumsig} is clearly satisfied in dimension $d$ $=$ $1$, and for $d$ $\geq$ $2$, it holds whenever $\sigma$ is affine in $x$.     
\end{Remark}

\begin{Theorem}\label{thm : local_existence_FBSDE_kernel}
    Under Assumptions \ref{assumption: existence_unicity_FBSDE} and \ref{eq:assumsig}, there exists $\delta>0$, depending only on the bound and Lipschitz constant of $\partial_x\sigma$, the Lipschitz and boundedness constants of $\partial_{x_k}\frac{\delta}{\delta\bm}D_K$ provided by Remark \ref{rmk : lipschitz_kernel} and Assumption \ref{assumption : regularity_pontryagin_optimality}(i), the penalization constant $\lambda$, and the number of observation times $N$, such that, for any $T\in(0,\delta]$, the FBSDE \eqref{eq : FBSDE_jumps_kernel} admits a unique solution
$(X,Y,Z)\in
\S^4([0,T];\R^d)
\times
\S^\infty([0,T];\R^d)
\times
\H^2([0,T];\R^{d\times n}),$
with $Z$ belonging to $\mathrm{BMO}$.
\end{Theorem}

\begin{proof}
For $1 \leq k \leq N$, denote
\begin{align*}
G_k(\boldsymbol{\mu},\bx) := \int_{(\R^d)^N} \partial_{x_k} K(\bx,\by) \boldsymbol{\mu}(\d \by) - \int_{(\R^d)^N} \partial_{x_k} K(\bx,\by) \rho(\d \by).
\end{align*}
By Remark \ref{rmk : lipschitz_kernel}, $G_k$ is bounded and Lipschitz in $(\boldsymbol{\mu},\bx)$, uniformly in $k$: there exists $L > 0$ such that
\begin{align*}
| G_k(\boldsymbol{\mu},\bx) - G_k(\boldsymbol{\mu}',\bx') | \leq L \big( \Wc_2(\boldsymbol{\mu}, \boldsymbol{\mu}') + |\bx - \bx'|\big).
\end{align*}
We carry out the fixed point argument on the Banach space $\Bc := \S^4([0,T];\R^d)$ endowed with the norm
$$
    \lVert \bar Y \rVert_{\Bc}^4 := \E \big[ \sup_{0 \leq t \leq T} | \bar Y_t|^4 \big].
$$
\noindent \textbf{Construction of the map $\Psi$.} Given $\bar Y \in \Bc$, we define $\Psi(\bar Y) := Y$ through the following two steps.
\begin{enumerate}
\item [(i)] \emph{Forward SDE.} Consider the forward SDE
\begin{align}\label{eq : forward_phi}
\d X_t = - \bar Y_t \d t + \sigma(t,X_t) \d W_t, \quad X_0 = 0.
\end{align}
Since $\sigma$ is Lipschitz in $x$ (by the boundedness of $\partial_x \sigma$, see Assumption \ref{assumption : regularity_pontryagin_optimality}) and $\bar Y \in \Bc \subset \H^2$, standard SDE theory yields a unique strong solution $X \in \S^4([0,T];\R^d)$. Denote $\bX := (X_{t_1}, \ldots, X_{t_N}) \in L^4(\Omega;(\R^d)^N)$ and $\boldsymbol{\mu} := \P_{\bX} \in \Pc_2((\R^d)^N)$.
\item [(ii)] \emph{Backward BSDE with jumps.} Given $X$ and $\boldsymbol{\mu}$ from step (i), define $(Y,Z)$ as the unique solution of the BSDE with jumps at the deterministic times $(t_k)_{1 \leq k \leq N-1}$
\begin{align}\label{eq : backward_phi}
\begin{cases}
    \d Y_t &= - \partial_x \big(\sigma(t,X_t) \cdot Z_t \big)\d t + Z_t \d W_t, \quad t_{k-1} \leq t < t_{k}, \; k \in \llbracket 1, N-1 \rrbracket, \; \text{and } t_{N-1} \leq t \leq t_N,\\
    Y_{t_k} - Y_{t_k^-} &= - \lambda \E \big[ G_k(\boldsymbol{\mu}, \bX) \big| \Fc_{t_k} \big], \quad k \in \llbracket 1, N-1 \rrbracket, \\
    Y_{t_N} &= \lambda G_N(\boldsymbol{\mu},\bX).
\end{cases}
\end{align}
Since $G_k$ is bounded, the jump terms and the terminal condition are bounded $\Fc_{t_k}$-measurable random variables. The driver $(y,z) \mapsto -  \partial_x \big(\sigma(t,X_t) \cdot z \big)\d t$ is Lipschitz in $z$ uniformly in $(t,\omega)$ by the boundedness of $\partial_x \sigma$. Solving the BSDE backward on each interval $[t_{k-1}, t_k)$ by the standard theory for Brownian BSDEs and imposing the prescribed jumps at $(t_k)_{1 \leq k \leq N-1}$ produces a unique solution $(Y,Z)$. Moreover, the boundedness of the terminal condition and of the jumps yields a stronger estimate on the backward solution. More precisely, let
$$
    \overline G:=\sup_{1\leq k\leq N}
    \sup_{\boldsymbol{\mu},\bx}
    |G_k(\boldsymbol{\mu},\bx)|<\infty.
$$
Then 
$$
    |Y_{t_N}|
    \leq \lambda\overline G,
    \quad
    |Y_{t_k}-Y_{t_k^-}|
    \leq \lambda\overline G,
    \quad k\in\llbracket1,N-1\rrbracket, \quad \P-\text{a.s},
$$
Now, under Assumption \ref{eq:assumsig}, we have for all $t \in [0,T]$, 
\begin{align}
\partial_x \big(\sigma(t,X_t) \cdot Z_t \big)  &=\;  Z_t \Sigma(t,X_t).   
\end{align}
Moreover, by Assumption \ref{assumption : regularity_pontryagin_optimality}, there exists a constant
$C_\sigma>0$ such that
$$
    |\Sigma(t,X_t)|\leq C_\sigma,
    \qquad 0\leq t\leq T, \quad \P-\text{a.s}, 
$$
thus,  by  Novikov's condition, and Girsanov's theorem, there exists an equivalent probability measure $\Q$ such that
$$
    W_t^\Q:=W_t-\int_0^t \Sigma(s,X_s)\,\d s,
$$
is a Brownian motion under $\Q$. Hence, on every interval
$[t_{k-1},t_k)$, $\d Y_t=Z_t\,\d W_t^\Q.$ Consequently,
$$
    Y_t
    =
    \E^\Q\left[
        Y_{t_k^-}\mid\Fc_t
    \right],
    \qquad
    t\in[t_{k-1},t_k).
$$
Starting from the terminal interval, we have $|Y_{t_N}|\leq\lambda\overline G,$ and therefore $|Y_t|
    \leq\lambda\overline G$, for $t\in[t_{N-1},t_N].$
At time $t_{N-1}$,
$$
    Y_{t_{N-1}^-}
    =
    Y_{t_{N-1}}
    +
    \lambda
    \E\left[
        G_{N-1}(\boldsymbol{\mu},\bX)
        \mid\Fc_{t_{N-1}}
    \right],
$$
and thus $|Y_{t_{N-1}^-}|\leq 2\lambda\overline G.$ Repeating this argument backward over the successive intervals gives
\begin{align}\label{eq : uniform_Y_bound}
\sup_{0\leq t\leq T}|Y_t|
\leq N\lambda\overline G,
\qquad \P\text{-a.s.}
\end{align}
In particular, $Y\in\S^\infty([0,T];\R^d),$
with a bound independent of the input $\bar Y$. We next derive a BMO estimate for $Z$. Let $\tau$ be an arbitrary stopping time. Applying Ito's formula to $|Y_t|^2$ on each interval between successive observation times and summing the resulting identities from $\tau$ to $T$, we obtain
\begin{align*}
\int_\tau^T |Z_t|^2\,\d t
&=
|Y_T|^2-|Y_\tau|^2
+2\int_\tau^T
Y_t\cdot \partial_x \big(\sigma(t,X_t) \cdot Z_t \big)\,\d t
-2\int_\tau^T
Y_t\cdot Z_t\,\d W_t
+R_\tau,
\end{align*}
where $R_\tau$ contains the finitely many contributions produced by the jumps of
$Y$ after $\tau$. Since both $Y$ and its jumps are uniformly bounded, there exists
a deterministic constant $C>0$, depending only on
$N$, $\lambda$, and $\overline G$, such that $|R_\tau|\leq C.$ Taking the conditional expectation with respect to $\Fc_\tau$ yields
\begin{align*}
\E\left[
\int_\tau^T |Z_t|^2\,\d t\Big|\Fc_\tau
\right]
&\leq
C
+
2\E\left[
\int_\tau^T
|Y_t|\,|\partial_x\sigma(t,X_t)|\,|Z_t|\,\d t\Big|\Fc_\tau
\right].
\end{align*}
By \eqref{eq : uniform_Y_bound} and the boundedness of $\partial_x\sigma$,

$$
    2|Y_t|\,|\partial_x\sigma(t,X_t)|\,|Z_t|
    \leq
    2N\lambda\overline G\,C_\sigma |Z_t|.
$$
Young's inequality gives
$$
    2N\lambda\overline G\,C_\sigma |Z_t|
    \leq
    \frac12|Z_t|^2
    +
    2N^2\lambda^2\overline G^2 C_\sigma^2.
$$
Therefore,
\begin{align*}
\E\left[
\int_\tau^T|Z_t|^2\,\d t\Big|\Fc_\tau
\right]
&\leq
C
+
\frac12
\E\left[
\int_\tau^T|Z_t|^2\,\d t\Big|\Fc_\tau
\right]
+
2N^2\lambda^2\overline G^2C_\sigma^2T.
\end{align*}
Hence, we obtain
$$
    \E\left[
        \int_\tau^T|Z_t|^2\,\d t\Big|\Fc_\tau
    \right]
    \leq C,
$$
where $C$ is deterministic and independent of $\tau$ and of the input $\bar Y$.
Consequently,  $Z\in\mathrm{BMO}$ and 
$$
    \sup_{\tau}
    \left\|
        \E\left[
            \int_\tau^T|Z_t|^2\,\d t
            \Big|\Fc_\tau
        \right]
    \right\|_{L^\infty}
    <\infty.$$
\end{enumerate}
A process $(X,Y,Z) \in \S^4([0,T];\R^d) \times \S^4([0,T];\R^d) \times \H^2([0,T];\R^{d \times n})$ solves \eqref{eq : FBSDE_jumps_kernel} if and only if $Y = \Psi(Y)$ and $(X,Z)$ are the processes defined in steps (i)-(ii).

\noindent \textbf{Contraction estimate.} Let $\bar Y^1, \bar Y^2 \in \Bc$, and denote $(X^i, Y^i, Z^i)$, $\boldsymbol{\mu}^i := \P_{\bX^i}$ for $i=1,2$, the processes constructed from $\bar Y^i$ as above. Write $\Delta \bar Y := \bar Y^1 - \bar Y^2$ and similarly $\Delta X, \Delta Y, \Delta Z, \Delta \bX = \bX^1 - \bX^2$.

\noindent \emph{Forward estimate.} Subtracting the two forward equations, applying Ito's formula to $|\Delta X_t|^4$, using the Lipschitz property of $\sigma$, BDG's inequality and Gronwall's lemma, we obtain
\begin{align}\label{eq : forward_contraction}
\E \big[ \underset{0 \leq t \leq T}{\text{sup}} | \Delta X_t |^4 \big]
\leq C_1 T^3 \, \E \Big[ \int_{0}^{T} |\Delta \bar Y_t|^4 \d t \Big]
\leq C_1 T^4 \, \lVert \Delta \bar Y \rVert_{\Bc}^4,
\end{align}
where $C_1$ depends only on the Lipschitz constant of $\sigma$ and is bounded in $T$ for $T \leq 1$.

\noindent \emph{Backward estimate.} From $\Wc_2(\boldsymbol{\mu}^1,\boldsymbol{\mu}^2)^2 \leq \E [|\Delta \bX|^2] \leq N \, \E[ \sup_{0 \leq t \leq T}|\Delta X_t|^2]$ and the Lipschitz property of $G_k$, the terminal and jump differences are controlled by $\Delta X$. Moreover,
\begin{align*}
\partial_x \big(\sigma(t,X_t^1) \cdot Z_t^1 \big) -\partial_x \big(\sigma(t,X_t^2) \cdot Z_t^2 \big) =
\partial_x \big( \sigma(t,X_t^1) \cdot \Delta Z_t) \big) 
+ \partial_x \big( (\sigma(t,X_t^1)-\sigma(t,X_t^2) ) \cdot Z_t^2 \big)
\end{align*}
The first term is controlled by the boundedness of $\partial_x \sigma$, while the Lipschitz continuity of $\partial_x \sigma$ gives
\begin{align*}
\left|
\partial_x \big( (\sigma(t,X_t^1)-\sigma(t,X_t^2) ) \cdot Z_t^2 \big)
\right|
\leq
L_\sigma |\Delta X_t|\,|Z_t^2|.
\end{align*}
Using the uniform BMO bound on $Z^2$, we have
\begin{align*}
\E\left[
\int_0^T |\Delta X_t|^2 |Z_t^2|^2 \d t
\right]
&\leq
\left(
\E\left[
\sup_{0\leq t\leq T}|\Delta X_t|^4
\right]
\right)^{1/2}
\left(
\E\left[
\left(
\int_0^T |Z_t^2|^2 \d t
\right)^2
\right]
\right)^{1/2}
\\&\leq
C
\left(
\E\left[
\sup_{0\leq t\leq T}|\Delta X_t|^4
\right]
\right)^{1/2}.
\end{align*}
Standard fourth-moment stability estimates for BSDEs with jumps at deterministic times, together with the uniform BMO bound, yield
\begin{align}\label{eq : backward_contraction}
\lVert \Delta Y \rVert_{\Bc}^4
\leq
C_3 \,
\E \big[
\underset{0 \leq t \leq T}{\text{sup}}
|\Delta X_t|^4
\big],
\end{align}
where $C_3$ depends only on $\lambda,L,N$, the bounds on $\partial_x \sigma$, its Lipschitz constant, and the uniform BMO bound, and is bounded in $T$ for $T\leq1$.

\noindent \emph{Combining the estimates.} Inserting \eqref{eq : forward_contraction} into \eqref{eq : backward_contraction}, we get, for $T \leq 1$,
\begin{align*}
\lVert \Psi(\bar Y^1) - \Psi(\bar Y^2) \rVert_{\Bc}^4
\leq
C \, T^4 \,
\lVert \bar Y^1 - \bar Y^2 \rVert_{\Bc}^4,
\end{align*}
where $C := C_1 C_3$ is independent of $T$. Choosing $\delta > 0$ such that $C \delta^4 < 1$, the map $\Psi$ is a strict contraction on $\Bc$ for any $T \in (0,\delta]$. By the Banach fixed-point theorem, $\Psi$ admits a unique fixed point $Y \in \Bc$, and the associated processes $X \in \S^4$ and $Z \in \H^2$ defined by steps (i)-(ii) provide the unique solution $(X,Y,Z)$ to the FBSDE \eqref{eq : FBSDE_jumps_kernel} on $[0,T]$.
\end{proof}

\section{Conclusion}\label{sec : conclusion}

This paper develops a finite path-dependent McKean--Vlasov control framework for problems in which the objective depends on the law of a sampled path. The dependence on the past is described through a fixed observation grid, while the joint law of the observed values captures the distributional information of the path. This provides a finite-dimensional formulation for control problems with path-dependent distributional objectives.

The main feature of the resulting optimality system is the adjoint equation. In contrast with the classical McKean--Vlasov setting, the backward process has jumps at deterministic observation times. These jumps propagate the sensitivity of the future cost with respect to each observed state value and to its law. The stochastic maximum principle obtained in this paper therefore gives a natural control-theoretic description of finite path-dependence and of distributional costs on sampled paths.

This framework also provides a data-driven mathematical foundation for sample-based time-series generation through a relaxed McKean--Vlasov control formulation. Moreover, it is related to optimal transport and Schrödinger-Bass bridge approaches for time-series generation; see, for instance, \cite{hamdouche2023generative,alouadi2026sbbts}. In this setting, the control process $(\alpha_t)_{0\leq t\leq T}$ acts as an additional drift that steers a reference dynamics toward a target law on sampled paths. The deterministic observation dates may induce structural changes in the optimal control, consistently with the jumps appearing in the adjoint backward equation of our maximum principle.

Several directions remain open. A first one is to study relaxed formulations in which a family of costs $J^\lambda(\alpha)$ approximates the Schrödinger bridge cost $J(\alpha)$ as the penalty parameter varies. Proving such a convergence would clarify the link between the present McKean--Vlasov control problem and the Schrödinger--Bass bridge problem for time-series generation. It would also be useful to identify verifiable assumptions under which the Hamiltonian and the terminal cost are convex in the McKean--Vlasov sense used in the sufficient maximum principle. Such conditions would strengthen the optimality theory and may lead to uniqueness of the optimal control, and of the optimal state-control law, beyond the linear-quadratic setting.

Further work also includes the numerical analysis of the forward--backward system arising from the control problem, the extension of global solvability results beyond the linear-quadratic case, and the treatment of less regular costs such as the Kullback--Leibler divergence.

\appendix

\renewcommand{\thesection}{\Alph{section}}

\section{Proof of Proposition \ref{thm : existence_unicity_X_SDE}}\label{app:proofs}

\begin{proof}[Proof of Proposition \ref{thm : existence_unicity_X_SDE}]
Let $\nu \in \Pc_2(\Cc([0,T];\R^d))$. Denote the projection maps 
\begin{align*}
\begin{cases}
   x_{s} &: \Cc([0,T];\R^d) \ni \omega \mapsto \omega_{s} \in \R^d, \quad  \text{ for any $s \in [0,T]$}, \\
   x_{(t_1,\ldots,t_N)} &: \Cc([0,T];\R^d) \ni \omega \mapsto (\omega_{t_1},\ldots, \omega_{t_N}) \in (\R^d)^N.
\end{cases}
\end{align*}
For any $\bnu \in\Pc_2(\Cc([0,T];\R^d))$, we denote $\bnu_t := x_t \sharp \bnu$ and $\bnu_{(t_1,\ldots,t_N)} := x_{(t_1,\ldots, t_N)} \sharp \bnu$ and  by $X^{\bnu}$ the solution to
\begin{align}\label{eq : SDE_nu_by_nu}
\begin{cases}
    \d X_t^{\bnu} &= b(t, \bX^{\bnu}, \bnu_{(t_1,\ldots,t_N)}, X_t^{\bnu}, \bnu_t, \alpha_t) \d t + \sigma(t, \bX^{\bnu}, \bnu_{(t_1,\ldots,t_N)}, X_t^{\bnu}, \bnu_t, \alpha_t) \d W_t, \\
    X_0^{\bnu} &= \xi.
\end{cases}
\end{align}
On each interval between the successive dates $t_0 = 0,t_1,\ldots, t_{N-1}$, at which the time dependence of the coefficients may be discontinuous, we construct a solution belonging to $\S^2([0,T];\R^d)$ as follows. On $[t_0,t_1)$, we consider the SDE $X^{(0),\bnu}$ solution to
\begin{align*}
\begin{cases}
    \d X_t^{(0),\bnu} &= b(t, X_t^{(0),\bnu}, \bnu_t , \alpha_t) \d t + \sigma(t, X_t^{(0),\bnu}, \bnu_t,\alpha_t) \d W_t, \\
    X_0^{(0),\bnu} &= \xi.
\end{cases}
\end{align*}
where, for $t < t_1$, we write $b(t,x,\mu,a)$ and $\sigma(t,x,\mu,a)$ for the common values $\tilde{b}_t(x,\mu,a)$ and $\tilde{\sigma}_t(x,\mu,a)$, the path arguments being void. Under our assumptions, $X^{(0), \bnu}$ admits a continuous $\F$-adapted  solution, unique up to a $\P$-null set, over the set $[0,t_1)$. By the assumption on the model coefficients and standard arguments, the solution can be extended continuously to $t_1$ since $\underset{t \nearrow t_1}{\text{ lim }} X_{t}^{(0),\bnu}$ exists $\P-\text{a.s}$ and by setting $X_{t_1}^{(0),\bnu} =\underset{t \nearrow t_1}{\text{ lim }} X_{t}^{(0),\bnu} $. Moreover, we have that $X_{t_1}^{(0),\bnu} \in L^2(\Omega, \Fc ,\P)$. Now, we consider on $[t_1,t_2)$ the SDE $X^{(1),\bnu}$ solution to
\begin{align*}
\begin{cases}
    \d X_t^{(1),\bnu} &=  b(t,X_{t_1}^{(0),\bnu}, \bnu_{t_1},X_t^{(1),\bnu}, \bnu_t , \alpha_t) \d t  + \sigma(t,X_{t_1}^{(0),\bnu}, \bnu_{t_1},X_t^{(1),\bnu}, \bnu_t , \alpha_t) \d W_t , \\
    X_{t_1}^{(1),\bnu} &= X_{t_1}^{(0),\bnu}.
\end{cases}
\end{align*}
Iterating this construction on each interval $[t_i, t_{i+1})$ for $i \in \llbracket 0, N-1 \rrbracket$, we obtain a family of continuous solutions $X^{(i),\bnu}$. We then define the global process $X^{\bnu}$ on $[0,T]$ as
\begin{align}\label{eq : process solution_nu_fixed}
    X_t^{\bnu} := \sum_{i=0}^{N-2} \mathds{1}_{t_i \leq t < t_{i+1}} X_{t}^{(i),\bnu} + \mathds{1}_{t_{N-1} \leq t \leq t_N} X_{t}^{(N-1),\bnu}, \quad 0 \leq t \leq T.
\end{align}
This process is well defined (we set $X_T^{(N-1),\bnu} := \lim_{t \nearrow t_N} X_t^{(N-1),\bnu}$) and $\P-\text{a.s}$ continuous (since at date $t_{i}$, we start from the previous equality) and it is easy to verify that $\E \big[ \underset{0 \leq t \leq T}{\text{sup }} |X_t^{\bnu}|^2 \big] < \infty$ from standard estimates and is clearly a solution to \eqref{eq : SDE_nu_by_nu}. Therefore, we can now define the  map $\Psi$ from $\Pc_2(\Cc([0,T];\R^d))$ into itself as follows
\begin{align*}
    \Pc_2(\Cc([0,T];\R^d)) \ni \bnu \mapsto \Psi(\bnu) := \P_{X^{\bnu}} \in \Pc_2(\Cc([0,T];\R^d)).
\end{align*}
where $X^{\bnu}$ is the process constructed in \eqref{eq : process solution_nu_fixed}. We now prove that $\Psi$ has a unique fixed point on the complete metric space $\Pc_2(\Cc([0,T];\R^d))$ endowed with the distance  $\Wc_{2,T}$. Let $\bnu^1, \bnu^2 \in \Pc_2(\Cc([0,T];\R^d))$ and denote $X^{\bnu^1}, X^{\bnu^2}$ the corresponding solutions to \eqref{eq : SDE_nu_by_nu}. We therefore have (we only treat the diffusion term for readability, the drift term being handled similarly)
\begin{align*}
    &\E \Big[ \underset{0 \leq r \leq s}{\text{sup }} |X_r^{\bnu^1} - X_r^{\bnu^2} |^2 \Big] \leq C \E \Big[ \int_{0}^{s} | \sigma(r, \bX^{\bnu^1}, \bnu^1_{(t_1,\ldots, t_N)}, X_r^{\bnu^1}, \bnu^1_r, \alpha_r) - \sigma(r, \bX^{\bnu^2}, \bnu^2_{(t_1,\ldots,t_N)}, X_r^{\bnu^2}, \bnu^2_r, \alpha_r)|^2 \d r     \Big].
\end{align*}
Using the inequalities $\Wc_2^2(\P_{X_s}, \delta_0) \leq \E[|X_s|^2]$ and $\Wc_2^2(\P_{\bX_{\eta(s)}}, \boldsymbol{\delta}_0) \leq \sum_{i=1}^{\eta(s)} \E \big[|X_{t_i}|^2 \big]$, and since $t_{\eta(s)} \leq s$, we have
\begin{align*}
    \sum_{i=1}^{\eta(s)} \E[ |X_{t_i}|^2] \leq N \E \big[ \underset{0 \leq r \leq s}{\text{ sup }} |X_r|^2 \big].
\end{align*}
Moreover, since the map $x_s$ and $x_{(t_1,\ldots,t_N)}$ are respectively $1$- and $\sqrt{N}$-Lipschitz with respect to $\lVert \cdot \rVert_{\infty}$ norm, it follows from the Lipschitz assumption on $\sigma$ that
\begin{align*}
    &\E \Big[ \underset{0 \leq r \leq s}{\text{sup }} |X_r^{\bnu^1} - X_r^{\bnu^2} |^2 \Big] \leq C \E \Big[ \int_{0}^{s} \Big(\E \big[\underset{0 \leq l \leq r }{\text{ sup }} |X_l^{\bnu^1} - X_l^{\bnu^2} |^2 \big] + \Wc_{2,r}^2(\bnu^{1},\bnu^2)\Big)  \d r     \Big].
\end{align*}
Applying Grönwall's inequality to the map $s \mapsto \E \Big[ \underset{0 \leq r \leq s}{\text{sup }} |X_r^{\bnu^1} - X_r^{\bnu^2} |^2 \Big]$, we end up with
\begin{align}\label{eq : intermediate_estimate}
    \Wc_{2,s}^2 \big(\Psi(\bnu^1), \Psi(\bnu^2) \big) \leq \E \Big[ \underset{0 \leq r \leq s}{\text{sup }} |X_r^{\bnu^1} - X_r^{\bnu^2} |^2 \Big] \leq C \int_{0}^{s} \Wc_{2,r}^2(\bnu^1,\bnu^2) \d r,
\end{align}
for some constant $C > 0$ that only depends on the Lipschitz constants of $b,\sigma$, $T$ and $N$. Setting $s=T$, we have the continuity of the map $\Psi$ and if we iterate \eqref{eq : intermediate_estimate}, we have from standard computations
\begin{align*}
    \Wc_{2,s}^2 \big(\Psi^{(k+1)}(\bnu^1), \Psi^{(k+1)}(\bnu^2) \big)  \leq \frac{C^{k+1}}{k!} \int_{0}^{s} \Wc_{2,r}^2(\bnu^1,\bnu^2)(s-r)^k \d r.
\end{align*}
Now, since $\Wc_{2,s}(\bnu^1,\bnu^2) \leq \Wc_{2,T}(\bnu^1,\bnu^2)$ for any $0 \leq s \leq T$, choosing an arbitrary $\bnu^{(0)} \in \Pc_2(\Cc([0,T];\R^d))$ and denoting the sequence $\bnu^{(k+1)} = \Psi(\bnu^{(k)})$ for $k \geq 0$, it follows
\begin{align}\label{eq : inequality_W_2T}
    \Wc_{2,T}^2(\Psi(\bnu^{(k)}), \bnu^{(k)}) \leq \frac{C^k T^k}{k!} \Wc_{2,T}^2(\bnu^{(1)},\bnu^{(0)}).
\end{align}
Therefore, the sequence $(\bnu^{(k)})_{k \in \N}$ is a Cauchy sequence for $\Wc_{2,T}$ and converges therefore on the complete metric space $\Pc_2(\Cc([0,T];\R^d))$ to a limit $\boldsymbol{\bar{\bnu}}$ which is a fixed point for $\Psi$ from the continuity of $\Psi$. The uniqueness follows from \eqref{eq : inequality_W_2T} and denoting $X^{\boldsymbol{\bar{\nu}}}$ the process associated to the fixed point $\boldsymbol{\bar{\nu}}$, it is therefore a solution. Moreover, if $\tilde{X}$ is another solution to \eqref{eq : state_process_dynamics}, then its law $\mathbb{P}_{\tilde{X}}$ belongs to $\Pc_2( \Cc([0,T];\R^d))$, is a fixed point for the map $\Psi$, and must therefore coincide with $\bar{\bnu}$; it then follows, by uniqueness of the solution to the SDE with fixed input law $\bar{\bnu}$, that $\tilde{X} = X^{\bar{\bnu}}$ up to indistinguishability. Denote by $X=(X_t)_{0 \leq t \leq T}$ the unique strong solution to \eqref{eq : state_process_dynamics}. Then, we have 
\begin{align*}
    X_t  = \xi + \int_{0}^{t} b(s, \bX, \P_{\bX}, X_s, \P_{X_s},\alpha_s) \d s + \int_{0}^{t} \sigma(s, \bX, \P_{\bX} , X_s, \P_{X_s} ,\alpha_s ) \d W_s, \quad 0 \leq t \leq T.
\end{align*}
From the Cauchy-Schwarz inequality, we have
\begin{align*}
    |X_t|^2 &\leq 3 \Big( |\xi|^2 + |\int_{0}^{t} b(s, \bX, \P_{\bX} , X_s, \P_{X_s}, \alpha_s) \d s |^2 +  | \int_{0}^{t} \sigma(s,\bX,\P_{\bX}, X_s, \P_{X_s},\alpha_s) \d W_s |^2 \Big),
\end{align*}
with $0 \leq t \leq T$. Now taking the supremum over $[0, T]$, taking the expectation, using the BDG inequality and the Cauchy--Schwarz inequality $|\int_{0}^{T} f(s) \d s|^2 \leq T \int_{0}^{T} |f(s)|^2 \d s $ for any measurable map $f : [0,T] \to \R^d$, we have
\begin{align*}
    \E \Big[ \underset{0 \leq s \leq t}{\text{ sup }} |X_s|^2 \Big] &\leq 3 \bigg( \E \big[|\xi|^2 \big]  + T \E \Big[\int_{0}^t | b(s,\bX,\P_{\bX}, X_s, \P_{X_s},\alpha_s)|^2  \d s \Big] \\
    &\quad + 4 \E \Big[\int_{0}^{t}  | \sigma(s,\bX,\P_{\bX}, X_s, \P_{X_s},\alpha_s)|^2 \d s \Big] \bigg).
\end{align*}
Now, from the Lipschitz continuity and linear growth of $b$ and $\sigma$ (Assumption \ref{assumption : regularity_b_and_sigma}), and denoting by $C$ a generic positive constant, we have
\begin{align*}
     \E \Big[ \underset{0 \leq s \leq t}{\text{ sup }} |X_s|^2 \Big]  &\leq C  \bigg( 1 + \E \big[ |\xi|^2 \big] + \E \Big[ \int_{0}^t \big(| \bX_{\eta(s)}|^2+ \Wc_2^2(\P_{\bX_{\eta(s)}}, \boldsymbol{\delta}_0) + |X_s|^2 + \Wc_2^2(\P_{X_s}, \delta_0) + |\alpha_s|^2  \big)\d s \Big] \bigg).
\end{align*}
Using again the inequalities $\Wc_2^2(\P_{X_s}, \delta_0) \leq \E[|X_s|^2]$ and $\Wc_2^2(\P_{\bX_{\eta(s)}}, \boldsymbol{\delta}_0) \leq \sum_{i=1}^{\eta(s)} \E \big[|X_{t_i}|^2 \big]$, and since $t_{\eta(s)} \leq s$, we have
\begin{align*}
    \sum_{i=1}^{\eta(s)} \E[ |X_{t_i}|^2] \leq N \E \big[ \underset{0 \leq r \leq s}{\text{ sup }} |X_r|^2 \big],
\end{align*}
which implies 
\begin{align*}
     \E \Big[ \underset{0 \leq s \leq t}{\text{ sup }} |X_s|^2 \Big]  \leq C \bigg(1 +  \E \big[ |\xi|^2 \big] +  \int_{0}^{t}     \Big(\E\big[ \underset{0 \leq r \leq s }{\text{ sup }} |X_r|^2 \big] + \E \big[ |\alpha_s|^2\big] \Big) \d s \bigg).
\end{align*}
The estimate then follows from Grönwall's inequality applied to the map $t \mapsto \E \big[ \underset{0 \leq s \leq t}{\text{ sup }} |X_s|^2 \big]$.
\end{proof}

\bibliographystyle{siamplain}
\bibliography{main}

\end{document}